\documentclass[11pt,letterpaper]{amsart}

\usepackage[english]{babel}

\usepackage{amssymb,amsfonts,amsthm,amsmath}
\usepackage{enumitem, caption, array, esint, bm}

\usepackage{pgfplots}
\pgfplotsset{compat=1.15}
\usepackage{mathrsfs}
\usetikzlibrary{arrows}
\usetikzlibrary{shapes}
\usepackage{wrapfig}
\usepackage[normalem]{ulem}

\newcommand{\interval}[3]{\draw ({#1*cos(#2)}, {#1*sin(#2)}) arc ({270+#2}:{90+#3}:{#1*tan((#3-#2)/2)})}
\newcommand{\intervalpartial}[4]{\draw[densely dotted] ({#1*cos(#2)}, {#1*sin(#2)}) arc ({270+#2}:{270-180*#4+#4*#3+(1-#4)*#2}:{#1*tan((#3-#2)/2)})}
\usepackage[margin=1.1in,textwidth=6.2in]{geometry}
\usepackage{xr-hyper} % References to external document
\newcommand{\eqrefmain}[1]{(\ref*{main-#1})}
\newcommand{\refmain}[1]{\ref*{main-#1}}
\newcommand{\sumstar}{\sideset{}{^*}\sum}

\usepackage{hyperref}

\theoremstyle{plain}
\newtheorem{theorem}{Theorem}[section]
\newtheorem{lemma}[theorem]{Lemma}
\newtheorem{corollary}[theorem]{Corollary}

\newtheorem{proposition}[theorem]{Proposition}

\newtheorem{problem}[theorem]{Problem}

\theoremstyle{remark}
\newtheorem{remark}[theorem]{Remark}

\newcommand\numberthis{\stepcounter{equation}\tag{\theequation}}

\renewcommand{\setminus}{\smallsetminus}
\newcommand{\ssum}[1]{\sum_{\substack{#1}}}
\newcommand{\df}{\mathop{}\!\mathrm{d}}
\newcommand{\dfs}{\mathop{}\!\mathrm{ds}}
\newcommand{\eps}{{\varepsilon}}

\let\C\relax
\let\P\relax
\let\H\relax

\newcommand{\1}{{\mathbf 1}}

\newcommand{\bbl}[1]{\expandafter\newcommand\csname #1\endcsname{\mathbb #1}}
\newcommand{\call}[1]{\expandafter\newcommand\csname c#1\endcsname{\mathcal #1}}
\newcommand{\scrl}[1]{\expandafter\newcommand\csname #1#1\endcsname{\mathscr #1}}

\bbl{P}
\bbl{E}
\bbl{H}
\bbl{C}
\bbl{D}
\bbl{R}
\bbl{Q}
\bbl{Z}
\bbl{N}

\scrl{C}
\scrl{L}
\scrl{N}
\scrl{R}
\scrl{P}

\call{B}
\call{C}
\call{D}
\call{E}
\call{F}
\call{G}
\call{H}
\call{I}
\call{M}
\call{L}
\call{R}
\call{T}
\call{U}
\call{V}

\newcommand{\hE}{h^\cE}
\newcommand{\hL}{h^L}

\newcommand{\vchi}{{\bm \chi}}
\newcommand{\vmu}{{\bm \mu}}
\newcommand{\vSigma}{{\bm \Sigma}}

\newcommand{\cb}{{\mathfrak b}}

\newcommand{\fd}{{\mathfrak d}}
\newcommand{\fm}{{\mathfrak m}}

\newcommand{\dD}{{\partial\D}}
\newcommand{\dH}{{\partial\H}}

\DeclareMathOperator{\e}{e}
\DeclareMathOperator{\sg}{sg}
\DeclareMathOperator{\sign}{sgn}
\DeclareMathOperator{\dist}{dist}
\DeclareMathOperator{\im}{Im}
\DeclareMathOperator{\Dom}{Dom}

\DeclareMathOperator{\Id}{Id}

\DeclareMathOperator{\Vol}{Vol}

\DeclareMathOperator{\GL}{GL}

\DeclareMathOperator{\SL}{SL}
\DeclareMathOperator{\SU}{SU}

\DeclareMathOperator{\tr}{tr}
\DeclareMathOperator{\denom}{denom}

\DeclareMathOperator{\rank}{rank}
\DeclareMathOperator{\diam}{diam}

\newcommand{\vth}{{\vartheta}}
\renewcommand{\tilde}{\widetilde}
\renewcommand{\bar}{\overline}
\renewcommand{\hat}{\widehat}

\newcommand{\norm}[1]{{\left\| {#1} \right\|}}

\newcommand{\vt}{{\bm t}}
\newcommand{\vL}{{\bm L}}

\newcommand{\vphi}{{\bm \phi}}
\newcommand{\inpd}[1]{{\left\langle #1 \right\rangle}}
\newcommand{\binpd}[1]{{\big\langle #1 \big\rangle}}

\newcommand{\abs}[1]{{\left| {#1} \right|}}
\newcommand{\babs}[1]{{\big| {#1} \big|}}

\renewcommand\Re{\operatorname{\mathfrak{Re}}}
\renewcommand\Im{\operatorname{\mathfrak{Im}}}

\newcommand{\smat}[1]{\left(\begin{smallmatrix} #1 \end{smallmatrix}\right)}

\numberwithin{equation}{section}
\title{Central values of additive twists of automorphic $L$-series}

\date{\today}

\author{Sandro Bettin}
\address{Dipartimento di Matematica, Università di Genova, via Dodecaneso 35, 16146 Genova, Italy}
\email{bettin@dima.unige.it}

\author{Sary Drappeau}
\address{Laboratoire de Mathématiques Blaise Pascal, Université Clermont Auvergne, Institut Universitaire de France, 3 place Vasarely, 63178 Aubière, France}
\email{sary.drappeau@uca.fr}

\author{Jungwon Lee}
\address{Caltech, Division of PMA, 
1200 E California Blvd, Pasadena, CA 91125, USA}
\email{jungwon@caltech.edu}

\begin{document}

\begin{abstract}
  We study series~$L(x, s) := \sum_n \rho(n) \e^{2\pi i n x} / n^s$ with coefficients~$\rho(n)$ given by Fourier coefficients of an automorphic form~$\phi$ for a discrete cofinite subgroup~$\Gamma\subset \SL(2, \R)$, and~$x\in \R$. 

  Let~$C(\Gamma)\subset \R\cup\{\infty\}$ be the set of parabolic fixed points of~$\Gamma$. Ordering~$C(\Gamma)$ by increasing denominators, we prove that the values~$L(x, 1/2)$ for~$x\in C(\Gamma)$, suitably normalized, become equidistributed according to a standard complex Gaussian law, under natural assumptions on the Laplace eigenvalue of~$\phi$. Previously this was known either when~$\Gamma = \SL(2, \Z)$ or when~$\phi$ is associated to a holomorphic cusp form.
%  Let~$\cb\in\R\cup\{\infty\}$ be a representative of a cusp of~$\Gamma$.
%  Ordering the orbit~$\Gamma \cb$ by increasing denominators, we prove that the values~$L(x, 1/2)$ for~$x\in \Gamma \cb$, suitably normalized, become equidistributed according to a standard complex Gaussian law, under natural assumptions on the Laplace eigenvalue of~$\phi$. Previously this was known either when~$\Gamma = \SL(2, \Z)$ or when~$\phi$ is associated to a holomorphic cusp form.

  This applies in particular to central values at~$s=1/2$ of additive twists of Maa\ss{} form $L$-functions~$\sum_{n\geq 1} \lambda(n) \e^{2\pi i n x} / n^s$, or to the Dirichlet-theta series~$\sum_{n\geq 1} \chi(n) \e^{2\pi i n^2 x}/n^s$.
\end{abstract}

\maketitle

\tableofcontents

\section{Introduction}
The present paper is concerned with the statistical distribution, as~$x$ varies among certain subsets of real numbers, of series of the shape
\begin{equation}
  L(x) := \sum_{n\geq 1} a(n) \e^{2\pi i n x}, \label{eq:Lx-naive}
\end{equation}
where~$a(n)$ is a sequence resulting from a number-theoretic construction. We normalize the sequence so that~$a(n) \simeq n^{-1/2 + o(1)}$ on average, and we write
$$ a(n) = \frac{\rho(n)}{n^{1/2}}. $$
This question arose historically from investigations on equidistribution modulo~$1$ of various sequences, and can be traced back at least to works of Fortet~\cite{Fortet1940}, Kac~\cite{Kac1946} and Salem--Zygmund~\cite{SalemZygmund1947}; we refer to the survey~\cite{Kac1949} and the recent paper~\cite{AistleitnerBerkes2010} for references. We are interested here in cases where the sequence~$\rho(n)$ arises from an arithmetic construction related to~$L$-functions.

When~$\rho(n)$ is constant or periodic, the function~$x\mapsto L(x)$ is a linear combination of polylogarithms, and its values are distributed along the values of a real-analytic map on~$\R$ with countably many points removed. This is the case when~$\rho(n)$ are the coefficients of a degree-$1$ $L$-function (in the sense \emph{e.g.} of~\cite[Chapter~5]{IwaniecKowalski2004}).

A natural next step is to consider coefficients~$\rho(n)$ of a degree~$2$ $L$-function, for instance the Dirichlet convolution~$\rho = \chi_1 \ast \chi_2$ of two Dirichlet characters, or the coefficients of the $L$-function associated to an automorphic form~$\phi$ for a congruence subgroup of~$\SL(2, \Z)$ (\emph{i.e.}\ $a(n)$ are Fourier coefficients of $\phi$).
The definition of the series~\eqref{eq:Lx-naive} then requires more care; it must be understood as a value obtained by a process of analytic continuation, which involves an interplay between properties of~$x$ and of the form~$\phi$. In the two cases just mentioned, we may define it for all~$x\in \Q$ by means of analytic continuation to the central point $s=1/2$ of the series~$L(x, s) := \sum_{n\geq 1} \rho(n) \e^{2\pi i n x} / n^s$, an argument going back to Hecke~\cite{Hecke1936} in the case of classical modular forms; see \emph{e.g.}~\cite[Proposition~2.3]{DrappeauNordentoft2022}.

This construction generalizes to forms~$\phi$ which are automorphic with respect to an arbitrary discrete, cofinite group~$\Gamma\subset \SL(2, \R)$ with cusps. In this context, the set~$C(\Gamma)\subset \R\cup\{\infty\}$ of fixed points of parabolic elements of~$\Gamma$ plays the role of rational numbers. There is also a corresponding notion of denominators~$\denom(x)$, depending on a seed cusp, which we will recall in detail below; in particular there are asymptotically~$Q^2 /(\pi \Vol(\Gamma\backslash \H))$ elements in~$\Z\backslash\Gamma\cb$ with~$\denom(x)\leq Q$, in each~$\Gamma$-equivalence class of cusps.
The process of analytic continuation of~$L(x, s)$ is sensitive to the properties of~$x \in C(\Gamma)$ as a cusp, which prompts the question of how the values~$L(x)$ distribute in the complex plane as~$x$ varies.
\begin{problem}\label{pb:CLT}
  Fix a form~$\phi$ as above and let~$(\rho(n))$ be the Fourier coefficients of~$\phi$ at some cusp.
  As~$Q\to\infty$, prove a central limit theorem for the values
  $$ \{ L(x), x \in C(\Gamma)/\Z, \denom(x) \leq Q \}. $$
\end{problem}

This problem has so far been solved in the following cases:
\begin{itemize}
  \item When~$\phi$ varies among weight~$2$ holomorphic cusp forms, the map~$\phi \mapsto L_\phi(x)$ coincides with the modular symbol~$\{\infty, x\}$. This family includes those forms associated with an elliptic curve over~$\Q$, which prompted interest due to its link with the rank growth problem for elliptic curves, see~\cite{MazurRubin2023,DavidEtAl2004}. Petridis--Risager~\cite{PetridisRisager2018} resolved Problem~\ref{pb:CLT} in this case, using a construction of perturbations of Goldfeld--Eisenstein series~\cite{Goldfeld1999, Petridis2002} (see~\cite{PetridisRisager2004} and~\cite{Constantinescu2022} for related results). A completely different proof was found by the third-named author and Sun~\cite{LeeSun2019}, using dynamical systems. Finally Nordentoft~\cite{Nordentoft2021} showed how to solve the problem for holomorphic cusp forms of any even weight. The methods of~\cite{PetridisRisager2018, Nordentoft2021} work for even weight cusp forms over any discrete, cofinite group~$\Gamma\subset\SL(2,\R)$ with cusps.
  \item When~$\rho(n) = d(n)$, the number of divisors of~$n$, then the value~$L(x)$ coincides with the central value~$D(\frac12, x)$ of the Estermann function~\cite{Estermann1930}, which Estermann used, in relation with the circle method, to count solutions to the determinant equation. This case corresponds to~$\phi$ being a weight~$0$ Eisenstein series for~$\SL(2, \Z)$. The first two named authors proved in~\cite{BettinDrappeau2022} that the central limit theorem of Problem~\ref{pb:CLT} holds true in this case.
  \item When~$\rho(n)$ are the coefficients of an $L$-function attached to a Maa\ss{} cusp form for~$\SL(2, \Z)$, the second-named author and Nordentoft~\cite{DrappeauNordentoft2022} showed, again, that Problem~\ref{pb:CLT} holds true.
\end{itemize}

To summarize, the central limit theorem in Problem~\ref{pb:CLT} has been proven in two different cases:
\begin{itemize}
  \item The proofs in~\cite{PetridisRisager2018, Nordentoft2021} rely on a \emph{holomorphicity} and \emph{cuspidality} assumption on the form~$\phi$, with essentially no assumptions on the underlying group~$\Gamma$. The holomorphicity assumption plays a crucial role in both arguments, see the discussion in~\cite[Section~1.5]{DrappeauNordentoft2022}, while the assumption of cuspidality is unavoidable both for the construction in~\cite{PetridisRisager2018} and for the method of moments used in~\cite{Nordentoft2021}.
  \item The proofs in~\cite{BettinDrappeau2022, DrappeauNordentoft2022} hold for the \emph{modular group} $\Gamma = \SL(2, \Z)$, with essentially no assumptions on the form itself. The assumption~$\Gamma=\SL(2,\Z)$ originates from the use of a theorem in~\cite{BettinDrappeau2022} on rational orbits of the Gauss map, which derives from work of Baladi--Vallée~\cite{BaladiVallee2005}.
\end{itemize}

The above results then leave open the cases where the group is not~$\SL(2, \Z)$ and the form is either non-holomorphic or non-cuspidal holomorphic.
The goal of this paper is to complete the solution of Problem~\ref{pb:CLT} in full generality.

\begin{theorem}\label{th:main}
  The central limit theorem claimed in Problem~\ref{pb:CLT} holds true in general for automorphic forms of integer weight and mildly growing Fourier coefficients. More precisely, let~$\Gamma\subset\SL(2,\R)$ be discrete cofinite, with a singular cusp at~$\infty$, and let~$\phi\neq 0$ be automorphic for~$\Gamma$ of integer weight~$k$, with spectral parameter~$s_\phi$ and~eigenvalue~$\lambda_\phi = s_\phi(1-s_\phi)$. Assume that:
  \begin{enumerate}[label=(\alph*)]
    \item either~$\phi$ is cuspidal, or else~$\Re(s_\phi) = 1/2$; 
    \item if~$s_\phi \equiv k/2\pmod{1}$ then~$\phi(z)=y^{k/2}f(z)$ for some holomorphic form~$f$.
  \end{enumerate}
  Let~$a(n)$ be the Fourier coefficients of~$\phi$ at~$\infty$.
  Then for some constant~$\sigma_\phi>0$ and~$\beta_\phi \in \{0, 1, 3\}$, the multiset
  \begin{equation}
    \Big\{\frac{L(x)}{\sigma_\phi \sqrt{(\log Q) (\log \log Q)^{\beta_\phi}}}, x\in C(\Gamma)/\Z, \denom(x) \leq Q \Big\}\label{eq:main-multiset}
  \end{equation}
  converges as~$Q\to \infty$ to a standard complex Gaussian density. We have
  \begin{enumerate}[label=(\roman*)]
    \item $\beta_\phi = 0$ if~$\phi$ is cuspidal,
    \item $\beta_\phi = 1$ if~$\phi$ is not cuspidal and\footnote{Here $d_\H$ denotes the hyperbolic distance.}~$\phi(z) = O(\e^{d_\H(z, i)/2})$,
    \item $\beta_\phi = 3$ otherwise.
  \end{enumerate}
  Moreover, when~$\phi$ is cuspidal, the value of~$\sigma_\phi$ is given explicitly in terms of~$\phi$ by the expression~\eqref{eq:def-variance-cusp} below.
\end{theorem}

By standard complex Gaussian we mean~$N_\C(\Id) = X + iY$ where~$X, Y$ are independent standard real Gaussian random variables.

Our more general statement is Theorem~\ref{th:main-cusp} below in the cuspidal case, and Theorem~\ref{th:main-eis} otherwise.
\begin{remark}\label{rmk:beta-growth}
  The distinction between the various cases for the parameter~$\beta_\phi$ is better explained in terms of the shape of the growth of~$\phi$ at cusps. This growth is characterized by the behaviour at infinity of maps~$\phi_\cb$ of the shape~\eqref{eq:def-phiinfty}. The case~$\beta_\phi=0$ arises when~$\phi$ vanishes at all cusps. The case~$\beta_\phi=1$ arises when the maps~$\phi_\cb$ grow asymptotically as~$y^{1/2}$, and the case~$\beta_\phi=3$ arises when some~$\phi_\cb$ behaves asymptotically as~$y^{1/2}\log y$.

  For non-cuspidal~$\phi$ for which~$\lambda_\phi < 1/4$, we don't expect a Gaussian limit in general, however we do expect that our methods could be adapted to prove convergence to a stable law with stability parameter given by the spectral parameter; see also Remark~\ref{rmk:r2k} below.

  The condition (b) excludes, for simplicity, those forms for which~$s_\phi = k/2 + \ell$, $\ell\in\Z_{\neq 0}$, and in the case~$s_\phi = k/2$, restricts the behaviour at cusps ($B_\cb = 0$ in~\eqref{eq:def-phiinfty}).
  
  Condition (b) excludes those~$\phi$ obtained from holomorphic forms by weight-raising and lowering operators, as well as those related to certain non-holomorphic harmonic Maa\ss{} forms (see~\cite{KudlaEtAl1999}).
  The present arguments apply to these cases with minor adjustments; in fact, in many cases, there is a related form with the same coefficients~$a(n)$, $n\geq 1$, for which our theorem applies.
\end{remark}

\begin{remark}
  Regarding the existence of cuspidal Maa\ss{} forms for non-arithmetic groups, we refer to the surveys~\cite{Sarnak1986,Sarnak1990} and the computations in~\cite{FarmerLemurell2005,Avelin2007,Stroemberg2019}.
\end{remark}
\begin{remark}
  In Hecke congruence groups, there are about~$q^{1 + o(1)}$ parabolic fixed points~$x$ with~$\denom(x) = q$ modulo~$1$. This raises the question of whether, in Problem~\ref{pb:CLT}, the limit law holds on the sets of fixed denominators~$\denom(x) = q$ as~$q\to\infty$. The computations in~\cite{MazurRubin2023}, motivated by the case of weight~$2$ holomorphic forms, indicate that this should be the case, however this is a much more difficult question, for which neither the dynamical methods used here nor the Goldfeld--Eisenstein series methods of~\cite{PetridisRisager2018, Nordentoft2021} are suitable. The best result in this direction, due to Blomer--Fouvry--Kowalski--Michel--Milićević--Sawin~\cite[Theorem~9.2]{BlomerEtAl2023}, is the estimation of the second moment~$\sum_{\denom(x) = q} L(x)^2$ for prime~$q$. See also~\cite{MilicevicQinWu,Pascadi2026} for a related result for general $q$, and~\cite{DiamantisEtAl2020} for a detailed analysis on the first moment.

  For generic groups, a natural condition could be~$\denom(x) \in [q, q+1)$, however it is unlikely that anything relevant can be asked for fixed~$\denom(x)=q$. Indeed, based on numerical computations with the deformations of~$\Gamma_0(5)$ considered in~\cite{FarmerLemurell2005}, we are led to believe that for generic~$\Gamma$, the multiplicities~$\{x\in C(\Gamma), \denom(x) = q\}$ remain bounded with respect to~$q$, and thus statistical studies for individual values of~$q$ are not meaningful. As a partial progress in that direction, the proof of Theorem~\ref{th:main} could be strengthened to give a limit law when restricting the denominator to lie in~$[Q, Q+Q^{1-\delta}]$ for some sufficiently small~$\delta>0$.
\end{remark}

\begin{remark}
  For a cuspidal, weight~$k=0$ form~$\phi$, Constantinescu-Nordentoft~\cite{ConstantinescuNordentoft2025} prove that asymptotically almost all closed geodesics periods~$\int_\gamma f$ are non-vanishing. Using Theorem~\ref{th:main}, it can be shown that their results holds for any group~$\Gamma$, see Remark~1.14 there.
\end{remark}

\begin{remark}
  Using a convexity bound~\cite[Chapter~5.2]{IwaniecKowalski2004}, it is easy to prove that~$L(x)=L^*(x)+O(1)$, where $L^*(x):=\sum_{n\geq 1} a(n) \e^{2\pi i n x}f(n/\denom(x)^2)$ and $f$ is a smooth cut-off function with $f(0)=1$. Thus, our work could equivalently be viewed as providing a central limit theorem for $L^*(x)$.

  This raises the following natural question: instead of looking at~$L(x,s)$ at the central point $s=\frac12$, we could truncate the sum~\eqref{eq:Lx-naive} at some parameter $Q$ going to infinity, with $x$ being chosen at random in~$[0, 1]$. This question will be the subject of forthcoming work~\cite{BettinEtAl3}.
\end{remark}

As a particular case of our results, consider~$r_2(n)$, the number of ways to write~$n$ as sums of two squares, and for~$x\in \Q$, form the series
$$ R(s, x) := \sum_{n\geq 1} \frac{r_2(n)\e^{2\pi i nx}}{n^s}. $$
This is initially defined on~$\Re(s)>1$ and has an analytic continuation to~$\C\setminus\{1\}$.
\begin{theorem}\label{th:main-r2}
  For some~$\sigma>0$, the multiset
  $$ \Big\{ \frac{R(1/2, x)}{\sigma \sqrt{\log Q \log\log Q}}, x \in \Q/\Z, \denom(x)\leq Q\Big\} $$
  becomes equidistributed, as~$Q\to\infty$, according to a standard complex Gaussian distribution.\footnote{Here~$\denom$ is the usual denominator.}
\end{theorem}

This follows from our more general Theorem~\ref{th:main-eis} below, upon writing~$r_2(n) = 4 (\1\ast \chi)(n)$ with~$\chi$ the non-trivial character~$\mod 4$.

\begin{remark}\label{rmk:r2k}
  One may wonder more generally about the values at~$s=1/2$ of the analytic continuation of
  $$ R_k(s, x) := \sum_{n\geq 1} \frac{r_{2k}(n) \e^{2\pi i nx}}{n^{s+(k-1)/2}}. $$
  The underlying automorphic form has eigenvalue~$\lambda_\phi = \frac k2(1 - \frac k2) < 1/4$ if~$k>1$. On the basis of our arguments, we don't expect a Gaussian law to hold for the values~$\{R_k(1/2, x)\}$, however our arguments suggest the convergence towards a stable law of stability parameter~$2/k$, \emph{e.g.}\ a Cauchy law for~$k=2$.
\end{remark}

Problem~\ref{pb:CLT} is relevant for non-integer weights as well, and our approach generalizes with little modification. Here we focus in particular on the case of theta series. Given an integer~$N\geq 1$, a primitive character $\chi$ modulo $N$ and~$x\in \Q$, we define
\begin{equation}\label{lcdf}
  L_\chi(s, x) := \sum_{n\geq 1}\chi(n) \frac{\e^{2\pi i n^2x}}{n^s}
\end{equation}
for $\Re(s)>1$ and by analytic continuation for~$s\in\C\setminus\{1\}$. In this case, we obtain a central limit theorem for $L_\chi(1/2,x)$ with variance of order~$\log Q$, independently of whether the underlying modular form is cuspidal or not.

\begin{theorem}[CLT for theta series]\label{thethe}
  Let $\chi$ be a primitive character mod $N$ and let $\sigma_\chi:=\sqrt{\frac{\varphi (N)}{2N}}$ where $\varphi$ is Euler's totient function. Then the multiset
  $$ \Big\{ \frac{L_{\chi}(1/2, x)}{\sigma_\chi \sqrt{\log Q}}, x \in \Q/\Z, \denom(x)\leq Q\Big\} $$
  becomes equidistributed, as~$Q\to\infty$, according to a standard complex Gaussian distribution.
\end{theorem}

Theorem~\ref{thethe} remains true if the character~$\chi$ is replaced by any non-zero, $N$-periodic function $f : \Z \to \C$, see Remarks~\ref{tfeo} and~\ref{rmk:theta-general-f} below. In this case, the variance is $\sigma_f^2 := \frac{1}{2N} \sum_{n \pmod N} |f(n)|^2$. We refer to~\cite{JurkatHorne1981,JurkatHorne1983,Marklof1999,CellarosiMarklof2016} for other types of distributional results for theta sums.

The Gaussian behaviour in Theorem~\ref{th:main} is reminiscent of results and conjectures on the behaviour of central~$L$-values~$\{L(\pi, \frac12)\}$ as $\pi$ varies in certain families of automorphic representations~\cite{Selberg1992,Montgomery1973,KeatingSnaith2000,ConreyEtAl2005,RadziwillSoundararajan2015}. The driving heuristic in these questions is the expected asymptotic independence of values~$\{\rho(p), p\text{ prime}\}$, which should translate into a CLT for the \emph{logarithm}~$\log L(\pi, \frac12)$ due to the existence of the Euler product expansion.
It is no surprise that the values~\eqref{eq:Lx-naive}, which do not possess an Euler product, behave in a different way. Our situation is instead more similar to that of the vertical distribution of the Hurwitz or periodic zeta-function with irrational parameters. For these no Euler product is available and it is conjectured that they also satisfy a central limit theorem~\cite{HeapSahay2025}.

\begin{remark}
  The relationship between the additive twist and~$\GL(2)$ automorphy is ultimately due to the shape of unipotent subgroups of~$\SL(2)$. It is a natural question to ask what remains of the limit law for sequences related to~$\GL(3)$ automorphic forms, starting for example with the third-order divisor function~$d_3(n) = \1\ast\1\ast\1(n)$. The additive twists
  $$ \sum_{n\geq 1} \frac{d_3(n)}{\sqrt{n}}\e^{2\pi i n x} $$
  can be defined in the same way as for the usual divisor function, however contrary to our case here, they do not seem to enjoy the quantum modularity property on which our methods rely, which leaves open the question of the existence of a limiting law for the distribution of their values.
\end{remark}

\subsection*{Acknowledgements}

We thank the organizers of the April 2023 MPI Bonn workshop ``Arithmetic statistics, automorphic forms and ergodic methods'' (N.\ Diamantis, M.\ Lee, P.\ Moree and Y.\ Petridis), and also the DIMA (Univ.\ of Genoa), the I2M and CIRM (Univ.\ Aix-Marseille) where parts of this work have been carried out since 2022.

We thank Yiannis Petridis, Morten Risager and Asbjørn Nordentoft for many discussions related to the topic of this work.

Large language models (LLM) were used to check for errors in an earlier version of this manuscript, and suggested a simplification of the proof of Lemma~\ref{lem:hE-at-half}, the manipulation underlying~\eqref{eq:I-infty-symmetry} \and the bounds~\eqref{eq:bounds-diff-sigma-chis}.

This work benefited from financial support from FWF-I4945-N/ANR-20-CE91-0006, SNF-10.003.145/ANR-24-CE93-0016, H2020-EU.1.1 (grant agreement 787304), ERC-Advanced Grant
833802 (Resonances), and from the INdAM group GNAMPA.

During the preparation of this work, S.B. was partially supported by PRIN 2022 ``The arithmetic of motives and $L$-functions'', by the Curiosity Driven grant “Value distribution of quantum modular forms” of the University of Genoa, funded by the European Union -- NextGenerationEU, and by the MIUR Excellence Department Project awarded to Dipartimento di Matematica, Università di Genova, CUP D33C23001110001.

\section{Setting and results}

\subsection{Automorphic forms}\label{sec:automorphic-forms}

Let~$\Gamma\subset \SL(2, \R)$ be a discrete, cofinite group with at least one cusp which, up to conjugation, we can assume to be at infinity and of width~$1$. We assume without loss of generality that $-\Id\in\Gamma$, extending the group if necessary. Denote by~$\H$ the Poincaré upper half-plane, and by~$C(\Gamma) \subset \P^1(\R)$ the set of parabolic fixed points of~$\Gamma$, which for brevity we will refer to as \emph{cusps} of~$\Gamma$ (although we do not identify~$\Gamma$-equivalent points).
In the following we let~$k\in \Z$, except in Section~\ref{sec:theta} (where~$k$ will be half-integral).
Let~$\chi :\Gamma\to\C^\times$ be a finite-order character, and assume the compatibility condition~$\chi(-\Id) = (-1)^k$.
We form the multiplier
\begin{align}\label{multp}
  u_\gamma(z) = j_\gamma(z)^k \chi(\gamma), \qquad \text{where} \quad \begin{cases} j(\gamma, z) := cz + d \\ j_\gamma(z) := \frac{j(\gamma, z)}{\abs{j(\gamma, z)}}, \end{cases} \qquad \big(\gamma = \smat{*&*\\c&d},\ z\neq \gamma^{-1}\infty \big).
\end{align}

To each cusp~$\cb \in C(\Gamma)$, we associate its stabilizer~$\Gamma_\cb$ in~$\Gamma$, and a \emph{scaling matrix} $\sigma_\cb \in \SL(2, \R)$ satisfying
\begin{equation}
  \sigma_\cb\infty = \cb, \qquad \Gamma_\cb =\pm \sigma_\cb \smat{1&\Z\\ & 1} \sigma_\cb^{-1}.\label{eq:def-scaling-matrix}
\end{equation}

Let~$\phi$ be an automorphic form for~$\Gamma$ of weight~$k$ with character~$\chi$, which means that
$$ \phi : \H \to \C $$
is a smooth map satisfying the following conditions:
\begin{enumerate}[label=$(H_{\arabic*})$]
  \item\label{hyp:modularity} For all~$\gamma \in \Gamma$ and~$z\in \H$, we have~$\phi(\gamma z) = u_\gamma(z) \phi(z)$.
  \item\label{hyp:decay} For all~$\cb\in C(\Gamma)$, we have~$\phi(\sigma_\cb z) \ll y^{O(1)}$ for~$y = \Im(z) \geq 1$.
  \item\label{hyp:laplacian} The map~$\phi$ is an eigenfunction of the hyperbolic Laplacian
  $$ \Delta_k = y^2\Big( \frac{\partial^2}{\partial x^2} + \frac{\partial^2}{\partial y^2}\Big) - iky \frac{\partial}{\partial x} $$
  with eigenvalue~$-\lambda_\phi$, where~$\lambda_\phi = s_\phi(1-s_\phi) = 1/4 + t_\phi^2$ with~$\Re(s_\phi) \geq 1/2$ (here~$s_\phi = 1/2 + it_\phi$).
\end{enumerate}
These hypotheses imply the existence of a Fourier expansion at every cusp. Those cusps~$\cb\in C(\Gamma)$ for which~$\chi(\Gamma_\cb\cap\{\tr(\gamma)=2\}) = \{1\}$ are called singular or essential (see~\cite[eq.~(2.65)]{Iwaniec1997} or \cite[eq.~(8)]{Proskurin2003}). For a singular cusp~$\cb$, the Fourier expansion takes the shape
\begin{equation}\label{eq:phi-Fourier}
  \begin{aligned}
    & j_{\sigma_\cb}(z)^{-k} \phi(\sigma_\cb z) \\
    & \quad = \phi_\cb(\Im(z)) + \sum_{n\neq 0} a_\cb(n) \e(n\Re(z)) W_{\frac k2 \sign(n), it_\phi}(4\pi \abs{n} \Im(z)), \qquad (\cb \text{ singular})
  \end{aligned}
\end{equation}
where~$W_{\ast}$ is the Whittaker function, 
$$ \e(z) := \e^{2\pi i z}, $$
and
\begin{equation}
  \phi_\cb(y) = \begin{cases} A_\cb y^{s_\phi} + B_\cb y^{1-s_\phi}, & (s_\phi \neq 1/2), \\ A_\cb y^{1/2} + B_\cb y^{1/2} \log y, & (s_\phi = 1/2). \end{cases}\label{eq:def-phiinfty}
\end{equation}
We remark that $\phi_\cb(y)$ also depends mildly on the choice of the scaling matrix $\sigma_\cb$, as changing $\sigma_{\cb}\mapsto-\sigma_{\cb}$ sends
$ \phi_\cb(y)\mapsto (-1)^k \phi_\cb(y)$.

For non-singular cusps, a similar expansion holds with all terms having a uniform exponential decay as~$\Im(z) \to \infty$, see~\cite[eq.~(13)]{Proskurin2003}. Thus, for some~$c>0$ and all non-singular cusps~$\cb$, we have
$$ \phi(\sigma_\cb z) \ll_\cb \e^{-c \Im(z)}, \qquad (\Im(z) \gg 1, \cb \text{ non-singular}). $$
We say that~$\phi$ is \emph{cuspidal} if~$A_\cb = B_\cb = 0$ for all singular~$\cb\in C(\Gamma)$.

We assume further the following bound on the Fourier coefficients:
\begin{enumerate}[label=$(H_{\arabic*})$]
  \setcounter{enumi}{3}
  \item\label{hyp:bound-fourier} The Fourier coefficients~$a_\cb(n)$ in~\eqref{eq:phi-Fourier} satisfy
    \begin{equation}
      \sum_{1\leq \abs{n}\leq X} \abs{n} \abs{a_\cb(n)}^2 \ll_\phi \begin{cases} X & (\phi\text{ cuspidal}), \\ X^{1+\eps} & (\text{otherwise}), \end{cases} \qquad (X\geq 1).\label{eq:rankinselberg-upper}
    \end{equation}
    In particular, we have
    \begin{equation}
      a_\cb(n) \ll_{\phi,\eps} \abs{n}^\eps \qquad \text{and} \qquad \sum_{n\neq 0} \frac{\abs{a_\cb(n)}}{\abs{n}^{1/2+\eps}} \ll_{\phi, \eps} 1. \label{eq:bound-meanabs-an}
    \end{equation}
\end{enumerate}
The hypothesis~\ref{hyp:bound-fourier} is automatically satisfied when~$\phi$ is cuspidal by the argument in~\cite[Theorem~3.2]{Iwaniec2002}. This argument can also be carried out when~$\phi$ is not cuspidal and~$\Re(s_\phi) = \frac12$ (\emph{i.e.}~$\lambda_\phi \geq 1/4$), using the~$L^2$ bound~$\int_{\Gamma\backslash \H} \abs{E^Y(z)}^2 \df\mu(z) \ll Y^\eps$ on truncated Eisenstein series in~\cite[eqs.~(6.35)--(6.36)]{Iwaniec2002}.

In the following we assume that either~$s_\phi \not\equiv \frac k2\pmod{1}$, or~$y^{-k/2} \phi(z)$ is holomorphic. Note also that if~$\phi$ is cuspidal with~$s_\phi \not \equiv k/2 \pmod{1}$, then necessarily~$s_\phi \in [1/2, 1)\cup(\frac12+i\R)$~\cite[p.~508]{DukeEtAl2002}.

\medskip{}

In the rest of the paper, the form~$\phi$ is fixed, and apart from Section~\ref{sec:theta}, we assume everywhere that either~$\phi$ is cuspidal, or~$\Re(s_\phi) = \frac12$.
Our main focus is on the Fourier coefficients at~$\infty$: we assume that~$\chi(\smat{1&1\\0&1}) = 1$, so that~$\infty$ is a singular cusp, and we denote
$$ a(n) := a_\infty(n). $$

\subsection{Additive twists}

Given~$\phi$ as above, we define for all~$x\in C(\Gamma)\setminus\{\infty\}$ the following~$L$-series\footnote{In this work we focus on the sum over positive~$n$ only; the arguments in~\cite{DrappeauNordentoft2022} could be adapted with minor adjustments if one wishes to include the sum over~$n\leq -1$ in the analysis.}:
\begin{equation}
  L(\phi, s, x) := \sum_{n\geq 1} a(n) \e(n x) n^{1/2-s}. \label{eq:def-Lphixs}
\end{equation}
By~\cite[Section~5]{DrappeauNordentoft2022}, this function of~$s$, initially defined for~$\Re(s)>1$, extends to a meromorphic function on~$\C$ which is regular at~$s=1/2$.

We define
$$ L_\phi(x) := L(\phi, 1/2, x), \qquad (x\in C(\Gamma)). $$

\subsection{Orbits and denominators}

In the case~$\Gamma = \SL(2, \Z)$, or more generally for congruence groups, the set of cusps is simply~$C(\Gamma) = \Q\cup\{\infty\}$, and we have a natural ordering according to denominators. We recall from~\cite{PetridisRisager2018} the appropriate generalization for cofinite Fuchsian groups with cusps, the role of the ``denominator'' being played by the bottom-left coefficient of an associated matrix.

Let~$\cb\in C(\Gamma)$ be some fixed cusp, and~$\sigma_\cb$ be a scaling matrix.
We order the set $\Gamma \cb$ of cusps equivalent to~$\cb$ according to the following definition: for all~$Q>0$, let
\begin{equation} \label{cqdfn}
  C_{\cb}(\Gamma; Q) := \{ \gamma \cb, \gamma \in \Gamma / \Gamma_\cb, \gamma \sigma_\cb = \smat{\ast&\ast\\ c & \ast}, 0 < |c| \leq Q\}.
\end{equation}

In the notation of~\cite[eq.~(1.5)]{PetridisRisager2018}, we have~$C_\cb(\Gamma; Q) = T_{\infty\cb}(Q) + \Z$. Thus, for example, $C_\infty(\SL(2, \Z); Q)$ is the set of all rationals of denominator at most~$Q$. The sets~$C_\cb(\Gamma; Q)$ are all invariant under translation by~$\Z$.

\subsection{Probability space}

We fix from now on a cusp~$\cb\in C(\Gamma)$, and consider the space
$$ \Omega_Q := C_\cb(\Gamma; Q), $$
for~$Q\geq 1$.

The following lemma, which corresponds essentially to~\cite[Theorem~3.8]{PetridisRisager2018}, is borrowed from~\cite[Lemma~\refmain{lem:psi-totalmass}]{BettinEtAl}.
\begin{lemma}
  There exists~$\delta>0$ such that for all smooth and compactly supported~$\psi:\R\to \C$, we have
  $$ \sum_{x\in \Omega_Q} \psi(x) = \frac{Q^2}{\pi \Vol(\Gamma\backslash\H)} \int_{\R} \psi(r) \df r + O_{\psi,\cb}(Q^{2-\delta}). $$
\end{lemma}

We fix once and for all a smooth, compactly supported, non-negative and non-identically zero weight function~$\psi : \R \to \R_{\geq0}$. By the previous lemma, for~$Q$ large enough in terms of~$\cb$ and $\psi$ we may endow~$\Omega_Q$ with the probability measure~$\P_Q$ proportional to~$\psi$,
$$ \P_Q = \varpi_{\psi, Q} \sum_{x\in \Omega_Q} \psi(x) \delta_x, $$
where~$\varpi_{\psi, Q} = (\sum_{x\in \Omega_Q} \psi(x))^{-1} > 0$ is a normalizing factor and~$\delta_x$ is the Dirac mass at~$x$. We call~$\E_Q$ the associated expectation. Thus, for any map~$f:C(\Gamma)\to \C$, we have
$$ \E_Q[f(x)] = \frac{\sum_{x\in \Omega_Q} \psi(x) f(x)}{\sum_{x\in \Omega_Q} \psi(x)}. $$
In Section~\ref{sec:proof-limit-thm} we will occasionally write~$\E_{Q; \Gamma, \cb}$ to specify the underlying group and cusp orbit.

\subsection{Distribution of values}

Our main results are the following.
\begin{theorem}[CLT for cuspidal forms]\label{th:main-cusp}
  Assume that~$\phi$ is cuspidal. Let~$\RR\subset \C$ be a Borel set with measure zero boundary. As~$Q\to \infty$, we have
  $$ \P_Q\Big( \frac{L_\phi(x)}{\sigma_\phi \sqrt{\log Q}} \in \RR\Big) \to \P(N_\C(\Id) \in \RR), $$
  where~$N_\C(\Id)$ is a standard complex Gaussian, and
  \begin{equation}
    \sigma_\phi^2 = \frac1{\Gamma(1-s_\phi + \frac k2) \Gamma(s_\phi + \frac k2)} \frac{\|\phi\|^2}{\Vol(\Gamma\backslash \H)}, \label{eq:def-variance-cusp}
  \end{equation}
  with~$\|\phi\|$ the Petersson $L^2$-norm.
\end{theorem}

\begin{remark}
  The proof gives a uniform and optimal rate of convergence~$O(1/\sqrt{\log Q})$ when~$\RR$ varies over convex Borel sets, as in the usual multivariate CLT when the third moment is bounded, see~\cite[Theorem~13.2]{BhattacharyaRangaRao2010}.
\end{remark}

\begin{remark}
  Theorem~\ref{th:main-cusp} is derived from an estimate on the characteristic function $\E_Q[\e^{i\inpd{\vt, L_\phi(x)}
  }]$ for~$\vt\in\R^2$. As in \cite[Theorem~2.4]{BettinDrappeau2022}, the present method also gives large deviations estimates, using the cuspidality of~$\phi$. This fact was known for forms which are holomorphic of weight~$2$~\cite{PetridisRisager2018,LeeSun2019} or of level~$1$~\cite{BettinDrappeau2022,DrappeauNordentoft2022}. The work of Nordentoft~\cite{Nordentoft2021}, which considers holomorphic forms in general, relies on the method of moments, which at present cannot give a large deviation estimate.
\end{remark}

\begin{theorem}[CLT for non-cuspidal forms]\label{th:main-eis}
  Let~$\phi$ be non-cuspidal with~$\Re(s_\phi) = \frac12$, and denote
  \begin{align}\label{beta}
    \beta_\phi = \begin{cases} 3, & (s_\phi = \frac12 \text{ and } \exists \cb, B_\cb \neq 0), \\ 1 & (\text{otherwise}), \end{cases}
  \end{align}
  where we recall the notation~\eqref{eq:def-phiinfty}. Then, there exists $\sigma_\phi>0$ such that for any~$\RR\subset \C$ as in Theorem~\ref{th:main-cusp}, as~$Q\to \infty$, we have
  \begin{align}\label{t2mf}\P_Q\Big( \frac{L_\phi(x)}{\sigma_\phi \sqrt{\log Q(\log\log Q)^{\beta_\phi}}} \in \RR\Big) \to \P(N_\C(\Id) \in \RR).
  \end{align}
\end{theorem}

The value of~$\sigma_\phi$ could be made explicit in terms of a fundamental region for~$\Gamma$ and certain values arising in the course of our proof, see Section~\ref{sec:proof-limitthm-eis}.

When~$\Gamma = \Gamma_0(N)$ is a Hecke congruence subgroup, given~$\chi_1$, $\chi_2$ primitive Dirichlet characters mod~$N_1$, resp.~$N_2$ with~$N = N_1N_2$, we may pick~$\phi = E_{\chi_1, \chi_2}$ as an Eisenstein series for which~$a(n) = \frac1{\sqrt{n}}\sum_{d\mid n} \chi_1(d) \bar{\chi_2}(n/d)$ and~$s_\phi = 1/2$. Then~$\beta_\phi = 3$ if~$\chi_1 = \bar{\chi_2}$ and~$\beta_\phi = 1$ otherwise.

As we mentioned, the proofs of Theorems~\ref{th:main-cusp} and~\ref{th:main-eis} are based on the asymptotic analysis of the characteristic functions, rather than the method of moments. Nonetheless, the method also yields the asymptotic estimate for all convergent moments of $L_\phi(x)$, which is to say, all moments in the cuspidal case and the first moment in the non-cuspidal case. In Section~\ref{sec:comp-moments} we analyze in particular the expectation and the variance (in the cuspidal case) in order to determine the parameters of the central limit theorems in Theorems~\ref{th:main-cusp} and~\ref{th:main-eis}.

\begin{remark}
  Given a finite family $\phi_1,\dots,\phi_n$ of orthogonal cuspidal forms for a group $\Gamma$ and a character $\chi$, with same weight and eigenvalue, the arguments described here can be used to obtain, more generally, the joint convergence of the vector
  \begin{align}\label{casdc}
    \Big(L_{\phi_j}(x)/\sigma_{\phi_j}\sqrt{\log Q}\Big)_{1\leq j \leq n}
  \end{align}
  to the standard $n$-dimensional complex Gaussian distribution, generalizing~\cite[Corollary 7.7]{DrappeauNordentoft2022} which dealt with the case $\Gamma=\SL(2,\Z)$. To prove this, we can follow~\cite{DrappeauNordentoft2022} or, more directly, apply Theorem~\ref{th:main-cusp} to $\sum_{j=1}^na_jL_{\phi_j}(x)=L_{\sum_{j=1}^na_j\phi_j}(x)$ for any $a_1,\dots,a_n\in\C$ and invoke the Cramér--Wold device~\cite{CramerWold1936} (see~\cite[Theorem 29.4]{Billingsley1995}).
  Normalizing suitably the components of the vector~\eqref{casdc} one can also prove a similar statement for non-cuspidal forms.
\end{remark}

\begin{remark}
  The non-central values~$L(\phi, x, s)$, $\Re(s)\neq 1/2$, behave very differently, see~\cite{Nordentoft2021b,BettinDrappeau2025}.
  The case of values~$L(\phi, x, \frac12+it)$, $t\neq 0$ is less clear; on the basis of the present work, we expect that they tend to follow Gaussian laws similar to Theorems~\ref{th:main-cusp} and \ref{th:main-eis}, but introducing an imaginary shift adds new difficulties when dealing with the effect from the multiplier system\footnote{More precisely, with the notations of Section~\ref{vtdf}, the set~$\Theta$ becomes dense in the unit circle.}.
\end{remark}

\subsection{Overview}

In the present paper, we take up the method introduced in~\cite{LeeSun2019,BettinDrappeau2022}.
One key ingredient is the authors' work~\cite{BettinEtAl}, where Baladi--Vallée's methods are transferred to the context of cofinite Fuchsian groups with cusps, and the role of the Gauss map is played by the Bowen--Series map~\cite{BowenSeries1979}.

The ``observables'' in our work are maps of the shape
\begin{equation}
  h_\gamma^L(x) := L_\phi(\gamma x) - u_\gamma(x) L_\phi(x), \label{eq:def-periodfun-informal}
\end{equation}
for which we have no sensible explicit expression. This poses a major issue regarding our two tasks: (i) to check that the hypotheses of~\cite{BettinEtAl} are satisfied in our case, and (ii) to compute the parameters of the limit laws.

One of the main features in~\cite{BettinEtAl} is a choice of induction of the Bowen--Series map whose inverse branches are easily described, as well as flexible hypotheses on the observables (existence of some positive absolute moment).
Section~\ref{sec:estim-period-fun} is devoted to proving that the necessary regularity and growth assumptions are satisfied by the maps~\eqref{eq:def-periodfun-informal} for~$\gamma = \gamma_0^n$, where~$\gamma_0$ is a fixed parabolic element, and~$n$ and~$x$ vary. This extends work of~\cite{DrappeauNordentoft2022}, which corresponds to the case~$n=1$.
The main challenge is to achieve complete uniformity, both in~$n$ and~$x$, and in fact, we will also need to analyze a more general family~$h_\gamma^L(x,s)$ of period functions associated with~$L(\phi, s, x)$, in particular their asymptotic behaviour uniformly in the complex parameter~$s$.

Section~\ref{sec:fourier} is devoted to the study of certain Fourier integrals that arise in the analysis of the main terms, while Sections~\ref{sec:disk-model} and~\ref{tos} introduce the setting required to apply the results of~\cite{BettinEtAl}.
Then, in Section~\ref{sec:proof-limit-thm}, we apply the results of~\cite{BettinEtAl}, proving Theorems~\ref{th:main-cusp} and~\ref{th:main-eis} up to the values of the parameters of the limit law.
In the course of the argument, we must show that the multiplier system~$u_\gamma$ satisfies a certain topological irreducibility condition. This will be carried out by an ad-hoc combinatorial construction with Bowen--Series codings, which we defer in Section~\ref{sec:top-irred}.
At that point, we will not have an explicit expression for the expectation and the covariance matrix of the limiting distributions yet;
computing the expectation and covariance matrix requires an unexpected amount of additional work.

\smallskip{}

First, in Section~\ref{sec:proof-limitthm-eis}, we show that the Gaussian law in Theorem~\ref{th:main-eis} is circularly symmetric.
This requires an analysis of various contributions to the variance, and two symmetries: the presence of the involution~$\gamma \mapsto \gamma^{-1}$ over parabolic elements of~$\Gamma$, and a left-right symmetry for the behaviour at~$\infty$ of the period function~$h_\gamma^L$.

In Section~\ref{sec:comp-moments}, we compute the expectations in both the cuspidal and non-cuspidal cases, and the covariance matrix in the cuspidal case. To carry this out, we introduce a perturbation argument which expresses both quantities in terms of limits, as~$s\to 1/2$ with~$s>1/2$, of continuous averages of~$L(\phi, s, x)$ over~$x$.
These averages are estimated using standard analysis (for the expectation) and Rankin--Selberg theory (for the variance). A more detailed overview is given in Section~\ref{sec:comp-moments-overview}.

Finally, in Section~\ref{sec:theta} we consider the case of half-integral weights and prove Theorem~\ref{thethe}.

\subsection*{Notations}

The Landau symbol~$f = O(g)$ means that there exists a constant~$C\geq 0$ such that~$\abs{f} \leq C g$ whenever~$f$ and~$g$ are defined. The notation~$f \ll g$ is equivalent to~$f = O(g)$. If the constant depends on a parameter, say~$\alpha$, this is typically indicated by a subscript, \emph{e.g.}~$f = O_\alpha(g)$ or~$f \ll_\alpha g$. However, we will not explicitly indicate dependencies on $\phi$, $\eps$, $\cb$, $\psi$ and on the group $\Gamma$, nor on the subgroup $G$ and the fundamental domain $R$ defined in Sections~\ref{sec:disk-model} and~\ref{tos}.

Given a matrix~$\gamma\in \SL(2,\C)$, we indicate by~$c_\gamma$ the lower-left entry of~$\gamma$.
Throughout the paper, the letter~$\eps$ indicates a positive number which can be taken arbitrarily small, and whose value might vary from line to line.
The indicator function of a condition $C$ is denoted by $\1_C$, with $\1$ denoting the constant function $1$.

\section{Estimates on the period function}\label{sec:estim-period-fun}

\subsection{Analytic continuation of the Eichler integral}

As in~\cite{DrappeauNordentoft2022}, we start with the Eichler integral $\cE_x(\phi, s)$.
For~$s\in \C$ with~$\Re(s)>1$ and all~$x\in \R$, we define
\begin{equation}
  \cE_x(\phi, s) := \int_0^\infty \big(\phi(x+iy) - \phi_\infty(y)\big) y^{s-1/2} \frac{\df y}y.\label{eq:def-Eichlerint}
\end{equation}
We will see later on, in~\eqref{EtoL}, how to transfer information about $\cE_x$ to $L$.
The maps~$\cE_x(\phi, s)$ are related to~$L(\phi, s, x)$ in the same way as automorphic periods relate to $L$ series, see~\cite[Sections~5.2--5.3]{DrappeauNordentoft2022}.

Recall our assumption that~$\Re(s_\phi) = 1/2$ whenever~$\phi$ is non-cuspidal.
For convenience, we denote
\begin{equation}
  \rho_\phi := \begin{cases} 0, & (\phi \text{ cuspidal}), \\ 1/2 & (\text{otherwise}), \end{cases}\label{eq:def-rho_phi}
\end{equation}
and we recall the notation~$\beta_\phi$ from~\eqref{beta}.
\begin{lemma}
  We have
  \begin{equation}\label{eq:decay-tail-phi}
    \abs{\phi(x + iy) - \phi_\infty(y)} \ll \begin{cases} y^{-\rho_\phi} \e^{-y/3}, & (\beta_\phi \neq 3) \\ y^{-\rho_\phi} (1 + \abs{\log y}) \e^{-y/3}, & (\beta_\phi = 3) \end{cases} \qquad (x\in \R, y>0).
  \end{equation}
\end{lemma}
\begin{proof}
  For~$y\gg 1$, this follows from the Fourier expansion~\eqref{eq:phi-Fourier}, the bound~\eqref{eq:bound-meanabs-an} and the exponential decay of Whittaker function~\cite[9.227]{GradshteynRyzhik2007}.

  Assume then that~$y$ is small. If~$\phi$ is cuspidal, the right-hand side is~$\asymp 1$ and thus the bound is obvious by boundedness of~$\phi$. Assume that~$\phi$ is not cuspidal, let~$z = x + iy$ and let~$F$ be the Ford fundamental domain~\cite[Chapter~2.2]{Iwaniec2002}. For all~$Y\geq 1$, denote also~$F(Y)$ the central part of~$F$, in the sense of~\cite[Chapter~2.2]{Iwaniec2002}, and let~$Y_0 \geq 1$ be such that the cuspidal parts~$F_\cb(Y_0)$ are disjoint. Let~$\gamma \in \Gamma$ be such that~$z' := \gamma z \in F$. If~$\gamma \in \Gamma_\infty$, it follows that~$\Im \gamma z = \Im z$. Otherwise we have on the one hand~$\Im \gamma z \geq \Im z$ since~$\gamma z \in F$, and on the other hand~$\gamma = \smat{\ast&\ast \\ c&d}$ with~$c\neq 0$ (since~$\gamma\not\in\Gamma_\infty$) and thus~$\Im \gamma z \leq \frac{1}{c^2 \Im z} \ll 1/y$. We deduce that~$y\leq \Im z' \ll 1/y$.

  Let~$\cb$ be such that~$\Im(\sigma_\cb^{-1} z') \gg 1$ (we pick~$\cb = \infty$ if~$z' \in F(Y_0)$, and otherwise~$\cb$ the unique cusp for which~$z' \in F_\cb(Y_0)$). Then~$\abs{\phi(z)} = \abs{\phi(z')} = \abs{\phi(\sigma_\cb w)}$ where~$w$ satisfies~$\Im w \gg 1$. Arguing as above, we have~$\Im(w) =\Im(z') \ll 1/y$ if~$\cb = \infty$, and otherwise~$\Im(w) \ll 1/\Im(z') \leq 1/y$ (since there are finitely many possibilities for~$\cb$). In any case~$\Im(w) \ll 1/y$.

  We conclude that
  \begin{align*}
    \abs{\phi(x + iy) - \phi_\infty(y)} \ll{}& \abs{\phi(z)} + 1 \\
    = {}& \abs{\phi(\sigma_\cb w)} + 1 \\
    \ll {}& \abs{\phi_\cb(\Im w)} + 1
  \end{align*}
  again by the exponential decay of the Whittaker function and the bound~\eqref{eq:bound-meanabs-an}. Finally by the assumption~$\Re(s_\phi) = 1/2$ and the definition of~$\beta_\phi$, we get~$\abs{\phi_\cb(\Im w)} \ll y^{-1/2}$ if~$\beta_\phi = 1$, and otherwise~$\abs{\phi_\cb(\Im w)} \ll y^{-1/2} \log(1+1/y)$ as claimed.
\end{proof}
Proceeding as in~\cite[Lemma~2.2]{DrappeauNordentoft2022}, we also have
\begin{equation}\label{eq:decay-tail-phi-diff}
  \abs{\phi(x + iy) - \phi(x'+iy)} \ll y^{-\rho_\phi - \eps} \e^{-y/3}\min(1,|x-x'|/y) \qquad (x,x'\in \R, y>0).
\end{equation}

\begin{lemma}\label{lem:properties-cE}\leavevmode
  \begin{enumerate}[label=(\roman*)]
    \item For each~$x\in \R$, the map~$s\mapsto \cE_x(\phi, s)$ is holomorphic on~$\Re(s)>\frac12+\rho_\phi$.
    \item For~$\Re(s)>\frac12+\rho_\phi$, the map~$x\mapsto \cE_x(\phi, s)$ is continuous.
    \item For each~$\eps>0$, we have~$\cE_x(\phi, s) \ll 1$ uniformly for~$\frac12+\rho_\phi+\eps \leq \Re(s) \ll 1$.
    \item For~$x\in C(\Gamma)\setminus\{\infty\}$, the map~$s\mapsto \cE_x(\phi, s)$ extends to an analytic function on
    $$ \cD_\phi := \begin{cases} \C, & (\phi \text{ cuspidal}), \\ \C\setminus\{\pm it_\phi, 1 \pm it_\phi\}, & (\phi \text{ non-cuspidal}). \end{cases} $$
    \item For each fixed~$x\in C(\Gamma)\setminus\{\infty\}$, and~$s\in\C$ satisfying
    $$ \abs{\Re(s)} \ll 1, \qquad \dist(s, \C\setminus\cD_\phi) \gg 1, $$
    we have
    \begin{equation}
      \cE_x(\phi, s) \ll_{x} 1.\label{eq:bound-Ex-anacont}
    \end{equation}
  \end{enumerate}
\end{lemma}
\begin{proof}
  The first three properties are consequences of the bound~\eqref{eq:decay-tail-phi}. The fourth property is~\cite[Proposition~2.3]{DrappeauNordentoft2022}, while the last property is a direct consequence of the uniform boundedness of~\cite[eq.~(2.13)]{DrappeauNordentoft2022}.
\end{proof}

\subsection{Hölder regularity of the period function~$\hE$ at~$s=1/2$}

In accordance with~\cite[eq.~(3.13)]{DrappeauNordentoft2022}, for all~$x\in C(\Gamma) \setminus\{\infty, \gamma^{-1}\infty\}$ and~$s\in \cD_\phi$, we let
$$ \hE_\gamma(x, s) := \abs{j(\gamma, x)}^{2s-1}\cE_{\gamma x}(\phi, s) - u_\gamma(x) \cE_x(\phi, s), $$
where we recall the definition~\eqref{multp} of~$u_\gamma(x)$.
Note that~$\hE_\gamma(x, s) = 0$ for~$\gamma \in \Gamma_\infty$; in the remainder of this section we assume
$$ \gamma \not\in \Gamma_\infty.$$

We also notice that for $\gamma_1,\gamma_2\in\Gamma$, $\hE$ satisfies the period relation
\begin{align}\label{ple}
  \hE_{\gamma_1\gamma_2}(x,s)=|j(\gamma_2, x)|^{2s-1}\hE_{\gamma_1}(\gamma_2x,s)+u_{\gamma_1}(\gamma_2x)\hE_{\gamma_2}(x,s)
\end{align}
whenever the involved quantities are defined.

\begin{lemma}\label{lem:hE-at-half}
  The map~$x\mapsto \hE_\gamma(x, \frac12)$ extends to a function on~$\R\setminus\{\gamma^{-1}\infty\}$ which is locally~$\kappa$-Hölder continuous for any~$0<\kappa < 1 - \rho_\phi$.
  More precisely, for any~$\eps>0$ we have
  \begin{equation}
    \abs{\hE_\gamma(x, \tfrac12) - \hE_\gamma(x', \tfrac12)} \ll_{ \gamma} (1+\abs{x} + \abs{x'})^{2\eps} \abs{x - x'}^{1 - \rho_\phi - \eps}\label{eq:holder-hE}
  \end{equation}
  uniformly in~$x, x'\in \R\setminus[\gamma^{-1}\infty-\eps, \gamma^{-1}\infty + \eps]$.
\end{lemma}
\begin{proof}
  This is~\cite[Theorem~3.1]{DrappeauNordentoft2022}, except that the exponent of $(1+\abs{x} + \abs{x'})$ appearing there is raised by an unspecified constant. This improved bound is already implicit in the proof in the cuspidal case, using Wilton's bound in~\cite[eq.~(2.8)]{DrappeauNordentoft2022} which we recall here for convenience:
  \begin{equation}
    \label{eq:wilton}
    \sum_{1\leq n \leq X} a(n) \e(nt) \ll_{\phi, \eps} X^\eps, \qquad (t\in\R, \phi\text{ cuspidal}).
  \end{equation}
  For the non-cuspidal case it immediately follows
  from~\cite[Theorems~3.1 and~4.1]{DrappeauNordentoft2022} unless $|x|\to\infty$ and $|x-x'|\leq |x|/2$. To handle this remaining case, one could simply import the techniques of Theorem~3.1 into the proof of Theorem~4.1 therein. An alternative method, suggested by an LLM and which is less direct but easier to implement here, is to proceed as follows.
  
  With the notation of the proof of \cite[Theorem 3.1]{DrappeauNordentoft2022}, we have
  \begin{align}\label{casdx}
    \hE_\gamma(x, \tfrac12) = -\Delta_0(x,\tfrac12)+\Delta_c(x,\tfrac12)-\Delta_\infty(x,\tfrac12),
  \end{align}
  with derivatives $\Delta'_c(x,\frac12),\Delta'_\infty(x,\frac12)\ll (1+|x|)$, by a trivial bound and~(3.17) \textit{ibid.}\ respectively. Also, $\Delta_0(x,\tfrac12)=H(x)$ satisfies $\Delta_0(x,\tfrac12)=O(|x|^{-1})$, by inserting the bound following~(4.12) in~(3.9) \textit{ibid.}\, and $\Delta_0(x,\tfrac12)-\Delta_0(x',\tfrac12)\ll \abs{x - x'}^{1/2 - \eps}$, by~(3.10) \textit{ibid.}\ if $|x-x'|\ll1$ and trivially otherwise.

  Furthermore,~\cite[Theorem~4.1]{DrappeauNordentoft2022} gives 
  \begin{align}\label{casdx2}
    \hE_\gamma(x, \tfrac12) = M(x)+E(x)
  \end{align}
  for an explicit main term $M$, clearly satisfying $M'(x)\ll |x|^{-1/2+\eps}$ and $M(x)-M(x')\ll |x-x'|^{1/2-\eps} |x|^{2\eps}$, and an error term $E(x)=O(|x|^{-1+\eps})$. 

  Combining~\eqref{casdx}-\eqref{casdx2}, we have that 
  $
  \hE_\gamma(x, \tfrac12) =M(x)-\Delta_0(x,\tfrac12)+R(x),
  $
  with 
  \begin{align*}
    R(x)&=E(x)+\Delta_0(x,\tfrac12)=\Delta_c(x,\tfrac12)-\Delta_\infty(x,\tfrac12)-M(x).
  \end{align*}
  In particular, it suffices to show that $R(x)$ satisfies the claimed bound, since $M$ and $\Delta_0$ do. This then immediately follows since the two expressions for $R$ give $R(x)\ll |x|^{-1+\eps}$ and $R'(x)\ll |x|$ respectively, whence 
  \begin{align*}
    R(x)-R(x')&\ll (|R(x)|+|R(x')|)^{1/2} |R(x)-R(x')|^{1/2}\ll |x|^{-1/2+\eps}\big(|x-x'|\cdot |x|\big)^{1/2},
  \end{align*}
  which readily gives the claim.
\end{proof}

\subsection{Upper-bound on the period function~$\hE$}

We recall from~\cite[Lemma~4.5]{DrappeauNordentoft2022} the definition
\begin{equation}
  \Upsilon_{k,\eta}(\tau) := \e^{\pi i (2-\eta) k/4} \frac{2\pi \Gamma(\tau)}{4^\tau \Gamma(\frac{\tau+1+k/2}2) \Gamma(\frac{\tau+1-k/2}2)}, \qquad (k\in\Z, \eta=\pm 1, \tau\in\C\setminus\Z_{\leq 0}),\label{eq:def-Upsilon}
\end{equation}
which will be relevant for non-cuspidal forms.
For~$s\in \C$, $s\not\in \pm(s_\phi - \frac12) + \Z_{\leq 0}$, and~$\beta \in \R_{\neq0}$, we define
\begin{equation}
  \cM_\beta(\phi, s) :=
  \begin{cases}
    A_\infty \abs{\beta}^{s_\phi - \frac12 + s} \Upsilon_{k,\eta}(s_\phi - \tfrac12 + s) + B_\infty \abs{\beta}^{\frac12-s_\phi+s} \Upsilon_{k,\eta}(\tfrac12-s_\phi+s), & (s_\phi \neq \frac12), \\
    (A_\infty + B_\infty \log \abs{\beta}) \abs{\beta}^s \Upsilon_{k,\eta}(s) + B_\infty \abs{\beta}^s \Upsilon_{k,\eta}'(s), & (s_\phi = \frac12),
  \end{cases}\label{eq:def-cM}
\end{equation}
where~$A_\infty, B_\infty$ are the coefficients in the expansion~\eqref{eq:def-phiinfty} at the cusp at~$\infty$, and~$\eta = \sign(\beta)$. We will denote them by~$A_\infty(\phi), B_\infty(\phi)$ whenever the dependence on~$\phi$ is relevant.

\begin{proposition}\label{prop:hE-upperbound-alls}
  The map~$\hE_\gamma$ extends to a function on~$(\R\setminus\{\gamma^{-1}\infty\}) \times \{\Re(s)>0\}$ with the following properties:
  \begin{itemize}
    \item for each~$x\in \R \setminus\{\gamma^{-1}\infty\}$, the map~$s\mapsto \hE_\gamma(x, s)$ is analytic in~$\{\Re(s)>0\}$;
    \item for~$\Re(s)>0$, the map~$x\mapsto \hE_\gamma(x, s)$ is continuous on~$\R\setminus\{\gamma^{-1}\infty\}$;
    \item given~$\eps>0$, uniformly for~$\eps\leq \Re(s) \ll 1$ and~$\abs{x - \gamma^{-1}\infty} \geq \eps$, we have
    \begin{equation}
      \begin{aligned}
        \hE_\gamma(x, s) \ll_{ \gamma} (1+\abs{s}) (1 + \abs{x - \gamma^{-1}\infty}^{\max(\Re(s) - \frac12 + \rho_\phi + \eps, 2\Re(s) - 1)} ).
      \end{aligned}
      \label{eq:bound-hEgamma}
    \end{equation}
    If~$s=1/2$, we have the more precise estimate
    \begin{equation}
      \label{eq:estim-hEgamma-half}
      \begin{aligned}
        \hE_\gamma(x, \tfrac12) ={}& \chi(\gamma) \sign(c_\gamma)^{-k} \big( \cM_{x - \gamma^{-1}\infty}(\phi, \tfrac12) + i^k \cE_{\gamma^{-1}\infty}(\phi, \tfrac12)\big) \\
        &{} + O_{\phi, \gamma, \eps} ( \abs{x - \gamma^{-1}\infty}^{-1}).
      \end{aligned}
    \end{equation}
  \end{itemize}
\end{proposition}

The case~$s=1/2$ of this proposition corresponds to~\cite[Theorem~4.1]{DrappeauNordentoft2022}.
Compared with~\cite{DrappeauNordentoft2022}, the need for uniformity in~$s$ requires a delicate analysis of certain terms denoted~$\cR_{\alpha, \upsilon}(\phi, s)$ below, which happen to vanish identically at~$s=1/2$ (as in~\cite{DrappeauNordentoft2022}), but do not vanish anymore in our context.

For all~$\alpha, \beta \in \R$ with~$\alpha \neq \beta$, let

\begin{minipage}{0.70\textwidth}
  \begin{align*} &\delta_{\alpha, \beta} := \tfrac12(\alpha-\beta), \\ &\xi_{\alpha, \beta} := \beta + \delta_{\alpha,\beta} + i\abs{\delta_{\alpha,\beta}}. \end{align*}
\end{minipage}
\hfill
\begin{minipage}{0.30\textwidth}
  \begin{tikzpicture}[scale=1.3]
    \draw[->] (-1.2,0) -- (1.2,0) node[right] {};

    \draw (0,0) ++(-1,0) arc[start angle=180, end angle=0, radius=1];
    \fill (-1,0) circle (0.02) node[below left] {$\beta$};
    \fill (1,0) circle (0.02) node[below] {$\alpha$};
    \fill (0,1) circle (0.02) node[above] {$\xi_{\alpha, \beta}$};

    \draw[dashed, <->] (-1,-0.1) -- (0,-0.1) node[midway,below] {$\abs{\delta_{\alpha, \beta}}$};
  \end{tikzpicture}
\end{minipage}
Hence~$\xi_{\alpha,\beta}$ is the point of highest ordinate on the geodesic from~$\alpha$ to~$\beta$.

For~$\Re(s)>1$, let
$$ \cE^*_{\alpha,\beta}(\phi, s) := \int_\alpha^{\xi_{\alpha,\beta}} \Big(\frac{z-\beta}{\abs{z-\beta}}\Big)^k (\phi(z) - \phi_\infty(\Im(z))) (\Im z)^{s-1/2} \df s(z), $$
where the integral is taken along the geodesic from~$\alpha$ to~$\xi_{\alpha,\beta}$, which is a quarter-circle, and $\df s(z)$ denotes the $\SL(2, \R)$-invariant metric $\df s(z) = \sqrt{\df x^2 + \df y^2}/y$.
Let also
\begin{equation}
  \cE_{\alpha,\beta}(\phi, s) := \cE^*_{\alpha,\beta}(\phi, s) + (-i\sign(\delta_{\alpha,\beta}))^k \cE^*_{\beta,\alpha}(\phi, s). \label{eq:def-Ealphabeta}
\end{equation}

Also, given~$\upsilon\in \SL(2, \R)$, we let~$\phi^\upsilon(z) = j_\upsilon(z)^{-k} \phi(\upsilon z)$. Thus~$\phi^\upsilon$ is an automorphic form of same weight, character and eigenvalue as~$\phi$ for the group~$\Gamma^\upsilon := \upsilon^{-1}\Gamma \upsilon$, with~$\rho_{\phi^\upsilon} = \rho_\phi$. Also note that~$C(\Gamma^\upsilon) = \upsilon^{-1} C(\Gamma)$.
Finally note that if~$\upsilon\in \Gamma$, then~$\phi^\upsilon = \chi(\upsilon) \phi$.

\begin{lemma}\label{lem:Ex-Eab}
  Let~$\upsilon \in \SL(2, \R)$ with non-zero bottom-left coefficient~$c=c_{\upsilon} \neq 0$, and let~$\alpha \in \R\setminus\{\upsilon^{-1}\infty\}$. Assume that~$\upsilon \alpha, \upsilon \infty \in C(\Gamma)$. Then for~$\Re(s)>1$ we have
  $$ \abs{j(\upsilon, \alpha)}^{2s-1} \cE_{\upsilon \alpha}(\phi, s) = \sign(c)^{-k} \big( \cE_{\alpha, \upsilon^{-1}\infty}(\phi^\upsilon, s) + \cM_{\alpha - \upsilon^{-1}\infty}(\phi^\upsilon, s)\big) + \cR_{\alpha, \upsilon}(\phi, s), $$
  with the definition~\eqref{eq:def-cM}, where~$\cR_{\alpha, \upsilon}(\phi, s)$ is holomorphic in~$s$ on~$\{\Re(s)>0\}\cap \cD_\phi$,
  is continuous with respect to~$\alpha$ on~$\R\setminus\{\upsilon^{-1}\infty\}$ uniformly in compact sets of~$\{\Re(s)>0\}\cap \cD_\phi$, and
  satisfies for~$\eps\leq \Re(s)\ll 1$ and~$\abs{s - (1 \pm it_\phi)} \geq \eps$ the estimate
  \begin{multline}\label{eq:bound-cR}
    \cR_{\alpha, \upsilon}(\phi, s) \ll_{\phi, \eps} (1+\abs{s})\big(|c|^{\frac12-\Re(s)+\rho_\phi+\eps} \abs{j(\upsilon, \alpha)}^{\Re(s)-\frac12+\rho_\phi+\eps} + \abs{j(\upsilon, \alpha)}^{2\Re(s)-1} \\
    + \abs{c}^{1-2\Re(s)} + (1+\abs{\log(\abs{c j(\upsilon, \alpha)})}) \abs{j(\upsilon, \alpha)}^{\Re(s)-1} \abs{c}^{-\Re(s)}\big)
  \end{multline}
  uniformly in~$\upsilon, \alpha$ and~$s$. Finally, we have~$\cR_{\alpha, \upsilon}(\phi, \frac12) = 0$.
\end{lemma}

\begin{proof}
  Denote~$\kappa_\upsilon := \upsilon^{-1}\infty$. For~$\Re(s)>1$ we write
  \begin{align*}
    \abs{j(\upsilon, \alpha)}^{2s-1} \cE_{\upsilon \alpha}(\phi, s) ={}& \int_{\upsilon\alpha}^\infty(\phi(z) - \phi_\infty(\Im z))\Big(\abs{j(\upsilon, \alpha)}^{2s-1} (\Im z)^{s-\frac12} - (\Im \upsilon^{-1} z)^{s-\frac12}\Big)\df s(z) \\
    &{} + \int_{\upsilon\alpha}^\infty \phi(z) (\Im \upsilon^{-1} z)^{s-\frac12} \df s(z) - \int_{\upsilon\alpha}^\infty \phi_\infty(\Im z) (\Im \upsilon^{-1} z)^{s-\frac12} \df s(z) \\
    ={}& \cI_1 + \cI_2 - \cI_3,
  \end{align*}
  say.

  We first focus on~$\cI_2$. Changing variables, we have
  \begin{align*}
    \cI_2 = {}& \int_\alpha^{\kappa_\upsilon} j_\upsilon(z)^k \phi^\upsilon(z) (\Im z)^{s-\frac12} \df s(z) \\
    ={}& \int_\alpha^{\kappa_\upsilon} j_\upsilon(z)^k (\phi^\upsilon(z) - \phi_\infty^\upsilon(\Im z))(\Im z)^{s-\frac12} \df s(z) + \int_\alpha^{\kappa_\upsilon} j_\upsilon(z)^k \phi_\infty^\upsilon(\Im z) (\Im z)^{s - \frac12} \df s(z).
  \end{align*}
  The second integral is evaluated as in~\cite[eq.~(4.17)]{DrappeauNordentoft2022}; it equals~$\sign(c)^{-k} \cM_{\alpha - \kappa_\upsilon}(\phi^\upsilon, s)$. In the first integral, we split the path of integration as~$\alpha \to \xi_{\alpha, \kappa_\upsilon}$ concatenated with~$\xi_{\alpha,\kappa_\upsilon} \to \kappa_\upsilon$. Since~$j_\upsilon(z) = \sign(c) \frac{z-\kappa_\upsilon}{\abs{z-\kappa_\upsilon}}$, the first portion contributes~$\sign(c)^{-k} \cE^*_{\alpha,\kappa_\upsilon}(\phi^\upsilon, s)$. Finally, for~$z$ in the geodesic~$\alpha \to \kappa_\upsilon$, we have~$\frac{z-\alpha}{\abs{z-\alpha}} = \sign(\alpha - \kappa_\upsilon) i \frac{z-\kappa_\upsilon}{\abs{z-\kappa_\upsilon}}$, so that the second portion contributes~$(i \sign(\kappa_\upsilon - \alpha))^k (\sign c)^{-k} \cE^*_{\kappa_\upsilon, \alpha}(\phi^\upsilon, s)$. We thus find
  $$ \cI_2 = \sign(c)^{-k} \big(\cE_{\alpha, \kappa_\upsilon}(\phi^\upsilon, s) + \cM_{\alpha - \kappa_\upsilon}(\phi^\upsilon, s)\big), $$
  which is our claimed main term.

  Let~$\cR_{\alpha, \upsilon}(\phi, s) := \cI_1 - \cI_3$. To finish the proof of Lemma~\ref{lem:Ex-Eab}, it remains to show that each of~$\cI_1$ and~$\cI_3$ satisfies the properties claimed about~$\cR$ in the statement. Consider first~$\cI_3$, which appears only when~$\phi$ is non-cuspidal. Recall that in that case, we assumed~$\Re(s_\phi) = \frac12$. Then, following the estimation of~\cite[eq.~(4.7)]{DrappeauNordentoft2022}, we have for~$\Re(s)>1$ the expression
  \begin{align*}
    \cI_3 = {}& \abs{\alpha - \kappa_\upsilon}^{s-\frac12} \int_0^\infty \phi_\infty\Big(\frac{y}{c^2 \abs{\alpha - \kappa_\upsilon}}\Big) \Big(\frac{y}{1+y^2}\Big)^{s-\frac12} \frac{\df y}y \\
    = {}& \abs{\alpha - \kappa_\upsilon}^{s-\frac12}
    \begin{cases}
      \frac{A_\infty}{(c^2 \abs{\alpha - \kappa_\upsilon})^{s_\phi}} g_{s_\phi}(s) + \frac{B_\infty}{(c^2 \abs{\alpha - \kappa_\upsilon})^{1-s_\phi}} g_{1-s_\phi}(s), & (s_\phi \neq \frac12), \\
      \frac{A_\infty - B_\infty \log(c^2 \abs{\alpha - \kappa_\upsilon})}{(c^2 \abs{\alpha - \kappa_\upsilon})^{1/2}} g_{\frac12}(s) + \frac{B_\infty}{(c^2 \abs{\alpha - \kappa_\upsilon})^{\frac12}} g_{\frac12}^*(s), & (s_\phi = \frac12),
    \end{cases}
  \end{align*}
  with
  $$ g_\tau(s) := \int_0^\infty \frac{y^{\tau + s - 1/2}}{(1+y^2)^{s - 1/2}} \frac{\df y}{y} = \frac{\Gamma(\frac{s+\tau-1/2}2) \Gamma(\frac{s-\tau-1/2}2)}{2\Gamma(s-1/2)}, \qquad g_\tau^*(s) = \frac{\df}{\df \tau} g_\tau(s). $$
  Since~$\Re(s_\phi) = \frac12$, we see from this explicit expression that~$\cI_3 = \cI_3(s, \alpha)$ is holomorphic in~$s$ for~$\Re(s)>0$ and~$s\in\cD_\phi$.
  For each such~$s$, it is continuous with respect to~$\alpha$ in~$\R\setminus\{\kappa_\upsilon\}$.
  By the Stirling formula, it satisfies for~$\abs{s - (1\pm it_\phi)}\gg 1$ the estimate
  \begin{equation}
    \cI_3 \ll \abs{c}^{-\Re(s)} \abs{j(\upsilon, \alpha)}^{\Re(s)-1} (1 + \abs{\log\abs{cj(\upsilon, \alpha)}}) (1+\abs{s})^{-1/2}. \label{eq:bound-phig-changingIm-I3}
  \end{equation}
  We clearly have~$\cI_3(\frac12, \alpha) = 0$, since~$g_\tau(\frac12) = g_\tau^*(\frac12) = 0$.

  Finally we examine~$\cI_1 = \cI_1(s, \alpha)$, which clearly vanishes at~$s=\frac12$. We have
  \[ \abs{j(\upsilon, \alpha)}^{2s-1} (\Im z)^{s-\frac12} - (\Im \upsilon^{-1} z)^{s - \frac12}
    = \abs{j(\upsilon, \alpha)}^{2s-1} (\Im z)^{s-\frac12} \Big( 1 - \abs{\frac{j(\upsilon^{-1}, z)}{j(\upsilon^{-1}, \upsilon \alpha)}}^{1 - 2s}\Big), \]
  and for~$z = \upsilon \alpha + iy$, $y>0$,
  $$ \frac{j(\upsilon^{-1}, z)}{j(\upsilon^{-1}, \upsilon \alpha)} = 1 - \frac{icy}{j(\upsilon^{-1}, \upsilon \alpha)} = 1 - icy j(\upsilon, \alpha). $$
  Thus, $\abs{\frac{j(\upsilon^{-1}, z)}{j(\upsilon^{-1}, \upsilon \alpha)}}^2 = 1 + c^2 y^2 j(\upsilon, \alpha)^2$ so that
  \begin{align*}
    \Delta(y, s) :={}& \abs{j(\upsilon, \alpha)}^{2s-1} (\Im z)^{s-\frac12} - (\Im \upsilon^{-1} z)^{s-\frac12} \\
    ={}& \abs{j(\upsilon, \alpha)}^{2s-1} y^{s-\frac12} \big(1 - (1 + c^2 y^2 j(\upsilon, \alpha)^2)^{\frac12 - s}\big) \\
    \ll{}& \abs{j(\upsilon, \alpha)^2 y}^{\Re(s) - \frac12}
    \begin{cases}
      (1+\abs{s}) c^2 j(\upsilon, \alpha)^2 y^2, & (y < 1/ |c j(\upsilon, \alpha)|), \\
      1 + \abs{c j(\upsilon, \alpha) y}^{1 - 2\Re(s)}, & (\text{otherwise}).
    \end{cases}
  \end{align*}
  Inserting this estimate in~$\cI_1$ and using the bound~\eqref{eq:decay-tail-phi},
  we deduce that~$\cI_1 = \cI_1(s, \alpha)$ is holomorphic in~$\{\Re s >-\frac32+\rho_\phi\}$.
  By uniform convergence, it is a continuous function of~$\alpha$ in~$\R\setminus\{\kappa_\upsilon\}$, and we have
  \begin{align*}
    \cI_1 \ll{}& \int_0^1 y^{-\rho_\phi-\eps} \abs{\Delta(y, s)} \frac{\df y}y + \int_1^\infty \e^{-y/3} \abs{\Delta(y, s)} \frac{\df y}y \\
    \ll{}& (1+\abs{s}) \abs{c}^{\frac12-\Re(s)+\rho_\phi+\eps} \abs{j(\upsilon, \alpha)}^{\Re(s)-\frac12+\rho_\phi+\eps} + \abs{c}^{1-2\Re(s)} + \abs{j(\upsilon, \alpha)}^{2\Re(s)-1}.
  \end{align*}
  Together with~\eqref{eq:bound-phig-changingIm-I3}, this provides the claimed bound on~$\cR_{\alpha, \upsilon}(\phi, s)$.
\end{proof}

\begin{lemma}\label{lem:Eab-Ex}
  Let~$\alpha, \beta \in C(\Gamma) \setminus\{\infty\}$ with~$\alpha\neq \beta$. Then for~$\Re(s)>1$ we have
  $$ \cE_{\alpha, \beta}(\phi, s) = \sign(\delta_{\alpha, \beta})^k \cE_\alpha(\phi, s) + i^k \cE_\beta(\phi, s) + \cT_{\alpha, \beta}(\phi, s), $$
  where~$\cT_{\alpha, \beta}(\phi, s)$ is holomorphic in~$s$ on~$\{\Re(s) > -\frac12+\rho_\phi\}$. For each such~$s$, it is continuous in~$\alpha, \beta\in \R$ with~$\alpha \neq \beta$ uniformly in $s$ in compact sets,
  and for~$-\tfrac12+\rho_\phi + \eps < \Re(s) \ll 1$, it satisfies
  \begin{equation}\label{eq:bound-cT}
    \cT_{\alpha, \beta}(\phi, s) \ll
    \begin{cases}
      \abs{\alpha - \beta}^{-1}, & (\abs{\alpha - \beta} \gg 1), \\
      1 + \abs{\alpha - \beta}^{\Re(s) - \rho_\phi - \frac12 - \eps}, & (\abs{\alpha - \beta}\ll 1).
    \end{cases}
  \end{equation}
\end{lemma}
\begin{proof}
  From the definition~\eqref{eq:def-Ealphabeta} of~$\cE_{\alpha, \beta}(\phi, s)$, it suffices to prove that the difference
  \begin{equation}
    \cE^*_{\alpha, \beta}(\phi, s) - \sign(\delta_{\alpha, \beta})^k \cE_\alpha(\phi, s)\label{eq:estim-phig-largleg-1}
  \end{equation}
  satisfies the properties claimed on~$\cT_{\alpha, \beta}(\phi, s)$ in the lemma, and the same for the difference
  $$ \cE^*_{\beta, \alpha}(\phi, s) - \sign(\delta_{\beta, \alpha})^k \cE_\beta(\phi, s). $$
  By symmetry we may focus \emph{e.g.} on~\eqref{eq:estim-phig-largleg-1}. We parametrize the arc~$\alpha \to \xi_{\alpha, \beta}$ as
  $$ z_{\alpha, \beta}(y) := \alpha - \nu_{\alpha, \beta}(y) + iy $$
  for~$y\in (0, \abs{\delta_{\alpha, \beta}})$, where
  $$ \nu_{\alpha, \beta}(y) = \delta_{\alpha,\beta} \big( 1 - (1-y^2/\delta_{\alpha,\beta}^2)^{1/2} \big), $$
  see~\cite[eq.~(3.5)]{DrappeauNordentoft2022}. The arc length element is~$\mu_{\alpha,\beta}(y)\df y / y$ with
  $$ \mu_{\alpha,\beta}(y) = (1- y^2 / \delta_{\alpha,\beta}^2)^{-1/2}. $$
  For clarity we drop the indices~$\alpha,\beta$ in the foregoing computation. A quick computation shows that
  \[ \nu(y) \ll y^2 / \abs{\delta},
    \quad \mu(y) - 1 \ll \frac{y^2}{\abs{\delta}^{3/2}\sqrt{\abs{\delta}-y}},
    \quad \Big(\frac{z(y) - \beta}{\abs{z(y) - \beta}}\Big)^k - \sign(\delta)^k \ll y / \abs{\delta}.
  \]
  We find
  \begin{align*}
    \cE^*_{\alpha,\beta}(\phi, s) &{}- \sign(\delta)^k \cE_\alpha(\phi, s) \\
    ={}& \int_0^{\abs{\delta}} \bigg( \Big(\frac{z(y)-\beta}{\abs{z(y)-\beta}}\Big)^k \mu(y) (\phi(z(y)) - \phi_\infty(y)) - \sign(\delta)^k (\phi(\alpha + iy) - \phi_\infty(y)) \bigg) y^{s-\frac12} \frac{\df y}y \\
    & \qquad {} - \sign(\delta)^k \int_{\abs{\delta}}^\infty (\phi(\alpha + iy) - \phi_\infty(y)) y^{s-\frac12} \frac{\df y}y.
  \end{align*}
  By~\eqref{eq:decay-tail-phi} the second integral is clearly $O(1+|\delta|^{\Re(s) - \rho_\phi - \frac12 - \eps})$ if $|\delta|<1$ and is $O(|\delta|^{-1})$ otherwise, whereas if $\Re(s)>-\tfrac12+\rho_\phi+\eps$ using the estimates~\eqref{eq:decay-tail-phi}-\eqref{eq:decay-tail-phi-diff} we find
  \begin{align*}
    & \int_0^{\abs{\delta}} \Big\{ \abs{\mu(y) - 1} y^{-\rho_\phi}\e^{-y/3} + \abs{\Big(\frac{z(y) - \beta}{\abs{z(y)-\beta}}\Big)^k - \sign(\delta)^k} y^{-\rho_\phi}\e^{-y/3} \\
    & \hspace{18em} + \abs{\mu(y)} \abs{\phi(z(y)) - \phi(\alpha + iy)}\Big\} y^{\Re(s)-1/2-\eps}\frac{\df y}{y} \\
    & \hspace{5em} \ll{} \int_0^{\abs{\delta}} \Big\{ \frac{y^{2-\rho_\phi}\e^{-y/3}}{\abs{\delta}^{3/2} \sqrt{\abs{\delta} - y}} + \frac{y^{1-\rho_\phi}(1+y)\e^{-y/3}}{\abs{\delta}\sqrt{1-y/\abs{\delta}}}\Big\} y^{\Re(s) - 1/2 - \eps}\frac{\df y}y \\
    & \hspace{5em} \ll{} \abs{\delta}^{-1}( 1 + \abs{\delta}^{-1})^{-\Re(s)-\frac12+\rho_\phi+\eps}.
  \end{align*}
  Since~$\abs{\delta} \asymp \abs{\alpha - \beta}$, this implies~\eqref{eq:bound-cT}. Finally, the continuity in $\alpha$ and $\beta$ follows by the uniform convergence of the various integrals.
\end{proof}

\begin{proof}[Proof of Proposition~\ref{prop:hE-upperbound-alls}]
  Assume~$\Re(s)>1$. Let~$c=c_\gamma$ be the lower-left entry of~$\gamma$, and note that
  $$ u_\gamma(x) = (\sign c)^{-k} \sign(\delta_{x,\gamma^{-1}\infty})^k \chi(\gamma), \quad\text{ and } \qquad \phi^\gamma = \chi(\gamma) \phi. $$
  Using Lemmas~\ref{lem:Ex-Eab} and \ref{lem:Eab-Ex}, for~$\Re(s)>1$ we have
  \begin{align*}
    \abs{j(\gamma, x)}^{2s-1} \cE_{\gamma x}(\phi, s)
    ={}& \sign(c)^{-k} \chi(\gamma)\big(\cE_{x,\gamma^{-1}\infty}(\phi, s) + \cM_{x - \gamma^{-1}\infty}(\phi, s)\big) + \cR_{x, \gamma}(\phi, s) \\
    ={}& u_\gamma(x) \cE_x(\phi, s) + i^k (\sign c)^{-k} \chi(\gamma) \cE_{\gamma^{-1}\infty}(\phi, s) \\
    &{} + (\sign c)^{-k} \chi(\gamma) \big(\cM_{x-\gamma^{-1}\infty}(\phi, s) + \cT_{x,\gamma^{-1}\infty}(\phi, s)\big) + \cR_{x,\gamma}(\phi, s).
  \end{align*}
  This identity holds for~$\Re(s)>0$, $s\in\cD_\phi$ by the same lemmas. Thus, for~$s$ outside of an~$\eps$-neighbourhood of~$1\pm it_\phi$, the properties claimed on the holomorphicity in~$s$ and continuity in~$x$ follow from the corresponding properties of~$\cR$ and~$\cT$ established in Lemmas~\ref{lem:Ex-Eab} and \ref{lem:Eab-Ex}, as well as the uniform bound~\eqref{eq:bound-hEgamma} for~$\Re(s) \geq \eps$.

  Next, using the values given in~\cite[Proof of Proposition~2.3]{DrappeauNordentoft2022}, we compute the polar part of the Laurent expansion of the map~$s \mapsto \hE_\gamma(x, s)$ at~$s = 1 \pm i t_\phi$ and see that these terms vanish. It follows that~$\hE_\gamma(x, s)$ is in fact holomorphic at these points. The bound~\eqref{eq:bound-hEgamma} persists by the maximum modulus principle.

  By the Cauchy formula and the bound~\eqref{eq:bound-hEgamma}, we deduce the continuity of~$x\mapsto \hE_\gamma(x, s)$ at~$s=1\pm it_\phi$, which finishes the proof of the first two points of Proposition~\ref{prop:hE-upperbound-alls}.

  If~$s=1/2$, then by Lemma~\ref{lem:Ex-Eab} we have~$\cR_{\ast}(\phi, 1/2) = 0$, and so the error term arises from the upper-bound on~$\cT$ given in Lemma~\ref{lem:Eab-Ex}.
\end{proof}

\subsection{The period function~$\hE$ at parabolic elements}

\begin{proposition}\label{prop:hE-gamman-estim}
  Let~$p\in \Gamma$ be a parabolic element with fixed point~$\omega \in C(\Gamma)\setminus\{\infty\}$, $\sigma\in \SL(2, \R)$ a scaling matrix for~$\omega$ and $m_0\in\Z_{\neq0}$ be such that~$\sigma^{-1} p \sigma = \pm \smat{1&m_0 \\ & 1}$.
  Then for any~$\eps>0$, there exists~$n_0\geq 1$ such that for~$n\geq n_0$ and~$\abs{x-\omega}\geq \eps$, the following estimates hold.
  \begin{itemize}
    \item For~$\eps < \Re(s) \ll 1$, we have
    \begin{equation}
      \begin{aligned}
        \hE_{p^n}(x, s) \ll_{p} (1+\abs{s}) \Big(&(1+\abs{x})^{2\Re(s)-1} \big\{1+ n^{2\Re(s)-1} + n^{\Re(s)-\frac12+\rho_\phi+\eps}\big\} \\
        & {}+ 1 + (1+\abs{x})^{\Re(s)-\frac12+\rho_\phi+\eps} \Big).
      \end{aligned}\label{eq:hE-rough}
    \end{equation}
    \item For~$s=1/2$, we have
    \begin{equation}
      \begin{aligned}
        \hE_{p^n}(x, \tfrac12) ={}& (i \sign(c_\sigma))^k \cE_{\sigma^{-1}\infty}(\phi^\sigma, \tfrac12) + u_{p^n}(x) (i\sign(x-\omega))^k \cE_\omega(\phi, \tfrac12) \\
        & {} + \sign(c_\sigma)^{-k} \cM_{nm_0 - 1/(c_\sigma^2(x-\omega))}(\phi^\sigma, \tfrac12) + u_{p^n}(x) \sign(x-\omega)^k \cM_{x-\omega}(\phi, \tfrac12) \\
        & {} + O_{p}\big(n^{-1} + (1+\abs{x})^{-1}\big).
      \end{aligned}\label{eq:hE-half}
    \end{equation}
  \end{itemize}
\end{proposition}

\begin{proof}
First we note that~$\hE_{-\gamma}(x, s) = \hE_\gamma(x, s)$, so that we may assume~$\sigma^{-1} p \sigma = \smat{1&m_0 \\ & 1}$. Suppose~$\Re(s)>1$, and let~$\gamma = p^n$, which is also parabolic with fixed point~$\omega$.
  By picking~$n\geq n_0$ large enough in terms of~$\eps$, we may suppose that~$\gamma^{-1}\infty$ is close enough to~$\omega$ so as to ensure~$\abs{x - \gamma^{-1}\infty} \gg 1$ uniformly in~$x$. We recall that the definition of~$\sigma$ amounts to the conditions~$\sigma \infty = \omega$ and~$\Gamma_\omega = \sigma \Gamma_\infty \sigma^{-1}$.
  We also note that~$\abs{c_\sigma}$ depends only on~$\omega$ and thus only on $p$.

  First we use Lemma~\ref{lem:Ex-Eab} with~$\upsilon \gets \sigma$ and~$\alpha \gets \sigma^{-1}\gamma x$, getting
  \begin{equation}
    \begin{aligned}
      &\abs{j(\sigma, \sigma^{-1}\gamma x)}^{2s-1} \cE_{\gamma x}(\phi, s) \\
      &{} = \sign(c_\sigma)^{-k} \big(\cE_{\sigma^{-1}\gamma x, \sigma^{-1}\infty}(\phi^\sigma, s) + \cM_{\sigma^{-1}\gamma x - \sigma^{-1}\infty}(\phi^\sigma, s)\big) + \cR_{\sigma^{-1}\gamma x, \sigma}(\phi, s).\label{eq:parab-1}
    \end{aligned}
  \end{equation}
  Next we use Lemma~\ref{lem:Eab-Ex} with~$(\alpha, \beta) \gets (\sigma^{-1}\gamma x, \sigma^{-1}\infty)$. A quick computation shows that~$\sign(\delta_{\sigma^{-1}\gamma x, \sigma^{-1}\infty}) = -\sign(\delta_{\gamma x, \omega})$, thus we find
  \begin{equation}
    \begin{aligned}
      & \cE_{\sigma^{-1}\gamma x, \sigma^{-1}\infty}(\phi^\sigma, s) \\
      & {} = (-\sign(\delta_{\gamma x, \omega}))^k \cE_{\sigma^{-1} \gamma x}(\phi^\sigma, s) + i^k \cE_{\sigma^{-1}\infty}(\phi^\sigma, s) + \cT_{\sigma^{-1}\gamma x, \sigma^{-1}\infty}(\phi^\sigma, s).\label{eq:parab-2}
    \end{aligned}
  \end{equation}
  Next, we use Lemma~\ref{lem:Ex-Eab} again with~$\upsilon \gets \sigma^{-1}\gamma$ and~$\alpha \gets x$. Note that~$[\phi^\sigma]^{\sigma^{-1}\gamma}(z) = \chi(\gamma) \phi(z)$, and~$(\sigma^{-1}\gamma)^{-1}\infty = \gamma^{-1} \omega = \omega$.
  We obtain
  \begin{equation}
    \label{eq:parab-3}
    \begin{aligned}
      & \abs{j(\sigma^{-1}\gamma, x)}^{2s-1} \cE_{\sigma^{-1}\gamma x}(\phi^\sigma, s) \\
      {}& = \chi(\gamma) \sign(c_{\sigma^{-1}\gamma})^{-k}\big(\cE_{x,\omega}(\phi, s) + \cM_{x-\omega}(\phi, s)\big) + \cR_{x, \sigma^{-1}\gamma}(\phi^\sigma, s).
    \end{aligned}
  \end{equation}
  Finally, by Lemma~\ref{lem:Eab-Ex} again, with~$(\alpha, \beta) \gets (x, \omega)$, we get
  \begin{equation}
    \cE_{x,\omega}(\phi, s) = \sign(\delta_{x,\omega})^k \cE_x(\phi, s) + i^k \cE_\omega(\phi, s) + \cT_{x, \omega}(\phi, s).\label{eq:parab-4}
  \end{equation}
  We insert the equality~\eqref{eq:parab-4} successively in~\eqref{eq:parab-3}, \eqref{eq:parab-2} and~\eqref{eq:parab-1}, and multiply throughout by~$\abs{j(\sigma^{-1}\gamma, x)}^{2s-1} = \abs{j(\gamma, x)/j(\sigma, \sigma^{-1} \gamma x)}^{2s-1}$. By the identity
  $$ \sign(\delta_{y, \tau^{-1}\infty}) = \sign(c_\tau) j_\tau(y), \qquad (\tau\in \SL(2, \R), y\neq \tau^{-1}\infty), $$
  we get
  $$ \sign(\delta_{x,\omega}) \sign(c_{\sigma^{-1}\gamma}) = j_{\sigma^{-1} \gamma}(x), \quad \sign(\delta_{\gamma x, \omega}) \sign(c_\sigma) = - j_{\sigma}(\sigma^{-1} \gamma x). $$
  Thus, the term~$\cE_x(\phi, s)$ arises with a factor of exactly~$\chi(\gamma) j_\gamma(x)^k = u_\gamma(x)$ by the chain rule, and we get
  \begin{align*}
    \hE_{p^n}(x, s)
    ={}&\abs{j(\gamma, x)}^{2s-1} \cE_{\gamma x}(\phi, s) - u_\gamma(x) \cE_x(\phi, s) \\
    ={}& (i \sign(c_\sigma))^k \abs{j(\sigma^{-1}\gamma, x)}^{2s-1} \cE_{\sigma^{-1}\infty}(\phi^\sigma, s) + u_\gamma(x) (i \sign(\delta_{x,\omega}))^k \cE_\omega(\phi, s) \\
    & {} + \abs{j(\sigma^{-1}\gamma, x)}^{2s-1} \sign(c_\sigma)^{-k} \cM_{\sigma^{-1}\gamma x - \sigma^{-1}\infty}(\phi^\sigma, s) + u_\gamma(x) \sign(\delta_{x,\omega})^k \cM_{x-\omega}(\phi, s) \\
    & {} + \abs{j(\sigma^{-1} \gamma, x)}^{2s-1} \cR_{\sigma^{-1}\gamma x, \sigma}(\phi, s) + j_\sigma(\sigma^{-1} \gamma x)^k \cR_{x, \sigma^{-1}\gamma}(\phi^\sigma, s) \\
    & {} + \sign(c_\sigma)^k \abs{j(\sigma^{-1}\gamma, x)}^{2s-1} \cT_{\sigma^{-1} \gamma x, \sigma^{-1}\infty}(\phi^\sigma, s) + u_\gamma(x) \sign(\delta_{x,\omega})^k \cT_{x,\omega}(\phi, s).
  \end{align*}
  Let~$\tau = \sigma^{-1} p \sigma= \smat{1&m_0 \\ & 1}$. Then~$\gamma = \sigma \tau^n \sigma^{-1}$. Write~$m = m_0 n$, so that~$\sigma^{-1} \gamma \sigma = \smat{1 & m \\ & 1}$. Using the property~$\sigma \infty = \omega$, we deduce
  \begin{align*}
    & j(\sigma^{-1} \gamma, x) = j(\tau^n \sigma^{-1}, x) = j(\sigma^{-1}, x), \\
    & \sigma^{-1} \gamma x - \sigma^{-1}\infty = \tau^n \sigma^{-1} x - \sigma^{-1}\infty = m + \sigma^{-1} x - \sigma^{-1}\infty,
  \end{align*}
  Hence the above simplifies to
  \begin{equation}
    \begin{split} \label{eq:expr-hE-final}
      \hE_{p^n}(x, s) & = (i \sign(c_\sigma))^k \abs{j(\sigma^{-1}, x)}^{2s-1} \cE_{\sigma^{-1}\infty}(\phi^\sigma, s) + u_\gamma(x) (i \sign(\delta_{x,\omega}))^k \cE_\omega(\phi, s) \\
      &\quad + \abs{j(\sigma^{-1}, x)}^{2s-1} \sign(c_\sigma)^{-k} \cM_{m + \sigma^{-1}x - \sigma^{-1}\infty}(\phi^\sigma, s) + u_\gamma(x) \sign(\delta_{x,\omega})^k \cM_{x-\omega}(\phi, s) \\
      &\quad + \abs{j(\sigma^{-1}, x)}^{2s-1} \cR_{m+\sigma^{-1}x , \sigma}(\phi, s) + j_\sigma(\sigma^{-1} \gamma x)^k \cR_{x, \tau^n \sigma^{-1}}(\phi^\sigma, s) \\
      &\quad + \sign(c_\sigma)^k \abs{j(\sigma^{-1}, x)}^{2s-1} \cT_{m+\sigma^{-1}x , \sigma^{-1}\infty}(\phi^\sigma, s) + u_\gamma(x) \sign(\delta_{x,\omega})^k \cT_{x,\omega}(\phi, s).
    \end{split}
  \end{equation}
  By Lemmas~\ref{lem:properties-cE}, \ref{lem:Ex-Eab} and~\ref{lem:Eab-Ex}, we get that this identity holds for
  $$ (x, s) \in (\R\setminus\{\gamma^{-1}\infty, \omega\}) \times (\{\Re(s)>0\}\cap \cD_\phi), $$
  and that $ \hE_{p^n}(x, s)$ is continuous in~$x$ for each~$s$, and holomorphic in~$s$ for each~$x$.

  Assume first~$\Re(s)\geq \eps$ and~$\abs{s - 1 \pm it_\phi}\geq \eps$. Since~$\abs{x-\omega}\geq \eps$, we have
  $$ \abs{j(\sigma^{-1}, x)} \asymp_p \abs{x-\omega}\gg_\eps 1,\qquad \sigma^{-1}x=O(1) $$
  and additionally
  $$ \abs{\sigma^{-1}x + m - \sigma^{-1}\infty} = \abs{\sigma^{-1}\gamma x - \sigma^{-1} \infty} \gg_\sigma \min(1, \abs{x - \gamma^{-1}\infty}) \gg 1. $$
  Thus, since $c_{\tau^n\sigma^{-1}}=c_{\sigma^{-1}}\asymp_p 1$ and recalling $ j(\tau^n \sigma^{-1}, x) = j(\sigma^{-1}, x)$, inserting the bounds~\eqref{eq:bound-Ex-anacont}, \eqref{eq:def-cM}, \eqref{eq:bound-cR} and \eqref{eq:bound-cT} in the expression~\eqref{eq:expr-hE-final} yields~\eqref{eq:hE-rough}. This bound persists for~$\abs{s - 1 \pm it_\phi}\leq \eps$ by the maximum modulus principle.

  In the special case when~$s=1/2$, the~$\cR_\ast$ terms vanish, and the estimate~\eqref{eq:hE-half} follows from~\eqref{eq:expr-hE-final}, after the elementary computation~$\sigma^{-1} x - \sigma^{-1}\infty = -1/(c_\sigma^2(x-\omega))$.
\end{proof}

\begin{remark}
  When~$\phi$ is cuspidal, a change of variables yields
  \begin{equation}
    \cE_{\sigma^{-1}\infty}(\phi^\sigma, \tfrac12) = (i\sign(c_\sigma))^{-k} \cE_\omega(\phi, \tfrac12).\label{eq:cusp-E-sigma-omega}
  \end{equation}
  This identity holds also when~$\phi$ is non-cuspidal, but this requires a more careful treatment of regularized integrals. We will not make use of~\eqref{eq:cusp-E-sigma-omega} in either case here.
\end{remark}

\subsection{Uniform Hölder estimates on~$\hE$}

\begin{lemma}\label{lem:hE-gamman-holder}
  Let~$p, \omega$ be as in Proposition~\ref{prop:hE-gamman-estim}. For any~$\eps>0$, there exists~$n_0\geq 1$ such that for~$x, x'\in \R\setminus[\omega-\eps, \omega+\eps]$ and~$n\geq n_0$, we have
  \begin{equation}
    \abs{\hE_{p^n}(x, \tfrac12) - \hE_{p^n}(x', \tfrac12)} \ll_{\phi, \eps, p} (1 + \abs{x} + \abs{x'})^{2\eps} n^\eps \abs{x-x'}^{1-\rho_\phi-\eps}. \label{eq:diff-hgamman}
  \end{equation}
\end{lemma}

\begin{proof}
  To simplify the notation in the following computation, we denote temporarily~$\hE_\ast(x) = \hE_\ast(x, \frac12)$. Since~$\omega\neq \infty$ and~$\omega$ is the unique fixed point of~$p$, we have~$\omega \neq p^{-j}\infty$ for all~$j\geq 0$. Since~$\abs{x-\omega}\geq \eps$, there exists~$n_0>0$ such that $\abs{x - p^{-n}\infty} \gg 1$ uniformly in~$n\geq n_0$ and~$x$, and also that $\sign(p^n x - p^{-j}\infty)$ is independent of~$n\geq n_0$,~$j\geq 1$, and~$x$. This implies that~$u_{p^j}(p^n x)$ is constant on~$\R\setminus[\omega-\eps, \omega+\eps]$ for all~$j\geq 0$ and $n\geq n_0$.

  By the period relation~\eqref{ple}, we write
  $$ \hE_{p^n}(x) = u_{p^{n-n_0}}(p^{n_0}x) \hE_{p^{n_0}}(x) + \sum_{\ell=n_0}^{n-1} u_{p^{n-1-\ell}}(p^{\ell+1} x) \hE_p(p^\ell x) $$
  and likewise for~$x'$.
  Taking differences and using the fact that the various~$u_{p^j}$ factors have the same value for~$x$ and~$x'$ by our choice of~$n_0$, we deduce
  \begin{equation}
    \abs{\hE_{p^n}(x) - \hE_{p^n}(x')} \ll \abs{\hE_{p^{n_0}}(x) - \hE_{p^{n_0}}(x')} + \sum_{\ell = n_0}^{n-1} \abs{\hE_p(p^\ell x) - \hE_p(p^\ell x')}.\label{eq:hgamman-sumhgamma}
  \end{equation}
  By our choice of~$n_0$, we have~$\abs{x- p^{-\ell}\infty} \gg 1$ uniformly for~$\ell\geq n_0$, and likewise for~$x'$, therefore
  $$ \abs{p^\ell x - p^\ell x'} \ll \ell^{-2} \abs{x - x'}. $$
  By Lemma~\ref{lem:hE-at-half}, we deduce the bounds
  \begin{align*}
    \abs{\hE_{p^n}(x) - \hE_{p^n}(x')} &{} \ll (1 + \abs{x} + \abs{x'})^{2\eps} \abs{x - x'}^{1-\rho_\phi-\eps}, \qquad (n = n_0) \\
    \abs{\hE_p(p^\ell x) - \hE_p(p^\ell x')} &{} \ll \abs{p^\ell x - p^{\ell} x'}^{1 - \rho_\phi-\eps} \ll \ell^{-2+2\rho_\phi+2\eps} \abs{x-x'}^{1 - \rho_\phi-\eps},
  \end{align*}
  uniformly in~$\ell$, $x$ and $x'$. The desired bound follows by picking~$\eps$ small and inserting in~\eqref{eq:hgamman-sumhgamma}.
\end{proof}

\subsection{Translation into estimates on the period function of~$L(\phi, s,x)$}

We now transfer our estimates from~$\cE$ to the additive twists~\eqref{eq:def-Lphixs}.
For all~$\gamma \in \Gamma$, $x\in C(\Gamma) \setminus\{\infty, \gamma^{-1}\infty\}$ and~$s\in \cD_\phi$, we denote
\begin{equation}
  \hL_\gamma(x, s) := \abs{j(\gamma, x)}^{2s-1} L(\phi, s,\gamma x) - u_\gamma(x) L(\phi,s,x). \label{eq:def-hL}
\end{equation}

\begin{lemma}\label{lem:hL-gamma}
  For each~$\gamma\in \Gamma$, the map~$(x, s) \mapsto \hL_\gamma(x, s)$ extends to a map on~$(\R\setminus\{\gamma^{-1}\infty\})\times \{\Re(s)>0\}$, which is continuous in~$x$ for each~$s$, and holomorphic in~$s$ for each~$x$.

  Given~$\eps>0$, $\gamma \in \Gamma$, $x\in \R$ with~$\abs{x-\gamma^{-1}\infty}\geq \eps$, and~$s$ with~$\eps\leq \Re(s)\ll1$, we have
  \begin{align} \label{eq:hL-rough_o}
    \hL_\gamma(x, s) \ll_{\gamma} \frac{\e^{\frac\pi2\abs{\Im(s)}}}{(1+|s|)^{\Re(s)+\Re(s_\phi)-2}}\cdot \big(1 + \abs{x-\gamma^{-1}\infty}^{\max(\Re(s) - \frac12 + \rho_\phi + \eps, 2\Re(s) - 1)}\big).
  \end{align}
  The map~$x\mapsto \hL_\gamma(x, \frac12)$ satisfies the Hölder estimate~\eqref{eq:holder-hE}, with $\hL$ in place of $\hE$.

  Moreover, letting for all~$\Re\tau>0$,
  \begin{align}\label{dfgk}
    G_k(\tau, x) ={}& \abs{x}^{\tau} \frac{(\frac{1+i\sign(x)}{\sqrt{2}})^k \Gamma(\tau)}{ 2^\tau\Gamma(\frac k2 + \tau)}\e^{-\sign(x)\frac{\pi i}2 (\tau - k/2)}, \qquad (x\neq 0), \\
    G_k(\tau, 0) :={}& 0 \notag,
  \end{align}
  we have the following estimates for~$\abs{x-\gamma^{-1}\infty} \geq \eps$:
  \begin{itemize}
    \item When~$\phi$ is cuspidal or when~$s_\phi \neq 1/2$,
    \begin{equation}
      \label{eq:estim-hL-largex}
      \begin{aligned}
        \hL_\gamma(x, \tfrac12) = u_\gamma(x)\big( & A_\infty(\phi) G_k(s_\phi, x) + B_\infty(\phi) G_k(1-s_\phi, x) \big) + O_{\gamma}(1),
      \end{aligned}
    \end{equation}
    \item When~$s_\phi = 1/2$,
    \begin{equation}
      \label{eq:estim-hL-largex-s1/2}
      \hL_\gamma(x, \tfrac12) =
      \begin{cases}
        u_\gamma(x) G_k(\tfrac12, x) A_\infty(\phi) + O_{\gamma}(1), & (B_\infty(\phi) = 0), \\
        u_\gamma(x) G_k(\tfrac12, x) B_\infty(\phi) \log\abs{x} + O_{\gamma}(1+\abs{x}^{1/2}), & (B_\infty(\phi) \neq 0).
      \end{cases}
    \end{equation}
  \end{itemize}
\end{lemma}

\begin{lemma}\label{lem:hL-parab}
  Let~$p, \omega, \sigma,m_0, \eps, x$ and~$s$ be as in Proposition~\ref{prop:hE-gamman-estim}.
  Then we have
  \begin{equation}
    \begin{aligned}
      \hL_{p^n}(x, s) \ll_{p} \frac{\e^{\frac{\pi }2|\Im(s)|}}{(1+|s|)^{\Re(s)+\Re(s_\phi)-2}}\cdot
      \Big(&(1+\abs{x})^{2\Re(s)-1} \big\{1+ n^{2\Re(s)-1} + n^{\Re(s)-\frac12+\rho_\phi+\eps}\big\} \\
      & {}+ 1 + (1+\abs{x})^{\Re(s)-\frac12+\rho_\phi+\eps} \Big).
    \end{aligned}\label{eq:hL-rough2}
  \end{equation}
  Furthermore, the Hölder estimate~\eqref{eq:diff-hgamman} holds also for~$\hL$ for~$n\geq n_0(\eps, p)$.

  Moreover we have the following estimates:
  \begin{itemize}
    \item When~$\phi$ is cuspidal or~$s_\phi \neq 1/2$, we have
    \begin{equation}
      \begin{aligned}
        \hL_{p^n}(x, \tfrac12) = {}& \sign(c_\sigma m_0n)^{-k}(A_\omega(\phi) G_k(s_\phi, m_0n) + B_\omega(\phi) G_k(1-s_\phi, m_0n)) \\
        &{} + u_{p^n}(x)\big( A_\infty(\phi) G_k(s_\phi, x) + B_\infty(\phi) G_k(1-s_\phi, x) \big) \\
        &{} + O_{p}(1).
      \end{aligned}
      \label{eq:estim-hL-parab}
    \end{equation}
    \item When~$s_\phi = 1/2$, define for all~$\cb\in C(\Gamma)$ and~$x\in \R_{\neq 0}$
    \begin{equation}
      D_\cb(\phi, x) := \begin{cases} A_\cb(\phi) G_k(\tfrac12, x), & (B_\cb(\phi) = 0), \\
        B_\cb(\phi) G_k(\tfrac12, x) \log\abs{x}, & (B_\cb(\phi)\neq 0). \end{cases}\label{eq:D-mainterm}
    \end{equation}
    Then
    \begin{equation}
      \begin{split}
        \hL_{p^n}(x, \tfrac12) &= \sign(c_\sigma m_0n)^{-k} D_\omega(\phi, m_0 n) + u_{p^n}(x) D_\infty(\phi, x)\\
        &\quad + \begin{cases} O_p(1), & (B_\infty(\phi) = B_\omega(\phi) = 0), \\
          O_p(\abs{x}^{1/2} + \abs{n}^{1/2}) & (\text{otherwise}). \end{cases}
        \label{eq:estim-hL-parab-s1/2}
      \end{split}
    \end{equation}
  \end{itemize}
\end{lemma}

We note in passing that for all~$x\neq 0$,
\begin{equation}
  \Big(\frac{G_k(\tfrac12, x)}{G_k(\tfrac12, -x)}\Big)^2 = -1.\label{eq:hparab-MT-flip}
\end{equation}

\begin{proof}[Proofs of Lemmas~\ref{lem:hL-gamma} and \ref{lem:hL-parab}]
  Assume first that if~$s_\phi\not\equiv k/2\pmod{1}$ then~$k\in \{0, 1\}$.
  By~\cite[eqs.~(5.11) \&~(5.16)]{DrappeauNordentoft2022}, we have
  \begin{align}\label{EtoL}
    L(\phi, s, x) = \begin{cases}
      \Delta_{it_\phi}(\frac k2, s)^{-1} ( g_\phi^+(s) \cE_x(\phi, s) + g_\psi^+(s) \cE_x(\psi, s) ) & (s_\phi \not\equiv k/2\pmod{1}), \\
      \Xi_\phi(s) \cE_x(\phi, s) & (y^{-k/2} \phi\text{ holomorphic}),
    \end{cases}
  \end{align}
  where~$\psi$ is the weight~$k+2$ automorphic form obtained from~$\phi$ by the weight-raising operator~$R_k$, and~$\Xi_\phi(s)$, $g_\phi^+(s)$, and $g_\psi^+(s)$ are described in the proof of~\cite[Lemma~5.4]{DrappeauNordentoft2022} (there~$\Xi_\phi(s)$ is denoted~$c_f(s)$), see in particular eqs.~(5.18) and~(5.21) \emph{ibid.}\ for the values at~$s=1/2$ and Proposition~5.1 and eq.~(5.6) for the definitions of $\Delta_\beta$ and $\Omega_\beta$.
  Recall that either~$\phi$ is cuspidal or~$\Re(s_\phi) = 1/2$, so that by~\cite[Corollary~4.4]{DukeEtAl2002}, we have~$\Re(s_\phi) < 1$ whenever~$s_\phi\not\equiv k/2\pmod{1}$\footnote{In~\cite{DukeEtAl2002} the statement restricts to~$\Gamma$ being a congruence subgroup, but the implication we need here follows in general by the same argument.}.
  Using~\cite[(9.111)]{GradshteynRyzhik2007} and Stirling's formula, we deduce that for $\eps<\Re(s)\ll1$
  $$\Delta_{it_\phi}(\tfrac k2, s)^{-1} g_\phi^+(s)\ll |\Gamma(s+s_\phi-1/2)|^{-1}\ll \e^{\frac\pi2 \abs{\Im(s)}} (1+|s|)^{1-\Re(s + s_\phi)}$$
  and the same bounds hold for $g_\psi^+$ in place of $g_\phi^{+}$ and, by~\cite[eq.~(5.20)]{DrappeauNordentoft2022}, for $\Xi_\phi(s)$.

  Thus, the first parts of Lemma~\ref{lem:hL-gamma} and Lemma~\ref{lem:hL-parab} follow immediately by linearity from Propositions~\ref{prop:hE-upperbound-alls} and~\ref{prop:hE-gamman-estim}, and from Lemmas~\ref{lem:hE-at-half} and~\ref{lem:hE-gamman-holder}.

  It remains to prove~\eqref{eq:estim-hL-largex}--\eqref{eq:estim-hL-largex-s1/2} and~\eqref{eq:estim-hL-parab}--\eqref{eq:estim-hL-parab-s1/2}. When~$\phi$ is cuspidal, the main terms on the right-hand side vanish, and the result follows directly from~\eqref{eq:estim-hEgamma-half} and~\eqref{eq:hE-half} by linearity.
  Suppose then that~$\phi$ is non-cuspidal, in particular~$\Re(s_\phi)=1/2$ by hypothesis, and assume first~$s_\phi \neq 1/2$.
  This implies~$s_\phi\not\equiv k/2\pmod{1}$ and thus, as we temporarily assumed, $k\in\{0, 1\}$.
  From the values in~\cite[Lemma~5.4]{DrappeauNordentoft2022}, we get
  \begin{align*}
    g_\phi^+(\tfrac12) ={}& \frac{\pi^{3/2}2^{1-k/2}}{\sin(\pi s_\phi)} \frac1{\Gamma(\frac k4 + \frac{s_\phi}2) \Gamma(\frac k4 + \frac{1-s_\phi}2)}, \\
    g_\psi^+(\tfrac12) ={}& \frac{-\pi^{3/2} 2^{-k/2}}{\sin(\pi s_\phi)} \frac1{\Gamma(\frac12 + \frac k4 + \frac{s_\phi}2) \Gamma(\frac12 + \frac k4 + \frac{1-s_\phi}2) }.
  \end{align*}
  Furthermore, by~\cite[eq.~(4.29)]{DukeEtAl2002},
  \begin{align*}
    A_\infty(\psi) = (\tfrac k2 + s_\phi) A_\infty(\phi), \quad
    B_\infty(\psi) = (\tfrac k2 + 1 - s_\phi) B_\infty(\phi).
  \end{align*}
  Thus, by~\eqref{eq:estim-hEgamma-half}, inserting the definition~\eqref{eq:def-Upsilon} and using the reflection and duplication formula for the Gamma function, we obtain~\eqref{eq:estim-hL-largex}
  with the claimed value of~$G_k(\tau, x)$.

  Assume next that~$s_\phi = 1/2$ and~$k=0$. Then by similar computations, using the values~\cite[eq.~(4.18)]{DrappeauNordentoft2022}, we obtain
  \begin{align*}
    \hL_\gamma(x, \tfrac12) = u_\gamma(x) \tfrac{1-i\eta}2 \big( &\{ A_\infty(\phi) - (\log 2 + i \eta \tfrac \pi2) B_\infty(\phi)\} \abs{\tilde x}^{1/2} + B_\infty(\phi) \abs{\tilde x}^{1/2} \log\abs{\tilde x}\big) + O(1),
  \end{align*}
  with $\tilde x:=x-\gamma^{-1}\infty$ and~$\eta = \sign(\tilde x)$.
  Assume, finally, that~$s_\phi = 1/2$ and~$k=1$. Then~$s_\phi = k/2$, so that by our assumptions~$y^{-k/2}\phi(z)$ is a holomorphic form; only its constant term contributes to~$\phi_\cb(y)$, in particular~$B_\cb = 0$ for every~$\cb$.
  Thus from~$\Xi_\phi(\frac12) = (2\pi)^{-1/2}$ we get
  \begin{align*}
    \hL_\gamma(x, \tfrac12) = u_\gamma(x) \sqrt{\pi}\tfrac{ 1+ i \eta}2 A_\infty(\phi) \abs{\tilde x}^{1/2} + O(1).
  \end{align*}
  This proves the asymptotic estimate~\eqref{eq:estim-hL-largex-s1/2}.

  The estimate~\eqref{eq:estim-hL-parab} is proved along a similar argument, using the expansion~\eqref{eq:hE-half} instead of~\eqref{eq:estim-hEgamma-half}. We check that for~$s_\phi \neq 1/2$ and~$\phi$ non-cuspidal, the corresponding contribution of~$\cM_{m_0n - 1/(c_\sigma^2(x-\omega))}(\dotsc)$ is
  $$ (\eta \sign c_{\sigma})^k (A_\omega(\phi) G_k(s_\phi, m_0 n) + B_\omega(\phi) G_k(1-s_\phi, m_0 n) ) + O(1), $$
  where now~$\eta = \sign(m_0 n)$, since~$\chi(\gamma) = 1$ unless~$\omega$ is non-singular in which case~$A_\omega(\phi) = B_\omega(\phi) = 0$.
  Thus,~\eqref{eq:estim-hL-parab} follows. The same computation at the terms involving~$\omega$ gives~\eqref{eq:estim-hL-parab-s1/2} when~$s_\phi = 1/2$.

  It remains to treat the case~$s_\phi \not\equiv k/2\pmod{1}$ and~$k\not\in\{0, 1\}$. If~$k<0$, we replace~$\phi$ by its weight-raised form~$\psi = R_k \phi$: by~\cite[eqs.~(4.25)--(4.29)]{DukeEtAl2002} this changes~$a(n)$ into its opposite, and we check that it also turns~$A_\cb(\phi) G_k(s_\phi, \ast)$, $B_\cb(\phi)G_k(1-s_\phi, \ast)$ (if $s_\phi \neq 1/2$) and~$D_\cb(\phi, x)$ (if~$s_\phi = 1/2$) into their opposites as well. Thus proving Lemmas~\ref{lem:hL-gamma} and~\ref{lem:hL-parab} for~$\psi$ implies the same statement for~$\phi$. Thus, since they were proved for~$k\in \{0, 1\}$, they hold for all~$k\leq 0$. When~$k>1$, we carry out the same argument using the weight-lowered form~$\psi = \Lambda_k\phi$, with the factor~$-1$ replaced by~$-(s_\phi - \frac k2)(1-s_\phi-\frac k2)$ (see~\cite[eq.~(4.27)]{DukeEtAl2002}).
\end{proof}

\section{Estimation of a Fourier integral}\label{sec:fourier}

The case of non-cuspidal forms (Theorem~\ref{th:main-eis}) will require the second order estimation of a Fourier integral, which is the content of the following lemma.

\begin{lemma}\label{lem:fourier}
  Let $x_0\in\R$. Let~$(f_n)_n$ and~$(W_n)_n$ be sequences of real-valued measurable maps on~$[0, 1]$ indexed by~$n\geq 1$, satisfying the following estimates:
  \begin{itemize}
    \item We have~$0\leq W_{ n}(x) \ll n^{-2}$, with a uniform implied constant.
    \item For some constant~$\alpha \geq 0$, we have~$\int_0^1 W_{n}(x)\df x = \alpha n^{-2} + O(n^{-3})$.
    \item For some constants~$\beta_{\pm} \geq 0$, we have~$\sum_n W_{n}(x) = \beta_\pm+ O(|x-x_0|)$ for $\sign(x-x_0)=\pm1$.
  \end{itemize}
  Assume additionally one of the following hypotheses:
  \begin{enumerate}[label=(\roman*)]
    \item\label{it:fourier-hyp-half-simple} There exist real numbers~$a, d_{\pm} \ll 1$ such that for $\sign(x-x_0)=\pm1$ one has
    $$ f_{n}(x) = a n^{1/2} + d_{\pm} |x-x_0|^{-1/2} + O(1); $$
    \item\label{it:fourier-hyp-nothalf} There exist a finite set $J$ and real numbers~$a_{ j}, d_{j,\pm} \ll 1$, $\theta_{ j}, \theta'_{j}\in \R$ and constants~$\tau_{ j}\in\R_{\neq 0}$ such that for $\sign(x-x_0)=\pm1$ one has
    $$ f_{n}(x) = \sum_{j\in J} a_{ j} n^{1/2}\cos(\tau_{ j}\log n+\theta_{ j}) + d_{ j,\pm} |x-x_0|^{-1/2}\cos(\tau_{j} \log |x-x_0| + \theta'_{ j}) + O(1); $$
    \item\label{it:fourier-hyp-half-double} There exist real numbers~$a, d_{\pm}\ll 1$ such that for $\sign(x-x_0)=\pm1$ one has
    $$ f_{n}(x) = a n^{1/2}\log n + d_{\pm} |x-x_0|^{-1/2} \log(|x-x_0|^{-1}) + O(n^{1/2}+|x-x_0|^{-1/2}). $$
  \end{enumerate}
  Then there exist~$\mu\in \R$ and~$\sigma\geq 0$ such that uniformly for~$t$ in a neighbourhood of~$0$, we have
  $$ \sum_{n\geq 1} \int_0^1 W_{ n}(x) (\e^{i t f_{ n}(x)} - 1) \df x = i t \mu - \frac\sigma2 t^2 \log(1/\abs{t})^\xi + O(t^2 (\log 1/\abs{t})^{\xi - \rho}), $$
  where, according to each case, we have
  \begin{equation*}\def\arraystretch{1.3}
    \begin{array}{|r|c|c|c|}
      \hline
      & \xi & \rho & \sigma \\
      \hline
      \text{(i)} & 1 & 1 & 2 \alpha \abs{a}^2 + 2\sum_{x_0\pm0\in[0,1]} \beta_\pm \abs{d_\pm}^2 \\
      \text{(ii)} & 1 & 1 & \sum_{v\in T} \big( \alpha \abs{A_v}^2 + \sum_{x_0\pm0\in[0,1]} \beta_\pm \abs{D_{v,\pm}}^2 \big)\\
      \text{(iii)} & 3 & 1-\eps & \frac83 \alpha \abs{a}^2 + \frac83 \sum_{x_0\pm0\in[0,1]} \beta_\pm \abs{d_\pm}^2 \\
      \hline
    \end{array}
  \end{equation*}
  where the sum over $\pm$ selects the sign(s) for which $[0,1]$ contains a right (if $\pm=+$) or left (if $\pm=-$) neighbourhood of $x_0$.
    Here, in case~(ii), we denoted~$T := \{\abs{\tau_j}, j\in J\}$, and for~$v\in T$,
    $$ A_v := \sum_{\tau_j = v} a_j \e^{i\theta_j} + \sum_{\tau_j = -v} a_j \e^{-i\theta_j}, \qquad
    D_{v, \pm} := \sum_{\tau_j = v} d_{j, \pm} \e^{i \theta'_j} + \sum_{\tau_j = -v} d_{j, \pm} \e^{-i \theta'_j}. $$
  The implied constant depends at most on~$\alpha$, $\beta_\pm$,~$\tau_{ j}$,~$\eps$ and on the implied constants in the corresponding hypothesis.
\end{lemma}

\begin{proof}
  For simplicity of notation we consider the case $x_0=0$ only; the general case follows in the same way. We denote~$\beta := \beta_+$ and~$d := d_+$.
  Consider first the case~\ref{it:fourier-hyp-half-simple}. By conjugation, we may assume that~$t>0$. The hypotheses imply the convergence of~$\mu := \sum_n \int_0^1 W_n(x) f_n(x) \df x$, and we are left to prove
  $$ \sum_{n\geq 1} \int_0^1 W_{ n}(x) (\e^{i t f_n(x)} - 1) \df x = i t \mu - \sigma \frac{t^2 \abs{\log t}}{2} + O(t^2) $$
  where~$\sigma = 2(\alpha \abs{a}^2 + \beta \abs{d}^2)$. Write~$f_n(x) = a n^{1/2} + d x^{-1/2} + g_n(x)$, where by hypothesis~$g_n(x)$ is uniformly bounded. Let~$\delta_j = \delta_j(n, x)$ be defined as
  $$ \delta_1 := \e^{i a t n^{1/2}}-1, \qquad \delta_2 := \e^{i d t x^{-1/2}}-1, \qquad \delta_3 := \e^{i t g_n(x)} - 1. $$
  Then we write
  $$ \e^{i t f_n(x)} - 1 = \delta_1 \delta_2 \delta_3 + \sum_{i<j} \delta_i \delta_j + \sum_i \delta_i. $$
  Since~$\delta_i \ll 1$ for each~$i$, and moreover~$\delta_i \delta_j \ll t^2 n^{1/2} x^{-1/2}$ whenever~$i\neq j$ (here we use~$a, d\ll 1$), we deduce that
  $$ \e^{i t f_n(x)} - 1 = O(t^2 n^{1/2} x^{-1/2}) + \sum_i \delta_i. $$
  Upon integrating over~$x$ and summing over~$n$, the first term contributes~$O(t^2)$, and it suffices to prove
  \begin{align}
    \sum_n \int_0^1 W_n(x) \big(\e^{i a t n^{1/2}} - 1 - i a t n^{1/2}\big) \df x ={}& -\alpha \abs{a}^2 t^2\abs{\log t} + O(t^2), \label{eq:fourier-int-n} \\
    \sum_n \int_0^1 W_n(x) \big(\e^{i d t x^{-1/2}} - 1 - i d t x^{-1/2}\big) \df x ={}& -\beta \abs{d}^2 t^2\abs{\log t} + O(t^2), \label{eq:fourier-int-x} \\
    \sum_n \int_0^1 W_n(x) \big(\e^{i t g_n(x)} - 1 - i t g_n(x)\big) \df x ={}& O(t^2). \label{eq:fourier-int-bdd}
  \end{align}
  The estimate~\eqref{eq:fourier-int-bdd} is clear by Taylor expansion at order~$2$ of~$\e^{i t g_n(x)}$.
  Regarding~\eqref{eq:fourier-int-n}-\eqref{eq:fourier-int-x}, by conjugation and rescaling~$t$ we reduce to the case~$a=d=1$ (here again we use the fact that~$a, d \ll 1$).

  We apply~\cite[Proposition~2.2]{BettinDrappeau2022a} (using the variant described in~\cite[Corollary~3.3]{BettinDrappeau2022a}) with~$\phi(x) = x^{-1/2}$ and~$\df\mu(x)=\sum_{n} \delta_{x=1/n} \int_0^1W_n(y)\df y$, with $\delta_{x=1/n}$ being the Dirac mass at $1/n$, and $\df\mu(x)=\sum_nW_n(x) \df x$ respectively, and parameters given by~$(\alpha, \xi, \rho) \gets (2, 1, 1)$.

  In case~\ref{it:fourier-hyp-half-double}, the same argument goes through by applying~\cite[Proposition~2.2]{BettinDrappeau2022a} with~$\phi(x) = x^{-1/2}\log(1/x)$, with parameters~$(\alpha, \xi, \rho) \gets (2, 3, 1-\eps)$.

  Finally, in case~\ref{it:fourier-hyp-nothalf}, a similar argument reduces the problem to the analysis of
    \begin{equation}
      \sum_n \int_0^1 W_n(x) (\e^{i t n^{1/2} \sum_j a_j \cos(\tau_j \log n + \theta_j)} - 1 - i t n^{1/2} \sum_j a_j \cos(\tau_j \log n + \theta_j))\df x,\label{eq:fourier-int-n-cos}
    \end{equation}
  and corresponding analogues of~\eqref{eq:fourier-int-x} and \eqref{eq:fourier-int-bdd}. The sum over~$j$ makes it less obvious to apply~\cite[Proposition~2.2]{BettinDrappeau2022a}, but since the underlying poles are simple in this case, an easier argument is as follows. Set~$N = t^{-2}$. For those~$n>N$ we bound the summand as~$O(n^{-2} + t n^{-3/2})$, which gives an error~$O(t^2)$. For those~$n\leq N$, expanding the exponential to degree~$2$ yields that~\eqref{eq:fourier-int-n-cos} equals
    \begin{align*}
      & O(t^2) - \frac{\alpha t^2}2 \sum_{n\leq N} \frac1n \Big(\sum_j a_j \cos(\tau_j \log n + \theta_j)\Big)^2 \\
      ={}& O(t^2) - \frac{\alpha t^2}2 \sum_{j_1, j_2} a_{j_1} a_{j_2} \sum_{n\leq N} \frac1n \cos(\tau_{j_1} \log n + \theta_{j_1}) \cos(\tau_{j_2} \log n + \theta_{j_2}).
    \end{align*}
    The contribution of those~$j_1, j_2$ for which~$\abs{\tau_{j_1}} \neq \abs{\tau_{j_2}}$ is again~$O(t^2)$ since~$\sum_{n\leq N} n^{-1-i\tau} = O(1)$ for fixed~$\tau\neq 0$. The rest contributes
    $$ O(t^2) - \frac{\alpha t^2 \abs{\log t}}2 \Re\Big( \sum_{\tau_{j_1} = \tau_{j_2}} a_{j_1} a_{j_2} \e^{i(\theta_{j_1} - \theta_{j_2})}
    + \sum_{\tau_{j_1} = -\tau_{j_2}} a_{j_1} a_{j_2} \e^{i(\theta_{j_1} + \theta_{j_2})}\Big), $$
    and the real part is seen to coincide with~$\sum_{v\in T} \abs{A_v}^2$.

\end{proof}

\begin{remark}\label{rmk:fourier-curves}
  Taking $W_{n}=0$ for all sufficiently large $n$, we see that the lemma holds, with~$\alpha:=0$, when~$(W_{n})_n$, $(f_n)_n$ are finite sequences. Moreover, with obvious modifications, the lemma remains valid when the integration is taken over any $C^2$ curve with bounded arc-length rather than over $[0,1]$. This follows by parametrizing the curve and applying the lemma to the resulting integral.
\end{remark}

\section{Translation to the disk model}\label{sec:disk-model}

To use the results in~\cite{BettinEtAl}, it will be more convenient to use the disk model for the hyperbolic plane. Define
\begin{equation}
  p_\D : \overline{\H} \to \overline{\D}, \qquad p_\D(z) = \frac{z-i}{z+i}\label{eq:def-pD}
\end{equation}
which sends~$\dH$ to~$\dD$. Its inverse is given by
\begin{equation}
  p_\H(z) = i \frac{1+z}{1-z}.\label{eq:def-pH}
\end{equation}
With a slight abuse of notation, we indicate by $p_\D$ and $p_{\H}$ the matrices in $\SL(2,\C)$ corresponding to these Möbius transformations. For~$z\in \C\cup\{\infty\}$, we also denote
$$ [z]_\D = p_\D z, \quad [z]_\H = p_\H z, $$
while for~$\gamma\in \SL(2, \R)$ and~$g\in \SU(1, 1)$, we denote
$$ [\gamma]_\D = p_\D \gamma p_\H, \qquad [g]_\H = p_\H g p_\D. $$
We use the same notation for sets of matrices.
Let~$\bar{G} := [\Gamma]_\D = p_\D \Gamma p_\H$ be the image of~$\Gamma$ in the disk model. The group~$\bar{G}\subset \SU(1,1)$ is discrete, finitely generated, and has finite covolume, since the same is true of~$\Gamma$.
Since~$\Gamma$ has a cusp at~$\infty$ of width $1$, we deduce that~$\bar{G}$ has a cusp at~$1$ of width $1$.

\subsection{Construction of the subgroup}

In order to apply the results of Section~\ref{sec:top-irred}, it will be convenient to work with a subgroup of sufficiently large rank.

\begin{lemma}\label{lem:free-subgroup}
  The group~$\bar{G}$ admits a subgroup~$G$ of finite index which is free and of rank at least~$6$.
\end{lemma}

\begin{proof}
  By \emph{e.g.}~\cite[Lemma~8]{Selberg1960}, $\bar{G}$ admits a torsion-free subgroup $G$ of finite index\footnote{The constructions in~\cite{Selberg1960,EdmondsEtAl1982} yield a normal subgroup, but we will not use this property.}, which is therefore free~\cite[p.~683]{Newman1985} of some rank~$s$ (see~\cite[p.~335]{EdmondsEtAl1982} for a short algebraic proof in our non-cocompact case).

  If~$s\geq 6$ then~$G$ meets our requirements. Otherwise, if~$G$ is freely generated by~$g_1, \dotsc, g_s$, say, we let~$G' = \ker(\pi)$ where~$\pi:G\to \Z/6\Z$ is the homomorphism given by~$\pi(g_i)=1$. Then~$G'$ is a subgroup of~$G$ of index~$6$ in~$G$. Therefore by~\cite[Corollary~2.9, Theorem~2.10]{MagnusEtAl2004}, it is free of rank~$\geq 6$.
\end{proof}

In the following we shall fix a $G$ satisfying the properties of Lemma~\ref{lem:free-subgroup}. Note that the group $G$ is discrete, of covolume~$\frac12[\bar{G}:G]\Vol(\bar{G}\backslash\D)<\infty$~\cite[Theorem~3.1.2]{Katok1992}. It inherits a cusp at~$1$ from~$\bar{G}$, and thus satisfies all the hypotheses from~\cite[Section~\refmain{sec:disk-model}]{BettinEtAl}.

Finally, for $G'\in\{G,\overline G\}$,~$\cb \in C(G')$ and~$Q>0$, we define an analogue of~\eqref{cqdfn} as
$$ C_\cb(G'; Q) := \{ x\in G' \cb, 0 < q_\cb^{G'}(x) \leq Q\} $$
where, following~\cite[eq.~\eqrefmain{eq:def-qx}]{BettinEtAl}, $q_\cb(x)$ is defined, for~$x=g' \cb$ with $g'\in G'$, as
$$ q_\cb^{G'}(x) := \abs{\Im(\alpha+\beta)}, \qquad g's'_\cb =\smat{ \alpha & \beta \\ \ast&\ast } $$
where $s'_\cb\in \SU(1,1)$ is a scaling matrix for $\cb$ in $G'$, which in this context means $[s'_\cb]_\H$ is a scaling matrix for $[\cb]_\H$ in~$[G']_\H$, in other words
$$ s'_\cb1 = \cb, \qquad G'_\cb = 
\begin{cases}
  s'_\cb\, \big\{ \eta^k \smat{1+ki/2&-ki/2\\ki/2&1-ki/2}\!\Big|\, k\in \Z\big\}\, {s'}_\cb^{-1}, & (G'=G), \\
  s'_\cb\, \big\{ \pm \smat{1+ki/2&-ki/2\\ki/2&1-ki/2}\!\Big|\, k\in \Z\big\}\, {s'}_\cb^{-1}, & (G'=\bar{G}).
\end{cases} $$
where~$\eta$ is the sign of the trace of a generator of~$G_\cb$.
  Note that~$q_\cb^{G'}(x)$ depends only on~$x$ (not on~$g'$), so that~$C_\cb(G'; Q)$ is well-defined.

A quick computation shows $q^G_\cb(g\cb) = \abs{c_{[g]_\H[s_\cb]_\H}}$; in particular, for any $\cb\in C(\Gamma)$ one has
\begin{align}\label{relationc}
  C_{\cb}(\Gamma; Q)=[C_{[\cb]_\D}(\bar{G}; Q)]_\H.
\end{align}

\begin{lemma}\label{lem:reduction-gamma}
  Let~$\cb\in C(G)$ and let $(\bar{g}_j)_j$ be a set of representatives of~$G\backslash \bar{G} / \bar{G}_\cb$. Then:
  \begin{itemize}
    \item For all~$j$, we have~$\bar{g}_j\cb \in C(G)$.
    \item Let~$n_j = \frac12[\bar{G}_{\cb}:\bar{g}_j^{\,-1}\,G_{\bar{g}_j\cb}\,\bar{g}_j]$.
    Then for all~$Q>0$, we have
    $$ C_\cb(\bar{G}; Q) = \bigcup_j C_{\bar{g}_j\cb}(G; Q\sqrt{n_j}) $$
    where the union is disjoint.
  \end{itemize}
\end{lemma}
\begin{proof}
  Let~$\cb_j := \bar{g}_j\cb$. Note that~$\cb_j \in C(\bar{G})$. To prove the first point, let~$\sigma \in G$ be parabolic and such that~$\sigma \cb = \cb$, and~$\tau = \bar{g}_j \sigma \bar{g}_j^{\,-1}$, so that~$\tau \cb_j = \cb_j$. Since~$[\bar{G}:G]<\infty$, we may find an integer~$n>0$ such that~$\tau^n \in G$. The equality~$\tau^n \cb_j = \cb_j$ and the fact that~$\tau^n$ is a parabolic element of~$G$ imply that~$\cb_j \in C(G)$.
  By definition of~$n_j$, we have
  $$ G_{\cb_j} = \bar{g}_j \bar{s}_\cb\{ (\pm 1)^{k/n_j} \smat{1+ki/2&-ki/2\\\ast&\ast}\!,\, k\in n_j\Z\} (\bar{g}_j \bar{s}_\cb)^{-1}, $$
  where~$\bar{s}_\cb$ is a scaling matrix for~$\cb$ as a cusp of~$\bar{G}$.
  We conjugate by~$t := \smat{(1+n_j)/2\sqrt{n_j} & (-1+n_j)/2\sqrt{n_j} \\ \ast & \ast}$, which gives the following scaling matrix for~$\cb_j$ as a cusp of~$G$,
  $$ s_{\cb_j} = \bar{g}_j \bar{s}_{\cb} t. $$
  Let now~$x = g \cb_j = g \bar{g}_j \cb$ for some~$g\in G$ and let~$\smat{\alpha' & \beta' \\ \ast & \ast}:=g \bar{g}_j \bar{s}_\cb$ and~$\smat{\alpha & \beta \\ \ast & \ast} := g \bar{g}_j \bar{s}_\cb t$. Clearly, $\smat{\alpha & \beta \\ \ast & \ast} = \smat{\alpha' & \beta' \\ \ast&\ast} t$, which implies $\alpha+\beta=(\alpha'+\beta')\sqrt{n_j}$. Thus,
  $$ q_{\cb_j}^G(x) = \sqrt{n_j} q^{\bar{G}}_\cb(x). $$
  Since we have a disjoint union~$\bar{G} = \cup_j G \bar{g}_j \bar{G}_\cb$, the conclusion of the lemma follows.
\end{proof}

\subsection{The multiplier system}\label{vtdf}

We also need to transfer the multiplier $u_\gamma(z)$ defined in~\eqref{multp} to the disk setting.
For all~$g\in G$ and~$x\in\overline{\D}\setminus\{1,g^{-1}1\}$, let
$$ \vth_g(x) := \chi'(g) j_\gamma(y)^{-k}, $$
where~$y = [x]_\H$,~$\gamma = [g]_\H$ and $\chi'(g):=\bar{\chi(\gamma)}$.
With this correspondence, we have
$$ \vth_g(x) = \bar{u_\gamma(y)} $$
and the period relation
\begin{align}\label{pere}
  \vth_{g_1g_2}(x) = \vth_{g_1}(g_2 x) \vth_{g_2}(x), \qquad (x\not\in\{1, g_2^{-1} 1, (g_1g_2)^{-1} 1\}).
\end{align}
which we will use repeatedly. Note that the set of values taken by~$\vth_g(x)$ forms a subgroup of the roots of unity, namely
$$ \Theta = \{\vth_g(x), g\in G, x\in C(G)\setminus\{1,g^{-1}1\}\} = \{\eta^k \chi(\gamma), (\gamma, \eta) \in [G]_\H \times\{\pm 1\}\}. $$

\section{Transfer operators}\label{tos}

\subsection{The inverse branches of the cuspidal acceleration}\label{sec:operator-cc}

We now summarize the set-up of~\cite{BettinEtAl}, which builds on~\cite{BowenSeries1979,Marchese2018}. We let~$R$ be a fundamental domain for~$G\backslash \D$ whose sides are geodesics. By~\cite[Proposition p.~15]{Tukia1973} we may pick~$R$ having all its vertices on~$\dD$ and, by translating the domain if necessary, we can assume $1$ is a vertex. Each side $s$ is paired with another side of~$R$, say~$\hat s$, by an element~$g_s\in G\setminus\{\Id\}$, in such a way that $ s = g_s \hat s $ and $g_{\hat s}=g_s^{-1}$; finally the group~$G$ is a free product of the elements~$g_s$.

We label the sides of~$R$ according to elements belonging to a finite set~$\Sigma$ of letters -- which we identify with $\Z/2m\Z$ for~$m = \rank(G)$ -- proceeding counterclockwise from $1$. Note that~$m\geq 6$ by Lemma~\ref{lem:free-subgroup}. For convenience of notation we import the additive structure of $\Z/2m\Z$ and in particular, we have that $a,b\in\Sigma$ are contiguous sides if and only if $a-b=\pm1$.

Given a word~$w = a_1 \dotsb a_\ell$ with~$a_j \in \Sigma$, we say that~$w$ is \emph{admissible} if~$a_j \neq \hat{a_{j+1}}$ for all~$j$. We denote by $\Sigma^*$ the set of admissible (finite) words in $\Sigma$. Given $w = a_1 \dotsb a_\ell\in\Sigma^*$, we let~$g_w := g_{a_1} \dotsb g_{a_\ell}$. We also let~$\hat w := {\hat{a_\ell}} \dotsb {\hat{a_1}}$, so that~$g_{\hat w} = g_w^{-1}$, and we denote by $h_w$ the Möbius transformation associated to $g_w$, that is $h_w(z):=g_w\cdot z$.

Given $\ell\geq1$, we say that $w= a_1 \dotsb a_\ell\in\Sigma^*$ is a \emph{cuspidal cycle} if for some \emph{direction} $\delta\in\{\pm1\}$ one has
\begin{equation}\label{ccy}
  a_{j+1}=\hat{a_j}+\delta\quad \forall j=1,\dots,\ell-1,\qquad a_{1}=\hat{a_\ell}+\delta,
\end{equation}
see the corresponding definition in~\cite[Section~3.2]{Marchese2018}. We say that~$w_1$ is a non-empty prefix of $w$ if $w_1=a_1\dotsb a_j$ with $1\leq j\leq \ell$, and say the cuspidal cycle $w$ is minimal if there is no proper non-empty prefix of $w$ which is cuspidal with the same direction~$\delta$.
Given $a\in\Sigma$, there are exactly two minimal cuspidal cycles with first letter $a$, one for each choice of the direction $\delta$.

Given~$a\in \Sigma$, let~$[a]$ and~$[a]^\circ$ be respectively the closed and open arcs in~$\dD$ delimited by the two endpoints of~$a$ and containing no other vertex of~$R$.
We have that $g_a(\dD\setminus [\hat a]^\circ)=[a]$. More generally, if $w = a_1 \dotsb a_\ell\in\Sigma^*$, then $g_w(\dD\setminus [\hat{a_\ell}]^\circ)$ is a sub-arc of $[a_1]$ which we denote by $[a_1,\dots,a_\ell]$ or~$[w]$. As the length of $w$ goes to infinity, the arc
$[a_1,\dots,a_\ell]$ shrinks and determines a unique limiting point indicated by $[a_1,a_2,\dots]$. Every element of $\dD$ can be written in this way, although not necessarily in a unique way (see~\cite{BowenSeries1979}). The Bowen--Series map $\cB$ is the map that sends $[a_1,a_2,\dots]$ to $[a_2,a_3,\dots]$. This is well-defined on~$\dD$ except on the set~$V$ of vertices of~$R$. Because of the presence of cusps, the map~$\cB$ is not uniformly expanding.
As in~\cite{BettinEtAl}, we consider two ``accelerations'' of $\cB$: the induced map $\cF_K$ introduced in~\cite{BowenSeries1979}, and the cuspidal acceleration $\cF_c:\dD\setminus V\to\dD$ defined in~\cite{BettinEtAl}, which is related to the cusp acceleration introduced in~\cite{ArtigianiEtAl2016}. We refer to~\cite[Section~\refmain{sec:bs:accel}]{BettinEtAl} for the precise definition, and recall first the classification of the inverse branches $\cH_c$ of $\cF_c$. These are Möbius transformations defined on suitable subsets of $\dD$ and can be divided into the following types:
\begin{enumerate}[label=(\roman*)]
  \item
  The maps~$h = h_w$ for~$w=a\in \Sigma$, with domain~$\Dom(h)$ consisting of the union of all arcs
  of the shape~$[b_1, b_2]$, $b_1 \neq \hat a$, which are not contiguous to~$[\hat a]$. This last condition means that either~$b_1 \neq \hat a \pm 1$, or else~$b_2 \neq \hat{b_1} + (b_1 - \hat a)$.
  \begin{center}
    \begin{tikzpicture}[scale=0.65]
      \draw (0, 0) circle (3);
      \interval{3}{0}{60} node at (3, 2) {$[\hat a]$}; %main
      \interval{3}{60}{120}; %next lv1
      \interval{3}{-90}{0}; %prev lv1
      \interval{3}{60}{80}; %next lv2
      \interval{3}{-30}{0}; %prev lv2
      \draw[line width=1.5px, blue] ({3*cos(80)}, {3*sin(80)}) arc (80:330:3) node at (-3.2, -3) {$\Dom(h_a)$};
      \draw[line width=1.2px] ({3*cos(0)}, {3*sin(0)}) arc (0:60:3);
    \end{tikzpicture}
  \end{center}

  \item
  For each triple $(w_0,w_1,n)$ with $w_0$ a minimal cuspidal cycle, $w_1$ a non-empty prefix of~$w_0$, and $n\in\Z_{\geq0}$ such that $n\geq1$ if $w_1$ is of length $1$, the map~$h=h_w$ where~$w=w_0^nw_1$. Its domain~$\Dom(h)$ consists of the union of all arcs~$[\hat{a_r} + \delta, b_2]$ which are not contiguous to~$[\hat{a_r}]$, where $a_r$ is the last letter of $w_1$ and~$\delta$ is the direction of~$w_0$.
  \begin{center}
    \begin{tikzpicture}[scale=0.65]
      \draw (0, 0) circle (3);
      \interval{3}{0}{60} node at (3, 2) {$[\hat{a_r}]$}; %main
      \interval{3}{60}{120}; %next lv1
      \interval{3}{-90}{0}; %prev lv1
      \interval{3}{60}{80}; %next lv2
      \interval{3}{-30}{0}; %prev lv2
      \draw[line width=1.5px, blue] ({3*cos(80)}, {3*sin(80)}) arc (80:120:3) node at (-.2, 3.6) {$\Dom(h_w) \quad (\text{if }\delta=+1)$};
      \draw[line width=1.5px, red] ({3*cos(-90)}, {3*sin(-90)}) arc (270:330:3) node at (2, -3.8) {$\Dom(h_w) \quad (\text{if }\delta=-1)$};
      \draw[line width=1.2px] ({3*cos(0)}, {3*sin(0)}) arc (0:60:3);
    \end{tikzpicture}
  \end{center}
\end{enumerate}

For simplicity of exposition, we shall often also write the index of the inverse branches of the first kind as $w=w_0^nw_1$ with $n=0$, $w_1=a$ and $w_0$ either of the two minimal cuspidal cycles starting with the letter~$a$.
By~\cite[Corollary~\refmain{cor:parabolic}]{BettinEtAl}, for~$h\in \cH_c$ and~$x\in\Dom(h)$, we have~$|h'(x)| \asymp 1/(n+1)^2$ where~$n$ is determined above.
In particular this is independent of~$x$, and for brevity, we denote~$\abs{h'}$ any quantity with
$$ \abs{h'} \asymp \abs{h'(x)}, \qquad (x\in \Dom(h)) $$
in all estimates where an order of magnitude is sought.

Notice that the set of possible $w_0$ and $w_1$ that appear in these decompositions is finite. We also remark that by~\cite[Lemma~\refmain{lem:parabolic1}]{BettinEtAl} the matrix $g_{w_0}$ is parabolic and generates the stabilizer of its fixed point $\omega=\omega_{w_0}= [w_0,w_0,w_0,\dots]$. The fixed points $\omega_{w_0}$ are vertices of $R$ and, as $w_0$ varies, every vertex of~$R$ arises in this way.
For each vertex $\omega$ of $R$, we let $s_\omega = [\sigma_{[\omega]_\H}]_\D$ be the scaling matrix for~$\omega$ as a cusp of~$\bar{G}$ which corresponds to~\eqref{eq:def-scaling-matrix}, and let $\upsilon_{w_0}\in\Z_{\neq 0}$, $\iota_{w_0}\in\{0,(\pm1)^k\}$ be defined as
\begin{align}\label{dfiu}
  [s_{\omega_{w_0}}^{-1}g_{w_0}s_{\omega_{w_0}}]_{\H}= \pm\smat{1 &\upsilon_{w_0}\\&1} \qquad\qquad
  \iota_{w_0}&:=\begin{cases}
    \sign(\upsilon_{w_0} c_{[s_{\omega}]_\H})^k &\omega\neq1,\\
    0&\omega=1.
  \end{cases}
\end{align}
We remark that $\omega_{w_0}=\omega_{\hat{w_0}}$ and
\begin{equation}
  \label{eq:upsilon-flip-hat}
  \upsilon_{w_0}=-\upsilon_{\hat{w_0}}.
\end{equation}

Furthermore, for any $h\in\cH_c$, we let $\im(h):=h(\Dom(h))$. For any~$h = h_w \in \cH_c$ with~$w = a_1 \dotsb a_\ell\in\Sigma^*$, we have that $\im(h_w)$ is contained in the arc~$[a_1]^\circ$, and is bounded away from its endpoints by a constant (depending on $w$). In particular, for any $w=w_0^nw_1$ as above we have
\begin{align}\label{di0}
  |g_w x-1|\gg_n 1,\ |g_{w_1}x-1|\gg1, \ |g_{w_1}x-\omega_{w_0}|\gg 1 \quad \forall x\in\Dom(h_w),
\end{align}
where the second and third bounds hold since $\Dom(h_w)\subseteq\Dom(h_{w_1})$. This also implies
\begin{align}\label{di1}
  |x-g_w^{-1}1|\gg_n 1,\ |x-g_{w_1}^{-1}1|\gg 1,\ |g_{w_1} x-g_{w_0}^{-n}1|\gg 1\quad \forall x\in\Dom(h_w).
\end{align}

\subsection{The operator~$\LL$}\label{opll}

Next we recall from~\cite{BettinEtAl} the definitions of the map~$\cF_K$, the set~$\cH$ of its inverse branches, and its transfer operator~$\LL$. In the present paper we will only need a few facts on the sets~$\cH$ and the operators~$\LL_s$.

Let
$$ K = \bigcup_{\substack{a, b \in \Sigma \\ |b-\hat a|>1}} [a, b]. $$
The set~$\cH$ is a set of smooth maps from subarcs of~$K$ to $\dD$, satisfying the following:
\begin{itemize}
  \item We have~$\min \abs{h'} \asymp \max\abs{h'}$, which justifies the notation~$\abs{h'}$ defined up to a constant depending on~$R$.
  \item For some~$\delta>0$, we have
  $$ \sum_{h\in \cH} \abs{h'}^{1-\delta} < \infty. $$
  \item Every~$h\in \cH$ decomposes as compositions of maps from~$\cH_c$, in the way described in~\cite[Lemma~\refmain{lem:decomp-Hc-HK}]{BettinEtAl}.
  \item The following unions are partitions up to measure zero sets:
  $$ \dD = \bigcup_{\substack{h\in \cH}} \im(h) = \bigcup_{h\in \cH_c} \im(h). $$
\end{itemize}

We can now quote from~\cite[eq.~\refmain{eq:def-twistedop-s}]{BettinEtAl} the following definition.
For all~$s\in \C$ with~$\Re(s) > 1 - \delta$, the operator~$\LL_s$ is defined on maps on~$\dD\times \Theta$ by
\begin{equation}
  \LL_s[f](y, \vth_1) = \ssum{h\in \cH \\ y \in \Dom(h)} \abs{h'(y)}^s f(h(y), \vth_1 \vth_h(y)), \qquad (y\in K, \vth_1\in \Theta), \label{eq:def-LLs}
\end{equation}
where for $h=h_w\in\cH$
\begin{align*}
  \vth_{h}(y):=\vth_{g_w}(y).
\end{align*}
We also denote
$$ \LL := \LL_1. $$
For all~$(x, \vth) \in K\times \Theta$, let
\begin{equation}
  \cF_K(x, \vth) := (\cF_K(x), \vth \vth_g(x)), \qquad (\cF_K(x) = gx, g\in G). \label{eq:def-FK-extended}
\end{equation}
By definition we have $\cF_K(h(x))=x$, and thus for any~$f_1, f_2 \in L^2(K\times \Theta)$,
\begin{equation}\label{cvl}
  \int_{K\times \Theta} \LL_s[f_1] f_2 = \int_{K\times \Theta} \LL[f_1 \cdot |\cF_K'|^{1-s}\cdot f_2\circ \cF_K ] = \int_{K\times \Theta} f_1 \cdot (f_2\circ \cF_K)\cdot \abs{\cF_K'}^{1-s} ,
\end{equation}
by~\cite[eq.~\eqrefmain{eq:transferop-invar}]{BettinEtAl}, where all integrals are with respect to~$\df s(x) \df \vth$, the product of the arc-length measure on~$\dD$ and the uniform counting probability measure on~$\Theta$.

Similarly we quote from~\cite[Section~\refmain{sec:transferop-stwist}]{BettinEtAl} the definition of the operator acting on maps~$f:\dD \to \C$ by
  \begin{equation}
    \label{eq:def-CC}
    \CC[f](y) := \ssum{h\in \cH_c \\ y \in \Dom(h)} \abs{h'(y)} f(h(y)),
  \end{equation}
  and recall the definition~\cite[eq.~\eqrefmain{eq:def-invar-Fc}]{BettinEtAl} of its invariant density~$\varphi:\dD \to \R_{>0}$, and the fact that it is piecewise smooth and continuous at vertices of~$R$.

The properties of~$\LL_s$ play a key role in our work. For the most part we will be able to directly import the required information from~\cite[Theorem~\refmain{th:limitlaw:quasi}]{BettinEtAl}.
However, we shall still need some key facts on~$\LL_s$ when computing the first and second moments in Section~\ref{sec:comp-moments}.

\begin{lemma}\label{lem:approx-eigen}
  There exists~$\delta\in(0,\frac1{10})$ and a compact connected neighbourhood~$\cU$ of~$[1-\delta/2, 2]$ contained in~$\{\Re(s)>1-\delta\}$ for which the following holds:
  \begin{enumerate}[label=(\roman*)]
    \item For each~$s\in \cU$, the operator~$\LL_s$ acting on
    $$ C^1(K\times \Theta, \|\cdot\|_1), \qquad \|f\|_1 = \|f\|_\infty + \|f'\|_\infty $$
    is quasicompact, where the derivative is taken with respect to the first variable and the norms are taken over~$K\times \Theta$.
    It has a simple eigenvalue~$\lambda(s)$ of largest modulus.
    The same holds in the space of Hölder continuous functions~$H^{1/6}(K\times\Theta, \|\cdot\|_{(1/6)})$ defined in~\cite[Section~\refmain{fsp0}]{BettinEtAl}.
    \item There exists a holomorphic family of maps~$\xi_s:K\to \C$, $s\in \cU$ such that~$(x, \vth) \mapsto \xi_s(x)$ is an eigenfunction associated to~$\lambda(s)$.\label{it:xis-anacont}
    \item We have~$\xi_s(x) = O(1)$ uniformly for~$x\in K$ and~$s\in \cU$. The same holds for~$\xi_s'(x)$\label{it:xis-bounded}.
  \end{enumerate}
\end{lemma}
\begin{proof}
  \begin{enumerate}[label=(\roman*)]
    \item For $s$ in a neighbourhood of $1$, this is~\cite[Proposition~\refmain{prop:spectrum-Ls0}]{BettinEtAl} in the Hölder space~$H^{1/6}$. The same proof gives the statement in~$C^1$ (\emph{cf.}\ also~\cite[Proposition~7.5]{LeeSun2019}). If $s\in\R_{>1}$, then the result follows from~\cite[Prop.~\refmain{prop:spectrum-Ls0}]{BettinEtAl}. Finally, by perturbation theory~\cite[Theorem~VII.1.7]{Kato1995}, we obtain the conclusion for~$s\in\cU$ if~$\cU$ is a small enough neighbourhood of~$[1-\delta/2, 2]$.
    \item First note that by~\cite[Theorem~VII.1.7]{Kato1995}, the essential spectral gap of~$\LL_s$ is bounded below by a positive constant for~$s\in\cU$. Let~$\PP_s$ be the eigenprojection associated to~$\lambda(s)$. Since~$s\mapsto \LL_s$ is bounded-holomorphic in the sense of \emph{ibid.}, we deduce that~$s\mapsto \PP_s$ also is, reducing the neighbourhood~$\cU$ if needed. Note also that the averaging operator~$A[f]:(x, \vth_1) \mapsto \abs{\Theta}^{-1} \sum_{\vth\in \Theta} f(x, \vth)$ commutes with~$\LL_s$ for all~$s$, and thus also with~$\PP_s$ for~$s\in \cU$. The claim then follows with~$\xi_s = c \PP_s[\1] \neq 0$ for any constant~$c\neq 0$, since then~$\|\xi_s\|_1 \ll \|\PP_s\| \ll 1$, and~$\xi_s = c\PP_s A[\1] = A[\xi_s]$ which only depends on the first variable. We choose~$c>0$ such that~$\xi_1 = \xi$, the map defined in~\cite[\eqrefmain{eq:def-inv-meas}]{BettinEtAl}.
    \item The claim follows since the eigenfunction~$\xi_s$ just defined satisfies~$\|\xi_s\|_1 \ll \|\PP_s\| \ll 1$.
  \end{enumerate}
\end{proof}

\section{Proofs of the main statements: limiting theorems}\label{sec:proof-limit-thm}

In this section we begin the proof of Theorems~\ref{th:main-cusp} and~\ref{th:main-eis}.

\subsection{Estimates on the observable}\label{sec:estimates-observable}

We translate the estimates of Lemmas~\ref{lem:hL-gamma} and~\ref{lem:hL-parab} in terms of the disk model.
For~$x\in C(G)$ and~$g\in G$, recalling the notation~\eqref{eq:def-pH}, we let
\begin{equation}
  \begin{aligned}
    \vL(x, s) &:= \abs{p_\H'(x)}^{1-s} L(\phi, s-\tfrac12, [x]_\H) \in \C, && \qquad (x\neq 1) \\
    \vphi_g(x,s) &:= \vth_g(x) \vL(gx, s) - \abs{g'(x)}^{s-1} \vL(x, s),&& \qquad (x\notin\{1,g^{-1}1\}).
  \end{aligned}\label{eq:defs-L-phigs-Gmodel}
\end{equation}
We warn the reader about the shift in the parameter~$s$, so that the additive twists~$L_\phi(x)$ now correspond to~$s=1$.
This choice is more natural with respect to the normalization in~\cite{BaladiVallee2005,BettinEtAl}.
We stress the difference in notation between the automorphic form~$\phi$, and the observable~$\vphi_g$ (in bold, indexed by elements of~$G$), which will also be clear from the context.

The quantities~\eqref{eq:defs-L-phigs-Gmodel} are initially defined for~$\Re(s)>3/2$, in which domain we have
\begin{equation}
  \vphi_g(x,s) = \abs{p_\H'(x)}^{1-s} \abs{g'(x)}^{s-1} \vth_g(x) h_\gamma^L([x]_\H,s-\tfrac12), \qquad (\gamma = [g]_\H),\label{eq:rel-phigs-hL}
\end{equation}
with $h_\gamma^L$ defined in~\eqref{eq:def-hL}.
Thus, by Lemma~\ref{lem:hL-gamma}, the map~$\vphi_g$ extends to a function on~$(\dD\setminus\{1, g^{-1}1\}) \times \{\Re(s)>1/2\}$
which is continuous in~$x$ for every~$s$, and holomorphic in~$s$ for every~$x$.

For~$s=1$ we drop the~$s$-dependence on the notation and let
\begin{equation}
  \begin{aligned}
    & \vL(x) := \vL(x, 1) = L_\phi([x]_\H), \\
    & \vphi_g(x) := \vphi_g(x, 1) = \vth_g(x) \vL(gx) - \vL(x).
  \end{aligned}\label{eq:defs-L-phig1-Gmodel}
\end{equation}
With a slight abuse of notation, for an inverse branch $h=h_w\in\cH_c$ (or $h\in\cH$) associated to a matrix $g=g_w\in G$, we write $\vphi_h:=\vphi_g$.

The following lemma translates the conclusion of Lemmas~\ref{lem:hL-gamma} and \ref{lem:hL-parab} to the maps~$\vphi_h$.
\begin{lemma}\label{lem:estimates-vphig}
  For all~$h\in \cH_c$, we have the following bounds:
  \begin{enumerate}
    \item For all~$\eps>0$, we have
    \begin{equation}
      \vphi_h(x) - \vphi_h(x') \ll (\abs{h'}^{-\eps} +\max_{y\in \{x,x'\}}\abs{y-1}^{-2(1-\rho_\phi)}) \abs{x-x'}^{1-\rho_\phi-\eps} \label{eq:vphi-holder}
    \end{equation}
    uniformly for~$h\in\cH_c$ and~$x, x' \in \Dom(h)\setminus\{1\}$.
    \item If~$\phi$ is cuspidal, then
    \begin{equation}
      \vphi_h(x) \ll 1 \qquad (h\in \cH_c, x\in \Dom(h)\setminus\{1\}).\label{eq:vphi-cuspidal-bound}
    \end{equation}
    \item If~$\phi$ is not cuspidal, then writing~$h$ in the form~$h = h_{w_0^nw_1}$ and letting $\omega=\omega_{w_0}$ and $\iota_{w_0},\upsilon_{w_0}$ be as defined in~\eqref{dfiu},
    we have for~$s_\phi \neq 1/2$
    \begin{equation}
      \begin{aligned}
        \vphi_h(x) ={}& A_\infty(\phi) G_k(s_\phi, [x]_\H) + B_\infty(\phi) G_k(1-s_\phi, [x]_\H) \\
        &{} + \iota_{w_0} \vth_h(x) \big(A_\omega(\phi) G_k(s_\phi, \upsilon_{w_0}n) + B_\omega(\phi) G_k(1-s_\phi, \upsilon_{w_0}n)\big) + O(1),
      \end{aligned}
      \label{eq:vphi-eisenstein-parab-estim}
    \end{equation}
    uniformly in~$x\in \Dom(h)\setminus\{1\}$ and~$n\geq 0$.
    For~$s_\phi = 1/2$ we have
    \begin{align*}
      \vphi_h(x) ={}& D_\infty(\phi, [x]_\H) + \iota_{w_0} \vth_h(x) D_\omega(\phi, \upsilon_{w_0} n) \\
      &{} + \begin{cases} O(1), & (B_\infty(\phi) = B_\omega(\phi) = 0), \\ O(\abs{x-1}^{-1/2} + n^{1/2}) & \text{(otherwise)} \end{cases}
    \end{align*}
    with the same uniformity.
    \item For $1/2+\eps<\Re(s)\ll1$ and $x\in \Dom(h)\setminus\{1\}$ we have
    \begin{equation}
      \begin{aligned}
        \vphi_h(x, s) \ll \frac{\e^{\frac{\pi }2|\Im(s)|}}{(1+|s|)^{\Re(s + s_\phi)-\frac52}}\cdot \Big(|h'|^{\Re(s)-1}\big(\abs{x-1}^{2\Re(s)-2} +\abs{x-1}^{\Re(s)-1-\rho_\phi-\eps}\big)+{}& \\
        {}+ 1+|h'|^{\Re(s)-1} + |h'|^{(\Re(s)-1-\rho_\phi-\eps)/2}&\Big)
      \end{aligned}\label{eq:ph-rough}
    \end{equation}
  \end{enumerate}
\end{lemma}
\begin{proof}
  Write~$h = h_{w}$, with $w=w_0^nw_1$ as in the statement. By~\eqref{di1}, the estimate~\eqref{eq:vphi-holder} follows from Lemma~\ref{lem:hL-gamma} for each fixed~$n$.
  Thus we may assume that~$n\geq n_0$ for any fixed~$n_0$. As a consequence, we have that $h$ is an inverse branch of kind (ii) and thus~$\Dom(h)$ is contained in some arc~$[a]$, $a\in \Sigma$. By the analogue of~\eqref{ple} for $\vphi$, we have
  \begin{equation}
    \vphi_h(x,s) = \vth_{g_{w_1}}(x) \vphi_{g_{w_0}^n}(g_{w_1} x,s) + |g_{w_0^n}'( g_{w_1} x)|^{s-1}\vphi_{g_{w_1}}(x,s)\label{eq:h-w0w1}
  \end{equation}
  for all~$x\in \Dom(h)$. Also, recall $\vphi_{g_{w_0}^n}\equiv 0$ if $\omega_{w_0}=1$. By Lemma~\ref{lem:hL-gamma} and~\eqref{di1}, we have
  \begin{align*}
    \vphi_{g_{w_1}}(x) - \vphi_{g_{w_1}}(x')& \ll \min(\abs{x-1}, \abs{x'-1})^{-2\eps} \big|[x]_\H-[x']_\H\big|^{1-\rho_\phi-\eps}\\
    & \ll \min(\abs{x-1}, \abs{x'-1})^{-(2-2\rho_\phi)} \abs{x-x'}^{1-\rho_\phi-\eps}
  \end{align*}
  whereas by Lemma~\ref{lem:hL-parab} and~\eqref{di0}-\eqref{di1}, we have
  $$ \abs{\vphi_{g_{w_0}^n}(g_{w_1} x) - \vphi_{g_{w_0}^n}(g_{w_1} x')} \ll n^\eps \abs{x - x'}^{1 - \rho_\phi - \eps} $$
  since $\Dom(h)\subseteq\Dom(h_{w_1})$ and so $\abs{g_{w_1}y-1}\gg 1$ for $y\in \Dom(h)$. Also, as implied by Lemma~\ref{lem:multiplier} below, the map~$x\mapsto \vth_{g_{w_1}}(x)$ is constant in $\Dom(h)\setminus\{1\}$. Thus, using the bound~$\abs{\vth_{g_{w_1}}(x)} \leq 1$ and taking finite differences, we get
  $$ \vphi_h(x) - \vphi_h(x') \ll (n^\eps + \min(\abs{x-1}, \abs{x'-1})^{-(2-2\rho_\phi)}) \abs{x-x'}^{1-\rho_\phi-\eps}, $$
  which implies the bound~\eqref{eq:vphi-holder} since~$\abs{h'} \asymp (1+n)^{-2}$.

  Next assume that~$\phi$ is cuspidal. Then, the bound~\eqref{eq:vphi-cuspidal-bound} follows as above by the bounds~\eqref{eq:estim-hL-largex}-\eqref{eq:estim-hL-parab} and~\eqref{di0}-\eqref{di1}.

  We focus next on~\eqref{eq:vphi-eisenstein-parab-estim}.
  By Lemma~\ref{lem:hL-parab}, with~$p = [g_{w_0}]_\H$, and~\eqref{di0}, we get
  $$ \vphi_{g_{w_0^n}}(g_{w_1} x) = \iota_{w_0} \vth_{g_{w_0^n}}(g_{w_1} x) \big( A_\omega(\phi) G_k(s_\phi, \upsilon_{w_0}n) + B_\omega(\phi) G_k(1-s_\phi, \upsilon_{w_0} n)\big) + O(1). $$
  On the other hand, by Lemma~\ref{lem:hL-gamma} with~$\gamma = [g_{w_1}]_\H$ and~\eqref{di0}-\eqref{di1}, we obtain
  $$ \vphi_{g_{w_1}}(x) = A_\infty(\phi) G_k(s_\phi, [x]_\H) + B_\infty(\phi) G_k(1-s_\phi, [x]_\H) + O(1). $$
  Inserting these estimates in~\eqref{eq:h-w0w1}, the claimed estimate~\eqref{eq:vphi-eisenstein-parab-estim} follows.

  Finally,~\eqref{eq:ph-rough} follows from~\eqref{eq:hL-rough_o}-\eqref{eq:hL-rough2} as above. We remark that we also have to account for the extra factors $\abs{p_\H'(x)}^{1-s} \abs{g'(x)}^{s-1}$ in~\eqref{eq:rel-phigs-hL}, which can be dealt with by using that by~\eqref{di0}-\eqref{di1} for $x\in\Dom(h)$ we have
  $$|p_\H'(x)|\asymp|x-1|^{-2},\quad |p_\H'(g_{w_1}x)|\asymp1,\quad |g_{w_1}'(x)|\asymp1,\quad |g_{w_0^n}'(g_{w_1}(x))|\asymp (1+n)^{-2}.\qedhere$$
\end{proof}
\begin{remark}\label{remrfn}
  In the non-cuspidal case, the bound~\eqref{eq:ph-rough} could be slightly refined, by using the~$\log$-factor in~\eqref{eq:decay-tail-phi} instead of bounding it by arbitrarily small powers of~$y$. Assuming for simplicity $|\Re(s-1)|<1/2-\eps$ and recalling~\eqref{beta}, this allows to replace the factors $|h'|^{-\eps/2}$ and $|x-1|^{-\eps}$ in the bound by $\log(2+|h'|^{-1})^{(\beta_\phi-1)/2}$ and $\log(2+|x-1|^{-1})^{(\beta_\phi-1)/2}$ respectively.
\end{remark}

We are now able to verify the main conditions in~\cite{BettinEtAl} on the norm of the period functions at parabolic elements.
As in~\cite{BettinEtAl}, we let
$$ \|f\|_{\infty,X} := \sup_{x\in X}|f(x)|, \qquad \| f\|_{(\kappa),X} := \sup_{\substack{x, x'\in X \\ x\neq x'}} \frac{\abs{f(x)-f(x')}}{\abs{x-x'}^\kappa}, \quad (\kappa>0). $$

\begin{corollary}\label{corolhyps}
  We have
  \begin{align}
    & \sum_{h_1, h_2 \in \cH_c} \abs{h_1'} \abs{h_2'} \| \vphi_{h_1} \|_{\infty, \im(h_2)}^\alpha < \infty &&
    \begin{cases}
      \text{for all } \alpha\geq 0 & (\phi \text{ cuspidal}), \\
      \text{for } 0\leq \alpha < 2 & (\text{otherwise}),
    \end{cases} \label{eq:vphi-hyp-supnorm} \\
    & \sum_{h_1, h_2,h_3 \in \cH_c} \abs{h_1'} \abs{h_2'}\abs{h_3'} \| \vphi_{h_1} \|_{\infty, \im(h_2)} \| \vphi_{h_2} \|_{\infty, \im(h_3)} < \infty \hspace{-6em} &&
    \label{eq:vphi-hyp-supnorm2} \\
    & \sum_{h_1, h_2 \in \cH_c} \abs{h_1'} \abs{h_2'} \| \vphi_{h_1} \|_{(1/3), \im(h_2)}^\lambda<\infty &&
    \text{for } 0\leq \lambda < 3/2. \label{eq:vphi-hyp-holder}
  \end{align}
\end{corollary}
\begin{proof}
  By~\cite[eqs.~\eqrefmain{eq:Hc-dist-V} and~\eqrefmain{eq:Hc-bounded-dist}]{BettinEtAl} we have $\dist(1, \im(h_2))\gg |h_2'|^{1/2}$ and~$\diam(\im(h_2)) \ll |h_2'|$. Then, using~\eqref{eq:vphi-holder}-\eqref{eq:vphi-eisenstein-parab-estim} we deduce that for each~$h_1, h_2 \in \cH_c$ with~$\im(h_2) \subset \Dom(h_1)$ we have
  \begin{align*}
    \|\vphi_{h_1}\|_{(1/3), \im(h_2)}
    \ll{}& \abs{h_1'}^{-\eps} + \abs{h_2'}^{-1/3-\eps}. 
  \end{align*}
  for any $\eps>0$. 
  Thus, for $\eps$ small enough and $\lambda<3/2$
  $$ \sum_{h_1,h_2 \in \cH_c} \abs{h_1'}\abs{h_2'} \|\vphi_{h_1}\|_{(1/3), \im(h_2)}^{\lambda} \ll
  \sum_{h_1, h_2 \in \cH_c} \abs{h_1'}^{1-\lambda\eps} \abs{h_2'}^{1-\lambda/3-\lambda\eps}\ll1 $$
  since, by~\cite[Lemma~\refmain{lem:concn}]{BettinEtAl} we have~$\sum_{h \in \cH_c} \abs{h'}^{\sigma} < \infty$ for all~$\sigma>\frac12$. This proves~\eqref{eq:vphi-hyp-holder}.

  Similarly, when~$\phi$ is non-cuspidal, eq.~\eqref{eq:vphi-eisenstein-parab-estim} yields 
  \begin{align*}
    \|\vphi_{h_1}\|_{\infty, \im(h_{2})} \ll \sup_{x\in \im(h_{2})} (1 + \abs{x - 1}^{-1/2-\eps} + \abs{h_1'}^{-1/4-\eps}) \ll \abs{h_1'}^{-1/4-\eps} + \abs{h_{2}'}^{-1/4-\eps}.
  \end{align*}
  Thus, we deduce as above that the left-hand side of~\eqref{eq:vphi-hyp-supnorm} is finite when~$\alpha < 2$, by taking~$\eps$ small enough.

  When~$\phi$ is cuspidal, by~\eqref{eq:vphi-cuspidal-bound} we have~$\|\vphi_{h_1}\|_{\infty, \im(h_2)} \ll 1$ independently of~$h_1,h_2$, and thus the left-hand side of~\eqref{eq:vphi-hyp-supnorm} is finite for all~$\alpha\geq 0$.
  
  Finally, we move to~\eqref{eq:vphi-hyp-supnorm2} for which a trivial application of~\eqref{eq:vphi-eisenstein-parab-estim} does not suffice, and we need to be more precise. Given $h=h_{w_0^nw_1}\in\cH_c$, denote $\omega_h=\omega_{w_0}$ and $\iota_{h}=\iota_{w_0}$, so that $\iota_h=0$ if $\omega_h=1$ and $|\iota_h|=1$ otherwise. Since $\dist(1, \im(h))\gg1$ for $\omega_h\neq1$, we have, in any case, $\dist(1, \im(h))\gg |\iota_{h}|+|h'|^{1/2}$. Recalling also the factor $\iota_{w_0}$ in the second line of~\eqref{eq:vphi-eisenstein-parab-estim}, we then have 
  \begin{align*}
    \|\vphi_{h_j}\|_{\infty, \im(h_{j+1})} \ll 1+\abs{\iota_{h_j}}\abs{h_j'}^{-1/4-\eps} + (1-|\iota_{h_{j+1}}|)\abs{h_{j+1}'}^{-1/4-\eps},\qquad j=1,2.
  \end{align*}
  Thus, since $(1-|\iota_{h_2}|)|\iota_{h_2}|=0$, we deduce
  \begin{align*}
    \|\vphi_{h_1}\|_{\infty, \im(h_{2})} \|\vphi_{h_2}\|_{\infty, \im(h_{3})} \ll (|h_1'||h_2'|)^{-1/4-\eps}+(|h_1'||h_3'|)^{-1/4-\eps}+(|h_2'||h_3'|)^{-1/4-\eps}
  \end{align*}
  and the result follows by summing over $h_1$, $h_2$ and $h_3$.
\end{proof}

In the rest of the proof, given $z\in\C$ we let $\vec z=(\smat{\Re(z)\\\Im(z)})\in\R^2$. Also, we let
\begin{align}\label{rcsc}
  \inpd{ z, w} := \Re(z \bar{w}) \qquad (z, w \in \C)
\end{align}
denote the usual scalar product of $\vec z$ and $\vec w$.

\subsection{Cuspidal case}

\begin{proof}[First part of the proof of Theorem~\ref{th:main-cusp}]
  Let~$\vt \in \C$. We wish to estimate the characteristic function
  $$ \chi_{Q, \cb}(\vt) := \E_{Q; \Gamma, \cb}(\e^{i \langle\vt, L_\phi(x)\rangle}). $$
  By Lemma~\ref{lem:reduction-gamma} and~\eqref{relationc}, we have
  $$ \chi_{Q, \cb}(\vt) = \E_j \E_{Q\sqrt{n_j}; G, \bar{ g}_j[\cb]_\D}(\e^{i \inpd{\vt, \vL(x)}}), $$
  where now~$x\in C(G)$, $\vL(x)$ is defined in~\eqref{eq:defs-L-phig1-Gmodel},
  and~$\E_j$ denotes a sum over~$j$ weighted by
  $$ \frac{\sum_{x\in C_{\bar{g}_j[\cb]_\D}(G; Q\sqrt{n_j})} \psi([x]_\H)}{\sum_{x\in C_{[\cb]_\D}(\bar{G}; Q)} \psi([x]_\H)} $$

  By Proposition~\ref{prop:topological-irred} below, and since the fundamental domain~$R$ has at least~$8$ sides, the multiplier system~$\vth_g(x)$ satisfies the stability and mixing hypothesis from~\cite[Section~\refmain{sec:skew-products}]{BettinEtAl}. We also have the bounds~\eqref{eq:vphi-hyp-supnorm} and \eqref{eq:vphi-hyp-holder} from Corollary~\ref{corolhyps}. By~\cite[Theorem~\refmain{th:limitlaw:quasi}]{BettinEtAl}, we deduce the existence of~$t_0>0$ such that for~$\abs{\vt}<t_0$ and any~$j$,
  \begin{equation}
    \E_{Q\sqrt{n_j}; G, \bar{ g}_j[\cb]_\D}(\e^{i \inpd{\vt, \vL(x)}}) = \exp\big\{ (\log Q) (i \inpd{\vt, \vmu} - \tfrac12 \vec\vt^{\,\,T} \vSigma \vec\vt + O(\abs{\vt}^3)) + O(Q^{-\delta}+\abs{\vt}) \big\},\label{eq:asympt-charfun-j}
  \end{equation}
  where
  \begin{align}
    \fm := {}& \int_K \xi(y) \dfs(y), \qquad \fd := \int_{K} \log\abs{\cF_K'(y)} \frac{\xi(y) \dfs(y)}{\fm} = -\lambda'(1), \label{eq:def-fd} \\
    \vmu := {}& 2 \frac{ \1_{\Theta=\{1\}}} {\fd} \int_{K} \Phi_K(y) \frac{\xi(y) \dfs(y)}{\fm}, \label{eq:def-mu}
  \end{align}
  and~$\vSigma$ is a~$2\times 2$ symmetric positive semi-definite matrix which is defined in~\cite[Lemma~\refmain{lem:estim-s0-clt}]{BettinEtAl}; in particular, it is independent of~$j$. Here we used the fact that~$\sum_{\vth\in \Theta}\vth = \1_{\Theta=\{1\}}$ in the present case.
  The map~$\Phi_K$ is defined from $\vphi_h$ as in~\cite[eq.~\eqrefmain{eq:def-PhiK}]{BettinEtAl}, and~$\xi:K \to \R_+$ is the~$\cF_K$-invariant density~\cite[Proposition~\refmain{prop:xi-invar}]{BettinEtAl}, which is smooth and positive on~$K$. We recall also that~$\lambda(s)$ was defined in Lemma~\ref{lem:approx-eigen} and that~$\dfs(y)$ denotes the arc-length measure on~$\dD$.

  Since~$\vmu$ and~$\vSigma$ are independent of~$j$,~$\chi_{Q, \cb}$ also satisfies the asymptotic relation~\eqref{eq:asympt-charfun-j}. By the Berry--Esseen theorem (see \emph{e.g.} \cite[Section~9]{BettinDrappeau2022}) along with the Portmanteau theorem, the convergence to a Gaussian law (possibly degenerate) follows, and to complete the proof of Theorem~\ref{th:main-cusp} it remains only to show that~$\vmu = 0$ and~$\vSigma$ is a positive multiple of the identity, which will be done in Section~\ref{sec:comp-moments}.
\end{proof}

\subsection{Non-cuspidal case}\label{sec:proof-limitthm-eis}

\begin{proof}[Proof of Theorem~\ref{th:main-eis}]
  For the proof, we will focus on the case when~$s_\phi = 1/2$ and~$B_\cb(\phi) = 0$ for all~$\cb\in C(\Gamma)$, and will explain afterwards the modifications required to treat the complementary cases.
  The proof of Theorem~\ref{th:main-eis} is similar to that of the cuspidal case, except that now~\eqref{eq:vphi-hyp-supnorm} holds for~$\alpha < 2$ only. By \cite[Theorem~\refmain{th:limitlaw:quasi}]{BettinEtAl}, we deduce
  \begin{equation}
    \chi_{Q, \cb}(\vt) = \exp\Big\{ (\log Q) \big( I(\vt) + O(\abs{\vt}^2) \big) + O(Q^{-\delta}+\abs{\vt}) \Big\},\label{eq:rel-chi-I}
  \end{equation}
  where, in view of~\cite[Lemma~\refmain{lem:estim-s0-subclt}]{BettinEtAl} and~\eqref{eq:vphi-hyp-supnorm2},
  \begin{align}
    I(\vt) &:= \frac{2}{\fd} \int_{\dD \times \Theta} (\e^{i \inpd{ \vt, \vth\Phi_c(x)}} - 1) \varphi(x) \frac{\dfs(x) \df \vth}{\fm} \notag\\
    &=i \inpd{\vt, \vmu} + o(|\vt|), \label{casx}
  \end{align}
  with $\vmu$ defined in~\eqref{eq:def-mu}.
  Here $\varphi$ is a positive piecewise smooth function on $\dD$ which is smooth at vertices of $R$, see~\cite[\eqrefmain{eq:def-invar-Fc}]{BettinEtAl}, and the map~$\Phi_c$ is defined for almost every~$x\in\dD$ by
  \begin{equation}
    \Phi_c(x) := \vth_h(y)^{-1}\vphi_h(y), \qquad (x = h(y), h \in \cH_c). \label{eq:def-Phic}
  \end{equation}
  The sets~$\im(h)$, for~$h\in \cH_c$, partition the circle~$\dD$ into intervals, and by changing variables (note that~$\df s[h(y)] = \abs{h'(y)}\dfs(y)$), we get
  \begin{align} \label{fii}
    I(\vt) = \frac{2}{\fd\fm}\sum_{h\in \cH_c} I_h(\vt), \quad I_h(\vt) := \int_{\Dom(h) \times \Theta} (\e^{i \inpd{\vt, \vth \vphi_h(y)}} - 1) \varphi(h(y)) \abs{h'(y)} \dfs(y) \df\vth.
  \end{align}

  Recall the classification of elements of~$\cH_c$ in Section~\ref{sec:operator-cc}. Accordingly, we divide the sum over inverse branches as
  \begin{align}\label{dcit}
    I(\vt) = \frac{2}{\fd\fm}\Big(\sum_{b\in\Sigma} I_{h_b}(\vt)+\sum_{w_0,w_1}\sum_{n}I_{h_{w_0^nw_1}}(\vt)\Big)
  \end{align}
  where $n$ is summed over $\Z_{\geq1}$ if $w_1$ has length $1$ and over $\Z_{\geq0}$ otherwise. Also, we recall that for inverse branches of the second kind we have that $\Dom(h_{w_0^nw_1})=\Dom(h_{w_0w_1})$ is independent of $n$.

  We wish to apply Lemma~\ref{lem:fourier} to estimate these integrals.
  By Lemma~\ref{lem:estimates-vphig}, the map~$\vphi_h$ is continuous on~$\Dom(h)$, with the exception of a possible blow-up at~$1$ if~$1 \in \Dom(h)$. This has the shape~$x \mapsto A_\infty(\phi) G_k(\frac12, [x]_\H) = A_\infty(\phi) G_k(\frac12, \mp 2) \abs{x-1}^{-1/2}+o(1)$ as~$\dD\ni x\to 1$ with~$\sign(\Im x) = \pm 1$ and with $G_k$ defined in~\eqref{dfgk}. Note that the coefficients here are independent of~$h$.

  Let us first assume that $\Theta\subseteq\{1, -1\}$. We abbreviate
  \begin{equation}
    \lambda_\infty(\pm) := A_\infty(\phi) G_k(\tfrac12, \mp 2), \qquad \lambda(w_0) := \1_{\omega_{w_0}\neq1} A_{\omega_{w_0}}(\phi) G_k(\tfrac12, \upsilon_{w_0}). \label{eq:def-coeffs-lambda}
  \end{equation}
  Combining the facts above, we apply Lemma~\ref{lem:fourier} to $\sum_{n}I_{h_{w_0^nw_1}}(\vt)$, where the integration is taken over the arc $\Dom(h_{w_0w_1})$ (cf.\ Remark~\ref{rmk:fourier-curves}). We let~$t := |\vt|\neq0$ and $x_0=1$. By Lemma~\ref{lem:estimates-vphig} and~\cite[Lemma~\refmain{lem:parabolic}]{BettinEtAl}, the hypotheses~\ref{it:fourier-hyp-half-simple} of Lemma~\ref{lem:fourier} are satisfied with the quantities~$\alpha$ and~$\beta$ given by
  \begin{align*}
    \alpha_{w_0,w_1} ={}& \lim_{n\to\infty} n^2 \int_{ \Dom(h_{w_0w_1})} \varphi(h_{w_0^nw_1}(y)) \abs{h_{w_0^nw_1}'(y)} \dfs(y) \\
    ={}& \varphi(\omega_{w_0}) \lim_{n\to\infty} n^2 \int_{ \Dom(h_{w_0w_1})} \abs{h_{w_0^nw_1}'(y)} \dfs(y),
  \end{align*}
  where $\omega_{w_0}$ is the fixed point of $g_{w_0}$,
  \begin{align*}
    \beta_{w_0,w_1,\pm} = \1_{1\in \overline{\Dom(h_{w_0w_1})}} \sum_{n} \varphi(h_{w_0^nw_1}(1\pm i0)) |h_{w_0^nw_1}'(1)| ,
    \numberthis\label{eq:def-betal}
  \end{align*}
  with $f(1\pm i0):=\lim_{\dD\ni x\to 1,\sign(\Im(x))=\pm1}f(x)$,
  and $a_{w_0,w_1},d_{w_0,w_1,\pm}$ satisfying
  \begin{equation}
    \begin{aligned}
      \abs{a_{w_0,w_1}}^2 ={}& \abs{\inpd{\vt/|\vt|, \lambda(w_0)}}^2, \\
      \abs{d_{w_0,w_1,\pm}}^2 ={}& \abs{\inpd{\vt/|\vt|, \lambda_\infty(\pm)}}^2,
    \end{aligned}\label{eq:coeffs-cl-dl-halfsimple}
  \end{equation}
  where we used the assumption~$\Theta\subseteq\{1, -1\}$. We deduce that
  \begin{align}\label{dcit2}
    \sum_{n}I_{h_{w_0^nw_1}}(\vt)&=i\inpd{ \vt,\vmu_{w_0,w_1}}- (K_{w_0,w_1}^{(1)}(\vt)+K_{w_0,w_1}^{(2)}(\vt)) |\vt|^2\abs{\log|\vt|}+O(|\vt|^2),
  \end{align}
  for some $\vmu_{w_0,w_1}\in\C$ and
  \begin{align*}
    K_{w_0,w_1}^{(1)}(\vt) &= \alpha_{w_0,w_1}\babs{\binpd{\vt/|\vt|, \lambda(w_0)}}^2,\\
    K_{w_0,w_1}^{(2)}(\vt) &= \sum_{n} \sum_{1\pm i0\in\Dom(h_{w_0w_1})}\babs{\binpd{\vt/|\vt|, \lambda_\infty(\pm)}}^2\varphi(h_{w_0^nw_1}(1\pm i0)) |h_{w_0^nw_1}'(1)|.
  \end{align*}
Regarding the first summand in~\eqref{dcit}, by applying Lemma~\ref{lem:fourier} as above to $ I_{h_b}(\vt)$ (with $\alpha_b:=0$, cf. Remark~\ref{rmk:fourier-curves}), we obtain
    \begin{align}\label{dcit4}
      I_{h_{b}}(\vt) &= i\inpd{ \vt, \vmu_{b}}- K_{b}(\vt) |\vt|^2\abs{\log|\vt|}+O(|\vt|^2),
    \end{align}
    for some $\vmu_b$ and
    \begin{align*}
      K_{b}(\vt) &:= \sum_{1\pm i0\in\Dom(h_{b})}\varphi(h_{b}(1\pm i0)) \abs{h_{b}'(1)}\abs{\inpd{\vt/|\vt|, \lambda_\infty(\pm)}}^2.
    \end{align*}

  We will first show that $\sum_{w_0,w_1} K_{w_0,w_1}^{(1)}(\vt)$ is a non-negative constant independent of $\vt$. We start by rewriting $\alpha_{w_0,w_1}$ in a more suitable shape.
  By~\cite[Lemma~\refmain{lem:parabolic}]{BettinEtAl}, we have
  $$ |h_{w_0^n w_1}'(y)| \sim \frac{C(w_0) |h_{w_1}'(y)|}{n^2 \dist(h_{w_1}(y), \omega_{w_0})^2}, \qquad (n\to \infty) $$
  where~$C(w_0)>0$ satisfies~$C(\hat{w_0}) = C(w_0)$ and~$\dist$ is the Euclidean distance. It follows that
  \begin{align*}
    \alpha_{w_0,w_1} &=\varphi(\omega_{w_0}) C(w_0)J(w_0,w_1),
  \end{align*}
  with
  \begin{align}\label{faf}
    J(w_0,w_1) &:= \int_{\Dom(h_{w_0w_1})}\frac{\abs{h_{w_1}'(y)}\dfs(y)}{\dist(h_{w_1}(y), \omega_{w_0})^2} = \int_{h_{w_1}(\Dom(h_{w_0w_1}))}\frac{\dfs(y)}{\dist(y, \omega_{w_0})^2} ,
  \end{align}
  after the change of variables~$y\gets h_{w_1}(y)$.
  Since~$g'_{w_0}(\omega_{w_0}) = 1$, a quick computation reveals that for all~$y\in \dD\setminus\{\omega_{w_0}\}$,
  $$ \frac{\abs{g_{w_0}'(y)}}{\dist(g_{w_0}(y), \omega_{w_0})^2} = \frac{1}{\dist(y, \omega_{w_0})^2}, $$
  thus, after a further change of variables, $y\mapsto g_{w_0} y$, we have
  \begin{align}\label{faf2}
    J(w_0,w_1) &= \int_{g_{w_0}^{-1}(h_{w_1}(\Dom(h_{w_0w_1})))}\frac{\dfs(y)}{\dist(y, \omega_{w_0})^2}.
  \end{align}
  Write $w_0=b_1\dotsb b_\ell$ and $w_1=b_1\dotsb b_r$ with $1\leq r\leq \ell$. If~$r\leq \ell - 3$, we check that
  \begin{align*}
    g_{w_0}^{-1}(h_{w_1}(\Dom(h_{w_0w_1})))= h_{w_1'}(\Dom(h_{\hat{w_0} w_1'}))
  \end{align*}
  where~$w_1'=\hat{b_{\ell}} \dotsb \hat{b_{r + 3}}$ is a prefix of $\hat{w_0}$.
  If~$\ell - 2 \leq r \leq \ell$ and~$(r, \ell)\not\in\{(1, 1), (2, 2)\}$, by iterating the change of variables we get instead
  \begin{align}\label{faf3}
    J(w_0,w_1) &= \int_{g_{w_0}^{-2}(h_{w_1}(\Dom(h_{w_0w_1})))}\frac{\dfs(y)}{\dist(y, \omega_{w_0})^2}
  \end{align}
  and we check that
  \begin{align*}
    g_{w_0}^{-2}(h_{w_1}(\Dom(h_{w_0w_1})))= h_{w_1'}(\Dom(h_{\hat{w_0} w_1'}))
  \end{align*}
  now with~$w_1' = \hat{b_\ell} \dotsb \hat{b_{\ell + 1 - r'}}$ and~$r' = 2\ell - r - 2$.
  A similar argument takes care of the remaining cases~$r=\ell \in \{1, 2\}$.

  % [The domain of integration in~$J(w_0, w_1)$ is the segment linking the two points~$g_{w_0}^{-1} x_i$, where
  % \begin{align*}
  %   x_1 ={}& [a_1, \dotsc, a_r, a_{r+1}, \hat{\ast} - \delta, \hat{\ast} - \delta, \dotsb], \\
  %   x_2 ={}& [a_1, \dotsc, a_r, a_{r+1}, \hat{\ast} + 2\delta, \hat{\ast} + \delta, \dotsb]
  % \end{align*}
  % where~$\ast$ is the previous letter in the sequence. Thus the two points are
  % \begin{align*}
  %   g_{w_0}^{-1}x_1 ={}& [\hat{a_\ell}, \dotsc, \hat{a_{r+2}}, \hat{a_{r+1}} - \delta, \hat{\ast} - \delta, \dotsb], \\
  %   g_{w_0}^{-1}x_2 ={}& [\hat{a_\ell}, \dotsc, \hat{a_{r+2}}, \hat{a_{r+1}} + 2\delta, \hat{\ast} + \delta, \dotsb]
  % \end{align*}
  % This indicates indeed that the new word~$w_1'$ should be of length~$\ell - r - 2$.
  % This is of the right shape: letting~$b_i = \hat{a_{\ell+1-i}}$,
  % \begin{align*}
  %   g_{w_0}^{-1}x_1 ={}& [b_1, \dotsc, b_{\ell-r-1}, b_{\ell-r} - \delta, \hat{\ast} - \delta, \dotsb], \\
  %   g_{w_0}^{-1}x_2 ={}& [b_1, \dotsc, b_{\ell-r-1}, b_{\ell-r} + 2\delta, \hat{\ast} + \delta, \dotsb]
  % \end{align*}
  % and we have indeed~$b_{\ell-r} - \delta = \hat{b_{\ell-r-1}} - 2\delta$ and~$b_{\ell+r} + 2\delta = \hat{\ast} + \delta$.]
  Since $\omega_{w_0}=\omega_{\hat{w_0}}$,~\eqref{faf}-\eqref{faf3} imply $J(w_0,w_1)=J(\hat{w_0},w_1')$ with the values of~$w_1'$ just defined. The map $(w_0,w_1)\mapsto (\hat{w_0},w_1')$ is a permutation among the admissible indices, thus we obtain
  \begin{equation}
    \begin{aligned}
      \sum_{w_0,w_1}K_{w_0,w_1}^{(1)}(\vt)&=\frac12\sum_{w_0,w_1}(K_{w_0,w_1}^{(1)}(\vt)+K_{\hat{w_0},w_1'}^{(1)}(\vt))\\
      &= \frac12 \sum_{w_0,w_1}\varphi(\omega_{w_0}) C(w_0)J(w_0,w_1)\big(\abs{\inpd{\vt/|\vt|, \lambda(w_0)}}^2+\abs{\inpd{\vt/|\vt|, \lambda(\hat{w_0})}}^2\big).
    \end{aligned}\label{eq:K1-invol-w0}
  \end{equation}
  As observed after its definition~\eqref{dfiu}, we have $\upsilon_{\hat{w_0}}=-\upsilon_{w_0}$. By~\eqref{eq:hparab-MT-flip} and~\eqref{eq:def-coeffs-lambda}, this implies that $\lambda(\hat{w_0})=\pm i \lambda( w_0)$ and so, recalling~\eqref{rcsc},
  $$ \abs{\inpd{\vt/|\vt|, \lambda(w_0)}}^2+\abs{\inpd{\vt/|\vt|, \lambda(\hat{w_0})}}^2=\abs{\vt/|\vt|\cdot\overline{ \lambda(w_0)}}^2=\abs{ \lambda(w_0)}^2. $$
  Thus,
  \begin{align}\label{pra1}
    \sum_{w_0,w_1}K_{w_0,w_1}^{(1)}(\vt) &= \frac12 \ssum{w_0,w_1 \\ \omega_{w_0}\neq1}\varphi(\omega_{w_0}) C(w_0)J(w_0,w_1)\abs{A_{\omega_{w_0}}(\phi) G_k(\tfrac12, \upsilon_{w_0})}^2=:c^{(1)},
  \end{align}
  say. Notice that all factors in the sum other than $\abs{A_{\omega_{w_0}}(\phi)}^2$ are strictly positive; also, as $w_0$ varies, $\omega_{w_0}$ spans the set of vertices of the fundamental domain. Thus $c^{(1)}>0$, unless $A_{\omega}(\phi)=0$ for every vertexs~$\omega\neq 1$, in which case $c^{(1)}=0$.

We now explain a similar property for the combined contributions of~$K_{w_0,w_1}^{(2)}$ and~$K_b$. By~\eqref{dcit} and the parametrization of inverse branches in Section~\ref{sec:operator-cc}, their contribution to~$I(\vt)$ is proportional to
    \begin{equation}
      \label{eq:I-infty-symmetry}
      \begin{aligned}
        & \sum_b K_b(\vt) + \sum_{w_0,w_1} K_{w_0,w_1}^{(2)}(\vt) \\
        = {}& \sum_{h\in\cH_c} \sum_{1 \pm i 0 \in \Dom(h)} \varphi(h(1 \pm i 0)) \abs{h'(1 \pm i 0)} \abs{\inpd{\vt/|\vt|, \lambda_\infty(\pm)}}^2 \\
        = {}& \sum_\pm \CC[\varphi](1 \pm i 0) \abs{\inpd{\vt/|\vt|, \lambda_\infty(\pm)}}^2 \\
        = {}& \sum_{\pm} \varphi(1 \pm i 0) \abs{\inpd{\vt/|\vt|, \lambda_\infty(\pm)}}^2,
      \end{aligned}
    \end{equation}
    by invariance of~$\varphi$~\cite[\eqrefmain{eq:def-invar-Fc}]{BettinEtAl}. Moreover, $\varphi$ is continuous at~$1$, since $1$ is a vertex of~$R$, so that by~\eqref{eq:hparab-MT-flip} and~\eqref{eq:def-coeffs-lambda} again the above is simply
    $$ \varphi(1) \big(\abs{\inpd{\vt/|\vt|, \lambda_\infty(+)}}^2 + \abs{\inpd{\vt/|\vt|, \lambda_\infty(-)}}^2\big) = \varphi(1) \abs{A_\infty(\phi) G_k(\tfrac12, 2)}^2 =: c^{(2)}, $$
    say.\footnote{The manipulation in~\eqref{eq:I-infty-symmetry} was suggested by an LLM; in an earlier version of the manuscript, we had shown this by an alternate argument which relied on certain symmetries of~$\Gamma$, which are true for Hecke congruence subgroups but false generically.}
  Also, $c^{(2)}>0$ if $A_\infty(\phi)\neq0$ and $c^{(2)}=0$ otherwise.
  Recalling~\eqref{pra1}, we deduce
  \begin{align}\label{dcit3}
    \sum_b K_b(\vt) + \sum_{w_0,w_1}(K_{w_0,w_1}^{(1)}(\vt)+K_{w_0,w_1}^{(2)}(\vt))=c
  \end{align}
  where $c:=c^{(1)}+c^{(2)}>0$ since, by assumption, $\phi$ is non-cuspidal with $B_{\omega}(\phi)=0$ for all $\omega$, and thus we must have $A_{\omega'}(\phi)\neq0$ for some cusp $\omega'$.

  Gathering~\eqref{dcit},~\eqref{dcit2},~\eqref{dcit3} and~\eqref{dcit4} and by~\eqref{casx}, we then have
  \begin{align}\label{ifff}
    I(\vt)&=i\inpd{ \vt, \vmu}- \frac{\sigma^2}2 |\vt|^2\abs{\log|\vt|}+O(|\vt|^2),
  \end{align}
  with $\sigma \asymp \sqrt{c} \neq 0$.  Inserting this estimate in~\eqref{eq:rel-chi-I}, we conclude using the Berry--Esseen theorem (see~\cite[Section~9]{BettinDrappeau2022}) that~\eqref{t2mf} holds with~$\sigma_\phi = \sigma/\sqrt{2}$, modulo the fact that~$\vmu = 0$.

  The case~$|\Theta|\geq3$ is similar and in fact simpler. Indeed, we can apply Lemma~\ref{lem:fourier} as above with the difference that now~\eqref{eq:coeffs-cl-dl-halfsimple} is to be replaced by
  \begin{equation*}
    \begin{aligned}
      \abs{a_{w_0,w_1}}^2 ={}& \abs{\inpd{\vt/|\vt|, \vth \vth_{w_0} \iota_{w_0} \lambda(w_0)}}^2, \\
      \abs{d_{w_0,w_1,\pm}}^2 ={}& \abs{\inpd{\vt/|\vt|, \vth\lambda_\infty(\pm)}}^2,
    \end{aligned}
  \end{equation*}
  where $\vth$ is the variable of integration over $\Theta$ in~\eqref{fii} and $\vth_{w_0}$ is the value taken by~$\vth_{g_{w_0^n w_1}}(x)$ in the interior of $\Dom(h_{w_0w_1})$ (see Lemma~\ref{lem:multiplier} below). Since $\Theta$ has order at least $3$, then a quick computation shows that for any~$\vt, \kappa \in \C$,
  $$ \int_\Theta \abs{\inpd{\vt, \vth \cdot \kappa}}^2\df \vth = \tfrac12 \abs{\vt}^2 \abs{\kappa}^2.$$
  Thus, performing the integration over $\vth$, we obtain~\eqref{ifff} with a positive $\sigma$ without using the symmetry above.

  This concludes the proof of Theorem~\ref{th:main-eis}, save for the equality~$\vmu = 0$, in the case when~$s_\phi = \frac12$ and~$B_\cb = 0$ for all~$\cb$.

  In the case when~$s_\phi = \frac12$ but~$B_\cb \neq 0$ for some~$\cb$, the previous arguments are adapted with the only changes that Lemma~\ref{lem:fourier} is applied with the hypotheses~\ref{it:fourier-hyp-half-double}, the logarithmic terms being due to the shape of~\eqref{eq:D-mainterm}. This setting is identical to~\cite{BettinDrappeau2022}, and leads to an asymptotic estimate of the shape
  \begin{equation}
    I(\vt) = i \inpd{\vt, \vmu} - \frac{\sigma^2}{2} |\vt|^2 \abs{\log|\vt|}^3 + O(|\vt|^2\abs{\log |\vt|}^{2+\eps})\label{eq:asympt-I-estermann}
  \end{equation}
  instead of~\eqref{ifff}. The rest of the proof is identical to the previous case with standard deviation~$\sigma_\phi = \sigma/(2\sqrt{2})$.

  Finally, let us focus on the case when~$s_\phi = \frac12 + it_\phi \neq \frac12$. We split $I(\vt)$ as in~\eqref{dcit}.
  For~$h=h_{w_0^nw_1}\in \cH_c$ and $\vth\in\Theta$, Lemma~\ref{lem:estimates-vphig} yields for~$\inpd{\vt/|\vt|, \vth\vphi_h(x)}$ an expansion of type
  $$ \sum_{j=1}^2 \big(a_j n^{1/2} \cos(\tau_j \log n + \theta_j) + d_{j,\pm} \abs{\Im x}^{-1/2} \cos(-\tau_j \log\abs{\Im x} + \theta'_j)\big) + O(1) $$
  as~$x\to 1 \pm i 0$ and~$n\to \infty$, where~$\tau_j = (-1)^{j+1} t_\phi$ and~$\theta_j, \theta'_j$ could depend on~$\vt,\vth$. The coefficients~$a_j, d_j\geq0$ and the phases~$\theta_j, \theta'_j$ are determined by the relations
  \begin{align*}
    a_1\e^{i\theta_1} ={}& \tfrac{\bar{\vt}}{\abs{\vt}} \iota_{w_0}\vth \vth_{w_0} A_{\omega_{w_0}}(\phi) G_k(s_\phi, \upsilon_{w_0}) & a_2\e^{i\theta_2} ={}& \tfrac{\bar{\vt}}{\abs{\vt}} \iota_{w_0} \vth \vth_{w_0} B_{\omega_{w_0}}(\phi) G_k(1-s_\phi, \upsilon_{w_0}), \\
    d_{1,\pm} \e^{i\theta'_1} ={}& \tfrac{\bar{\vt}}{\abs{\vt}} \vth A_\infty(\phi) G_k(s_\phi, \mp2), & d_{2,\pm} \e^{i\theta'_2} ={}& \tfrac{\bar{\vt}}{\abs{\vt}} \vth B_\infty(\phi) G_k(1-s_\phi, \mp2).
  \end{align*}
  Note that the moduli~$a_j, d_j$ do not depend on~$\vt$ (contrary to the earlier case~\eqref{eq:coeffs-cl-dl-halfsimple}).
  We apply Lemma~\ref{lem:fourier} with hypotheses~\ref{it:fourier-hyp-nothalf} and obtain
  $$ \sum_{w_0,w_1}\sum_{n}I_{h_{w_0^nw_1}}(\vt)= i\inpd{\vt, \vmu'} - c |\vt|^2 | \log(|\vt|)| + O(|\vt|^2), $$
for some $\vmu'$ and for~$c$ given by
    \begin{equation}
      c = \frac12 \sum_{w_0, w_1} \Big(\varphi(\omega_{w_0}) C(w_0) J(w_0, w_1) \big|a_1\e^{i\theta_1} + a_2 \e^{-i\theta_2}\big|^2 + [\dotsb] \Big)\label{eq:K1-caseii-after-fourier}
    \end{equation}
    where~$[\dotsb]$ represents the corresponding contribution from the possible pole at~$x\to 1\pm i0$. We focus on the~$n$ contribution written above.
    We expand the modulus squared and first consider the cross-product contribution. It is the real part of
    \begin{equation}
      \label{eq:K1-caseii-cross}
      \begin{aligned}
        & \big( \sum_{w_0, w_1} \varphi(\omega_{w_0}) C(w_0) J(w_0, w_1) a_1\e^{i\theta_1} a_2 \e^{i\theta_2} \\
        ={}& \big(\tfrac{\bar{\vt}}{\abs{\vt}}\big)^2 \vth^2 \sum_{w_0, w_1} \varphi(\omega_{w_0}) \iota_{w_0}^2\vth_{w_0}^2 C(w_0) J(w_0, w_1) A_{\omega_{w_0}}(\phi) B_{\omega_{w_0}}(\phi) G_k(s_\phi, \upsilon_{w_0}) G_k(1-s_\phi, \upsilon_{w_0}).
      \end{aligned}
    \end{equation}
    We proceed with the change of variables~$(w_0, w_1) \gets (\hat{w_0}, w_1')$ which we used above in~\eqref{eq:K1-invol-w0}.
    This change of variables preserves the fixed point~$\omega_{w_0} = \omega_{\hat{w_0}}$ and the values~$C(w_0)$ and~$J(w_0, w_1)$.
    On the other hand, we have~$\upsilon_{\hat{w_0}} = -\upsilon_{w_0}$ by~\eqref{eq:upsilon-flip-hat}, and we check that
    $$ \frac{G_k(s_\phi, -1) G_k(1-s_\phi, -1)}{G_k(s_\phi, 1) G_k(1-s_\phi, 1)} = (-1)^k \e^{\pi i(1-k)} = -1. $$
    Thus the cross-term~\eqref{eq:K1-caseii-cross} is zero. The remaining contribution from the squares in~\eqref{eq:K1-caseii-after-fourier} can then be handled in the exact same way as~\eqref{pra1} above.
    The contribution from the singularity at~$x\to 1\pm i0$ in~\eqref{eq:K1-caseii-after-fourier} is analyzed by the same method:
    the corresponding cross-term is shown to vanish, after the manipulation in~\eqref{eq:I-infty-symmetry}, by matching opposite~$\pm$ signs.
  This leads to an expansion of type~\eqref{ifff}, hence Theorem~\ref{th:main-eis}, also in this case.
\end{proof}

\section{Computation of the first and second moments}
\label{sec:comp-moments}

\subsection{Overview}
\label{sec:comp-moments-overview}

In order to complete the proof of Theorems~\ref{th:main-cusp} and \ref{th:main-eis}, we are left to prove:
\begin{itemize}
  \item that in both the cuspidal and non-cuspidal cases, the mean-value~$\vmu$ is zero,
  \item that in the cuspidal case, the covariance matrix~$\vSigma$ is equal to~$\sigma_\phi^2 \Id$ where~$\Id$ is the identity and~$\sigma_\phi$ is given by~\eqref{eq:def-variance-cusp}.
\end{itemize}

In the modular group case~$\Gamma = \SL(2,\Z)$ treated in~\cite{BettinDrappeau2022}, this was done by computing the empirical expectations~$\E_Q[\vL(x)]$, $\E_Q[\vL(x)^2]$ and~$\E_Q[\abs{\vL(x)}^2]$ directly.
The computations made crucial use of the existence, for~$\Gamma = \SL(2,\Z)$, of an explicit expression for sums
\begin{equation}
  \sum_{x\in \Omega_Q}\e(\ell x), \qquad (\ell \in \Z)\label{eq:def-weyl-Omega}
\end{equation}
in terms of Ramanujan sums. The Ramanujan sums enjoy a large amount of cancellation and a multiplicative behaviour, related to the existence of Hecke operators, which was crucial in the arguments in~\cite{BettinDrappeau2022}. In contrast, for arbitrary groups~$\Gamma$, existing results on the estimation of~\eqref{eq:def-weyl-Omega} rely on the analytic properties of Selberg's Kloosterman zeta function (see~\cite[eq.~(3.9)]{PetridisRisager2018}), and the available uniformity in~$\ell$ is insufficient.

The dynamical systems approach naturally yields the expressions~\eqref{eq:def-mu} for the mean-value, and an analogous expression for the variance.
However they involve the period functions~$\vphi_g$ whose definition in~\eqref{eq:defs-L-phig1-Gmodel}, as the continuous extension of a map initially defined on the countable set~$C(\Gamma)$, is of little use by itself.

We solve this issue by approximating the function~$\vphi_g$ by the non-central $L$-values~\eqref{eq:defs-L-phigs-Gmodel} with~$s=1+\eps$ for~$\eps>0$. This is inspired by the construction in~\cite{DrappeauNordentoft2022}, see equation~(3.13) there.
The main point regarding our argument here is that the maps~$\vL(x, s)$ are now continuous in~$x$ by uniform convergence.
We insert the expression~\eqref{eq:defs-L-phigs-Gmodel} into the definition~\eqref{eq:def-mu} for~$\vmu$, separate the contributions of~$\vL(g x, s)$ and~$\vL(x, s)$, and recombine them using the invariance of~$\xi$. We essentially arrive at an equality of the shape
\begin{equation}
  \vmu = \frac2\fm \lim_{s\to 1^+} (s-1) \int_K \vL(y, s) \xi(y) \dfs(y), \label{eq:mu-heuristic-integral}
\end{equation}
where the factor~$s-1$ arises from the effect of the factor~$\abs{g'(x)}^{s-1}$ in~\eqref{eq:defs-L-phigs-Gmodel}.
The equality~\eqref{eq:mu-heuristic-integral} can be thought of as a continuous analogue of the empirical expectation $\lim_{Q\to\infty} \frac1{\log Q} \E_Q[\vL(x)]$.
In a way, the dynamical systems argument took care of replacing the discrete average over~$C(\Gamma)$ by a continuous average over~$\dD$.
The continuous average~\eqref{eq:mu-heuristic-integral} is easily estimated by inserting the Fourier expansion stemming from~\eqref{eq:def-Lphixs}, using a rough bound on~$a(n)$ and on the Fourier coefficients of~$\xi$, and this leads to the value~$\vmu = 0$.

Regarding the covariance~$\vSigma$, although the computations are more involved, we eventually find that the same idea leads to the expression
$$ \vSigma = \frac4\fm\lim_{s\to 1^+} (s-1) \int_K \vec\vL(y, s) \vec\vL(y, s)^T \xi(y) \dfs(y). $$
Expanding the product leads to sums of the shape~$\sum_n a(n) \bar{a(n+h)}$, which are handled by means of Rankin--Selberg theory.
Contrary to the mean-value, the diagonal terms $h=0$ survive the integration over~$y$ and give the main contribution.

\subsection{Approximation lemmas}

\subsubsection{Approximating~$\vphi_h(y)$ by non-central $L$-values}

Recall the definition of the non-central~$L$-values~$\vL(x, s)$ and period functions~$\vphi_g(x, s)$ in~\eqref{eq:defs-L-phigs-Gmodel}.
Recall also the definition of~$\xi_s$ and~$\cU$ in Lemma~\ref{lem:approx-eigen}, of the maps~$p_\H, p_\D$ in~\eqref{eq:def-pD}, and that~$\cH$ denotes the set of inverse branches of~$\cF_K$ (cf. Section~\ref{opll}).
From the proof of Lemma~\ref{lem:approx-eigen}, we have~$\xi_1 = \xi$.

\begin{lemma} \begin{enumerate}[label=(\roman*)]
    \item \label{it:phigs-ancont} For all~$g \in G$ and~$x\in \dD \setminus\{1,g^{-1}1\}$, the map~$s\mapsto \vphi_g(x, s)$ extends to an analytic function on $\Re(s)>1/2$.
    \item For~$\Re(s)> 1 + \rho_\phi$, the map~$x\mapsto \vL(x, s)$ is continuous and bounded on~$\dD$.
    \item \label{it:bound-sumK-phih} Uniformly for~$s\in\cU$, we have
    \begin{align}\label{bsa}
      \sum_{h\in \cH} \|h' \abs{\vphi_h(\cdot, s)}^\alpha \|_{\infty, \Dom(h)} \ll 1, \qquad \big(\alpha < \min(\tfrac{1}{\delta+\rho_\phi},\tfrac1{2\delta})\big)
    \end{align}
  \end{enumerate}\label{lem:approx-phi}
\end{lemma}
\begin{proof}
  The first point is proven in Section~\ref{sec:estimates-observable}. The second point follows from the bound~\eqref{eq:bound-meanabs-an} in the non-cuspidal case, and from Wilton's bound~\eqref{eq:wilton} in the cuspidal case. The third point follows from~\cite[Lemma~\refmain{lem:moments-c-implies-K}]{BettinEtAl} and~\eqref{eq:ph-rough}.
\end{proof}

\subsubsection{Leading eigenparameters}

For~$s\in \cU$, $x\in \R$, we let
\begin{align}\label{dfnzt}
  \zeta_s(x) = \xi_s\Big(\frac{x-i}{x+i}\Big) \frac{2^s}{(x^2+1)^s}\1_{x\in [K]_\H},
\end{align}
where we recall~$[K]_\H = p_\H(K)$. Note that, by a change of variable, for all $n\in\Z$ we have
\begin{equation}
  \hat\zeta_s(n) := \int_\R \zeta_s(x) \e(-nx)\df x = \int_K \abs{p'_\H(y)}^{1-s} \xi_s(y) \e(-n p_\H(y)) \dfs(y).\label{eq:def-fouriercoeff-phi}
\end{equation}

\begin{lemma}\label{lem:fourier-eigen-approx}
  We have~$ \hat\zeta_s(n) \ll \abs{n}^{-1}$ uniformly for~$n\in\Z_{\neq 0}$ and for~$s\in\cU$.
\end{lemma}
\begin{proof}
  Note that~$[K]_\H$ is a finite union of bounded intervals.
  By Lemma~\ref{lem:approx-eigen} we have~$\abs{\zeta_s(x)}, \abs{\zeta_s'(x)} \ll 1$ uniformly for~$s\in \cU$ and~$x\in [K]_\H$, whence by partial summation we deduce~$ \hat\zeta_s(n)\ll 1/\abs{n}$ with the same uniformity.
\end{proof}

\subsection{First moment}

We recall the expression
$$ \vmu := 2\frac{\1_{\Theta=\{1\}}}{\fd} \int_K \Phi_K(y) \frac{\xi(y)\dfs(y)}{\fm} $$
for the mean-value, where~$\fd$ and~$\fm$ are defined in~\eqref{eq:def-fd}.

\begin{proposition}\label{prop:average}
  We have~$\vmu = 0$.
\end{proposition}
\begin{proof}
  We may assume~$\Theta=\{1\}$. For all~$s\in\cU$, define
  \begin{equation}
    \label{eq:def-muphis}
    \vmu(s) := \sum_{h\in \cH} \int_{\Dom(h)} \vphi_h(y, s) \xi_s(h(y)) \abs{h'(y)} \dfs(y)
  \end{equation}
  and note that by Lemma~\ref{lem:approx-eigen}.\ref{it:xis-bounded} and~\eqref{bsa} we have
  \begin{align*}
    \sum_{h\in \cH} \int_{\Dom(h)} \abs{\vphi_h(y, s) \xi_s(h(y)) h'(y)} \dfs(y) \ll 1.
  \end{align*}
  Note also that the maps~$\xi_s$ are defined on~$K$.
  By Lemmas~\ref{lem:approx-phi}.\ref{it:phigs-ancont} and~\ref{lem:approx-eigen}.\ref{it:xis-anacont}, we deduce that the map~$s\mapsto \vmu(s)$ is analytic in~$\cU$, and
  $$ \vmu = \frac{2}{\fd\fm}\vmu(1). $$
  Assume next~$\Re(s)>3/2$. We insert the explicit expression from~\eqref{eq:defs-L-phigs-Gmodel}, and separate both contributions to the integral. Note that~$\vth_h(y) = 1$ since we assumed~$\Theta=\{1\}$. We find for all~$\Re(s)>3/2$
  $$ \vmu(s) = M_1(s) - M_2(s), $$
  where
  \begin{align*}
    M_1(s) ={}& \sum_{h\in \cH} \int_{\Dom(h)} \vL(h(y), s) \xi_s(h(y)) \abs{h'(y)} \dfs(y), \\
    M_2(s) ={}& \sum_{h\in \cH} \int_{\Dom(h)} \vL(y, s) \xi_s(h(y)) \abs{h'(y)}^s \dfs(y).
  \end{align*}
  Since $\bigcup _{h\in \cH}\im(h)$ is a partition of~$\dD$ up to a countable set, after a change of variables~$M_1(s)$ becomes
  $$ M_1(s) = \sum_{h\in \cH} \int_{\im(h)} \vL(y, s) \xi_s(y) \dfs(y) = \int_K \vL(y, s) \xi_s(y) \dfs(y). $$
  Concerning~$M_2(s)$, by definition we have for all~$y\in K$
  $$ \sum_{h\in \cH, y\in \Dom(h)} \abs{h'(y)}^s \xi_s(h(y)) = \LL_s[\xi_s](y) = \lambda(s) \xi_s(y)$$
  with~$\lambda(s)$ as in Lemma~\ref{lem:approx-eigen}. We deduce
  $$ M_2(s) = \lambda(s)\int_K \vL(y, s) \xi_s(y) \dfs(y). $$
  We conclude that
  \begin{equation}
    \vmu(s) = (1-\lambda(s)) \int_K \vL(y, s) \xi_s(y) \dfs(y).\label{eq:muphi-intqf}
  \end{equation}
  We insert the definition
  \begin{equation}
    \vL(y, s) = \abs{p_\H'(y)}^{1-s} \sum_{n\geq 1} a(n) \e(n p_\H(y)) n^{1-s},\label{eq:qf-xs-cusp}
  \end{equation}
  which is allowed by the bound~\eqref{eq:bound-meanabs-an}, since~$\Re(s)>3/2$. We obtain
  \begin{equation}
    \vmu(s) = (1 - \lambda(s)) \sum_{n\geq 1} a(n) n^{1-s} \hat\zeta_s(-n), \qquad (\Re(s)>3/2),\label{eq:identity-mus-anacont}
  \end{equation}
  where~$ \hat\zeta_s(\cdot)$ is defined in~\eqref{eq:def-fouriercoeff-phi}.
  By Lemma~\ref{lem:fourier-eigen-approx}, the sum over~$n$ on the right-hand side converges uniformly for~$s\in\cU$. Therefore, we conclude by analytic continuation that the identity~\eqref{eq:identity-mus-anacont} holds for all~$s\in \cU$. Since~$\lambda(1) = 1$, the proof is concluded by setting~$s=1$.
\end{proof}

\subsection{Second moment in the cuspidal case}

In the remainder of this section, we restrict~$s$ to a real neighbourhood of~$1$. Let~$\delta>0$ be fixed and
$$ \cV := [1, 1+\delta]. $$
For~$s\in \cV$,~$x \in K$ and~$\vth\in\Theta$, we define~$\Phi_{K,s}$ by
\begin{equation}
  \Phi_{K,s}(x, \vth) = \vth \vth_h(y)^{-1} \vphi_h(y, s), \qquad (h\in \cH, x = h(y)),\label{eq:def-Phics}
\end{equation}
in analogy with~\eqref{eq:def-Phic}. We let
  \begin{equation}
    \vchi = \frac1{\xi} ({\Id} - \NN)^{-1}\LL[\Phi_K \xi(\cdot)], \label{eq:def-chi}
  \end{equation}
  where~$\NN = \LL - \PP$, with~$\PP$ the spectral projection onto the eigenspace spanned by~$\xi$.

  \begin{lemma}\label{lem:approx-chi}
    \begin{enumerate}[label=(\roman*)]
      \item The map~$\vchi$ belongs to~$H^{1/6}(K\times \Theta)$.
      \item We have
      \begin{equation}
        \vchi \xi - \LL[\vchi \xi] = \LL[\Phi_K \xi]. \label{eq:chis-implicit-eq}
      \end{equation}
      \item For~$s\in \cV$, $s\neq 1$, we have
      \begin{align}
        & \int_K \abs{\vec\vL(x, s)}^2 \dfs(x) \ll 1/(s-1), \label{eq:L2bound-Lxs} \\
        &\norm{(\LL - \LL_s)[\vchi \xi]}_{\infty, K\times\Theta} \ll s-1, \label{eq:LLchixi-1s} \\
        & \int_{K\times\Theta} \abs{\Phi_K \xi - \Phi_{K,s}\xi_{2s-1} \abs{\cF_K'}^{1-s}}^2 \df \nu \ll (s-1)^2. \label{eq:LLPhixi-1s}
      \end{align}
    \end{enumerate}
  \end{lemma}
\begin{proof}
  From the definition of~$\LL$ in~\eqref{eq:def-LLs}, we have
  \begin{equation}
    \LL[\Phi_K \xi](y, \vth_1) = \vth_1 \ssum{h\in \cH \\ y \in \Dom(h)} \abs{h'(y)} \vphi_h(y) \xi(h(y)).\label{eq:LL-PhiXi}
  \end{equation}
  By Lemma~\ref{lem:approx-eigen} and~\eqref{bsa}, uniformly in $y\in K$ we have
  $$ \ssum{h\in \cH \\ y \in \Dom(h)} \big| \abs{h'(y)} \vphi_h(y, 1) \xi(h(y)) \big| \ll \sum_{h\in\cH} \abs{h'} \|\vphi_h\|_{\infty,\Dom(h)} \ll 1. $$
  Let~$h\in\cH$ and~$\vth_1$ be fixed. For~$y_1, y_2 \in \Dom(h)$, we have
  $$ \abs{h'(y_1)} - \abs{h'(y_2)} \ll \abs{h'} \times \abs{y_1 - y_2} $$
  by the bounded distortion property~\cite[eq.~\eqrefmain{eq:bounded-dist}]{BettinEtAl}. We also have
  $$ \abs{\xi(h(y_1)) - \xi(h(y_2))} \ll \abs{h'} \times \abs{y_1 - y_2} $$
  by Lemma~\ref{lem:approx-eigen}.(iii). Therefore, we get that the map
  $$ y \mapsto \abs{h'(y)} \vphi_h(y, 1) \xi(h(y)) $$
  is~$\frac16$-Hölder continuous on~$\Dom(h)$, with Hölder constant bounded by
  $$ \abs{h'} (\|\vphi_h\|_{\infty, \Dom(h)} + \|\vphi_h\|_{(1/6), \Dom(h)})\ll \abs{h'} (\|\vphi_h\|_{\infty, \Dom(h)} + \|\vphi_h\|_{(1/3), \Dom(h)}), $$
  by~\cite[eq.~\eqrefmain{hld4}]{BettinEtAl}.
  By~\eqref{eq:vphi-hyp-supnorm} and~\eqref{eq:vphi-hyp-holder} in conjunction with~\cite[Lemma~\refmain{lem:moments-c-implies-K}]{BettinEtAl} we have that $\sum_{h\in\cH}\abs{h'} (\|\vphi_h\|_{\infty} + \|\vphi_h\|_{(1/3)})\ll1$ and thus 
  the map~\eqref{eq:LL-PhiXi} belongs to~$H^{1/6}(K\times \Theta)$ as claimed, and therefore so does~$\vchi$.
  The identity~\eqref{eq:chis-implicit-eq} follows from the definition~\eqref{eq:def-chi}
  since~$\PP\NN=0$, $\PP\LL=\PP$ and~$\PP[\Phi_K\xi] =  c \vmu \xi = 0$ for some constant~$c>0$.

  We turn to the bounds~\eqref{eq:L2bound-Lxs}--\eqref{eq:LLPhixi-1s}. First, by partial summation and the hypothesised upper-bound~\eqref{eq:rankinselberg-upper}, for all~$s>1$ we have
  $$ \int_K \abs{\vL(x, s)}^2\dfs(x) \ll \int_K \abs{p_\H'(x)}^{2(1-s)} \abs{L(\phi, s-1/2, [x]_\H)}^2\dfs(x) \ll \int_0^1 \abs{L(\phi, s-1/2, y)}^2\df y $$
  since~$[K]_\H$ is bounded. Let~$M(s)$ denote the right-hand side. By the Wilton bound~\eqref{eq:wilton}, the map~$s\mapsto L(\phi, s-1/2, y)$ is holomorphic and uniformly bounded in compacta of~$\{\Re(s)>1\}$. By Parseval's identity, we get
  $$ M(s) = \sum_{n\geq 1} \abs{a(n)}^2 n^{2(1-s)}. $$
  We conclude using~\eqref{eq:rankinselberg-upper} and partial summation that~$ M(s) \ll 1/(s-1) $, which gives~\eqref{eq:L2bound-Lxs}.

  The bound~\eqref{eq:LLchixi-1s} follows immediately from~\cite[Lemma~\refmain{lem:perturbation-LL1t}]{BettinEtAl} since~$\norm{\vchi\xi}_{\infty, K}<\infty$.

  To prove the bound~\eqref{eq:LLPhixi-1s}, we use the holomorphicity in~$s\in \cU$ of the maps~\eqref{eq:rel-phigs-hL} along the Cauchy formula on a rectangle~$\cR\subset\cU$ around~$\cV$. We obtain for the left-hand side of~\eqref{eq:LLPhixi-1s} an upper-bound of
  \begin{align*}
    &\abs{s-1}^2 \int_{K\times\Theta} \Big| \int_\cR \Phi_{K,u} \xi_{2u-1} \abs{\cF_K'}^{1-u}\frac{\df u}{(u-1)(u-s)}\Big|^2 \df\nu \\
    \ll{}& \abs{s-1}^2 \int_{K\times\Theta} \int_\cR \abs{\Phi_{K,u} \xi_{2u-1} \abs{\cF_K'}^{1-u}}^2 \df u \df\nu 
  \end{align*}
  after Cauchy-Schwarz's inequality. Inverting the integrals, we are left to bound
  $$ \int_{K\times\Theta} \abs{\Phi_{K, u} \xi_{2u-1} \abs{\cF_K'}^{1-u}}^2 \df\nu $$
  uniformly in~$u\in \cU$. By Lemma~\ref{lem:approx-eigen}.(iii) and since~$\Re(u)\geq 1-\delta$, using~\eqref{cvl} we upper-bound this by
  $$ \int_{K\times\Theta} \LL_{2\Re(u)-1}[ \abs{\Phi_{K, u}}^2] \df\nu \ll \sum_{\substack{h\in \cH}} \abs{h'}^{2\Re(u)-1} \|\vphi_h(\cdot, u)\|^2_{\infty, \Dom(h)}, $$
  and this is seen to be uniformly bounded whenever~$\delta$ is small enough by the same argument as for~\eqref{bsa}, by using~\cite[Lemma~\refmain{lem:moments-c-implies-K}]{BettinEtAl} and~\eqref{eq:ph-rough}.
\end{proof}

% \begin{lemma}\label{lem:L-Pchi-vanish}
%   For all~$\vth\in\Theta$, as~$s\to 1$ with~$s>1$, we have
%   \begin{equation}
%     \int_K \vec\vL(x, s) \vec\PP_s[\vchi_s \xi_s]^T(x, \vth) \dfs(x) \to 0.\label{eq:limit-residue-chi}
%   \end{equation}
% \end{lemma}

% \begin{proof}
%   Since~$\PP \NN = 0$, so that $\PP({\Id} - \NN)^{-1}=\PP$, and~$\PP\LL = \PP$, we have
%   $$ \PP[\vchi \xi] = \PP[\Phi_K \xi] = \Big(\int_{K\times \Theta} \Phi_K \xi \Big) \xi = 0 $$
%   by Proposition~\ref{prop:average}. Thus we have $\PP_s[\vchi_s \xi_s] = (s-1) g_s$ for some bounded, holomorphic family of maps~$g_s \in H^{1/6}(K\times \Theta)$.
%   Using the complex number representation of~$\vL$ and~$g_s$, the limit~\eqref{eq:limit-residue-chi} is equivalent to
%   \begin{equation}
%     (s-1) \sum_{n\geq 1} \frac{a_n}{n^{s-1}} \int_K \abs{p_\H'(y)}^{1-s} \e(n p_\H(y)) g_s(y, \vth) \dfs(y) \to 0 \qquad (s\to 1) \label{eq:residue-chi-fourier}
%   \end{equation}
%   and the same with~$g_s$ replaced by~$\bar{g_s}$. We focus on~\eqref{eq:residue-chi-fourier}.
%   Integrating by parts (see \emph{e.g.}~\cite[Lemma~4.1]{BettinDrappeau2022}), we deduce
%   $$ \int_K \abs{p_\H'(y)}^{1-s} \e(n p_\H(y)) g_s(y, \vth) \dfs(y) \ll n^{-1/6} $$
%   uniformly for~$s$ in a neighbourhood of~$1$. Thus, as in the proof of Proposition~\ref{prop:average}, the sum over~$n$ in~\eqref{eq:residue-chi-fourier} is uniformly bounded, and we conclude that~\eqref{eq:residue-chi-fourier} holds as claimed.
% \end{proof}

The covariance matrix in~\eqref{eq:asympt-charfun-j}, which we seek to evaluate, is given explicitly in~\cite[Lemma~\refmain{lem:estim-s0-clt}]{BettinEtAl}. An alternative expression, more useful for our purpose, is given in~\cite[eq.~\eqrefmain{eq:expr-sigma-variant}]{BettinEtAl} and states
\begin{equation}
  \vSigma = \frac{2}{\fd} \int_{K\times \Theta} \big(\vec\Phi_{K} \vec\Phi_{K}^T + \vec\vchi \vec\Phi_{K}^T + \vec\Phi_{K} \vec\vchi^T \big)\xi\df\nu,
  \label{eq:def-Sigma-K}
\end{equation}
where~$\df \nu = \fm^{-1} \df s(x) \df \vth$.

\begin{proposition}\label{prop:variance}
  We have~$\vSigma = \sigma_\phi^2 \Id$, with~$\sigma_\phi$ given in~\eqref{eq:def-variance-cusp}.
\end{proposition}

\begin{remark}
  It is instructive to compare with the value obtained by Nordentoft in~\cite{Nordentoft2021} in the case when~$k$ is even and~$\phi$ is associated to a cuspidal holomorphic form~$f(z) = y^{-k/2} \phi(z)$. We then have~$s_\phi = k/2$ and~$\|\phi\| = \|f\|_k$ the Petersson $k$-norm, therefore
  $$ \sigma_\phi^2 = \frac{\|f\|_k^2}{(k-1)! \Vol(\Gamma\backslash \H)} $$
  which recovers the value obtained in~\cite[eq.~(1.4)]{Nordentoft2021}.
  Note that the factor~$(4\pi)^k$ in~\cite{Nordentoft2021} is due to a difference in normalization originating from the Whittaker function~$W_{k/2, (k-1)/2}(4\pi y)$.
\end{remark}

\begin{proof}[Proof of Proposition~\ref{prop:variance}]
  % First we reduce to the case~$k\in \{0, 1\}$. To this end, let~$\psi = L_k\phi$ be the weight-lowered form, of weight~$k-2$ and eigenvalue~$s_\psi = s_\phi$. Then~$a_\psi(n) = -a_\phi(n)$ for all~$n\geq 0$, so that~$L_\psi(x) = - L_\phi(x)$. On the other hand,
  % $$ \frac{\sigma_\psi^2}{\sigma_\phi^2} = \frac{\|\psi\|^2}{\|\phi\|^2} (-s_\phi + \tfrac k2)(-1+ s_\phi + \tfrac k2) = 1, $$
  % where the last equality is~\cite[eq.~(4.38)]{DukeEtAl2002}. When~$k\geq 2$, repeating this operation shows that we may replace~$\phi$ in Proposition~\ref{prop:variance} by its ``ancestor'' of weight~$0$ or~$1$. A similar argument holds when~$k\leq -1$, using instead the weight-raising operator. Thus we may assume~$k\in \{0, 1\}$.

  We start from the expression~\eqref{eq:def-Sigma-K}.
  By Lemmas~\ref{lem:approx-eigen}, \ref{lem:approx-phi}, \ref{lem:approx-chi}, and dominated convergence, we get
  $$ \vSigma = \lim_{s\to 1^+} \vSigma(s), $$
  where
  \begin{equation}
    \vSigma(s) := \frac2{\fd} \int_{K\times \Theta} \Big(\vec\Phi_{K,s} \vec\Phi_{K,s}^T \xi_{2s-1} + (\vec\vchi \vec\Phi_{K,s}^T + \vec\Phi_{K,s} \vec\vchi^T)\xi \Big)\df\nu.\label{eq:def-Sigmas}
  \end{equation}
  Recall the definition~\eqref{eq:def-FK-extended}, and for $s>1$, let
  \begin{equation}
    \vL_s(x, \vth) := \vth \vL(x, s) \label{eq:def-vL-subs}
  \end{equation}
  which is continuous in~$x$ by Lemma~\ref{lem:approx-phi}.
  Thus, by~\eqref{eq:def-Phics} and~\eqref{eq:defs-L-phigs-Gmodel}, we may write
  \begin{equation}
    \Phi_{K, s}(x, \vth) = \vL_s(x, \vth) - \abs{\cF_K'(x)}^{1-s} (\vL_s \circ \cF_K)(x, \vth),\label{eq:expr-PhiKs-difference}
  \end{equation}
whence
    \begin{align*}
      \int_{K\times \Theta} \vec\Phi_{K,s} \vec\vchi^T \xi \df\nu
      ={}& \int_{K\times \Theta} (\vec\vL_s - \abs{\cF_K'}^{1-s} (\vec\vL_s \circ \cF_K)) \vec\vchi^T \xi \df \nu \\
      ={}& \int_{K\times \Theta} \vec\vL_s (\vec\vchi \xi -\vec \LL_s[\vchi\xi])^T \df \nu,
    \end{align*}
    by~\eqref{cvl}. Using~\eqref{eq:chis-implicit-eq} and~\eqref{cvl}, this becomes
    \begin{align*}
      \int_{K\times \Theta} \vec\Phi_{K,s} \vec\vchi^T \xi \df\nu={}& \int_{K\times \Theta} \vec\vL_s \vec\LL[\Phi_K \xi]^T\df\nu + E_1 \\
      ={}& \int_{K\times \Theta} (\vec\vL_s\circ\cF_K) \vec\Phi_K^T \xi\df\nu + E_1 \\
      ={}& \int_{K\times \Theta} (\vec\vL_s\circ\cF_K) \vec\Phi_{K,s}^T \xi_{2s-1} \abs{\cF_K'}^{1-s}\df\nu + E_1 + E_2, \\
    \end{align*}
    where
    \begin{equation}
      \label{eq:bounds-diff-sigma-chis}
      \begin{aligned}
        \abs{E_1} \leq{}& \int_{K\times\Theta} \abs{\vL_s} \cdot \abs{(\LL - \LL_s)[\vchi\xi]} \df \nu, \\
        \abs{E_2} \leq{}& \int_{K\times\Theta} \abs{\vL_s \circ \cF_K} \cdot \abs{\Phi_{K,s}\xi_{2s-1}\abs{\cF_K'}^{1-s} - \Phi_K \xi} \df \nu.
      \end{aligned}
    \end{equation}
    By Cauchy-Schwarz and~\eqref{eq:L2bound-Lxs}--\eqref{eq:LLchixi-1s}, we get~$E_1 \ll \abs{s-1} (\int_{K\times\Theta} \abs{\vL_s}^2\df\nu)^{1/2} \ll \abs{s-1}^{1/2}$.
    Similarly, using~\eqref{eq:LLPhixi-1s} we get~$E_2 \ll \abs{s-1} (\int_{K\times\Theta} \abs{\vL_s\circ \cF_K}^2\df\nu)^{1/2} \ll \abs{s-1}^{1/2}$.
    We have therefore obtained as~$s\to 1$
    $$ \int_{K\times \Theta} \vec\Phi_{K,s} \vec\vchi^T \xi \df \nu =
    \int_{K\times \Theta} (\vec\vL_s\circ\cF_K) \vec\Phi_{K,s}^T \xi_{2s-1} \abs{\cF_K'}^{1-s}\df\nu + o(1). $$
  A similar identity holds for~$\int \vec\vchi \vec\Phi_{K,s}^T \xi$. We deduce
  $$ \vSigma(s) = \frac2{\fd} \int_{K\times \Theta} \big(\vec\Phi_{K,s} \vec\Phi_{K,s}^T + (\vec\vL_s\circ\cF_K) \vec\Phi_{K,s}^T \abs{\cF_K'}^{1-s} +\vec\Phi_{K,s} (\vec\vL_s\circ\cF_K)^T \abs{\cF_K'}^{1-s}\big) \xi_{2s-1} \df\nu+o(1). $$
  Inserting the expression~\eqref{eq:expr-PhiKs-difference} and simplifying, we get
  $$ \vSigma(s) = \frac2{\fd} \int_{K\times \Theta} (\vec\vL_s \vec\vL_s^T - (\vec\vL_s\circ\cF_K) (\vec\vL_s \circ \cF_K)^T \abs{\cF_K'}^{2-2s}) \xi_{2s-1}\df\nu+o(1) , $$
  and finally, since~$\LL_{2s-1}\xi_{2s-1} = \lambda(2s-1) \xi_{2s-1}$ and using~\eqref{cvl} once again, we arrive at
  $$ \vSigma(s) = o(1) + 2\frac{1 - \lambda(2s-1)}{\fd} \int_{K\times\Theta} \vec\vL_s \vec\vL_s^T \xi_{2s-1}\df\nu, \qquad (s\to 1^{+}). $$
  By definition of~$\fd$~\cite[eq.~\eqrefmain{eq:def-fd}]{BettinEtAl}, we have~$1-\lambda(2s-1) \sim 2(s-1)\fd$ as~$s\to 1^+$.
  Thus, we are left to show that
  \begin{equation}
    \int_{K\times\Theta} \vec\vL_s \vec\vL_s^T \xi_{2s-1}\df\nu \sim \frac{\sigma_\phi^2}{4(s-1)} \Id\label{eq:asympt-vL-vLT}
  \end{equation}
  as $s\to1^+$.
  The asymptotic behaviour~\eqref{eq:asympt-vL-vLT} is equivalent to the limits
  \begin{align*}
    & \int_{K\times\Theta} \abs{\vL_s}^2 \xi_{2s-1}\df\nu \sim \frac{\sigma_\phi^2}{2(s-1)},
    & \int_{K\times\Theta} \vL_s^2 \xi_{2s-1}\df\nu = o\big(\frac1{s-1}\big),
  \end{align*}
  where here~$\vL_s(x)$ is interpreted as a complex number. Recall the definition~\eqref{eq:def-vL-subs}. Note that in the first integral, the~$\Theta$ integral trivially evaluates to~$1$ (the integrand is independent of~$\vth$), while in the second integral, we have~$\int_\Theta \vth^2 = \1_{\Theta\subset\{\pm 1\}}$. Changing variables and using the definition~\eqref{eq:defs-L-phigs-Gmodel}, we obtain that it suffices to show
  \begin{align}
    & \int_\R \abs{L(\phi, s-1/2, x)}^2 \zeta_{2s-1}(x) \df x \sim \frac{\sigma_\phi^2}{2(s-1)} \fm, \label{eq:estim-int-fs2-abs}\\
    & \int_\R L(\phi, s-1/2, x)^2 \zeta_{2s-1}(x) \df x = o\big(\frac1{s-1}\big), \label{eq:estim-int-fs2-rel}
  \end{align}
  where we recall $\zeta_{2s-1}(x)$ is as defined in~\eqref{dfnzt}.

  Let~$H\geq 1$ be a parameter.
  By Vaaler approximation~\cite[Theorem~11]{Vaaler1985} (choosing~$\delta \gets H$), there exist two maps~$\zeta_{2s-1}^*$ and~$\zeta_{2s-1}^\sharp$ whose Fourier transforms satisfy
  $$
  \begin{cases}
    \hat\zeta_{2s-1}^*(0) = \int \zeta_{2s-1}, \\
    \hat\zeta_{2s-1}^\sharp(0) \ll 1/H, \\
    \hat\zeta_{2s-1}^*(h), \hat\zeta_{2s-1}^\sharp(h) \ll 1 & (h\in \Z, s\in [1, 2]), \\
    \hat\zeta_{2s-1}^*(h) = \hat\zeta_{2s-1}^\sharp(h)=0 & (\abs{h}>H),
  \end{cases}
  $$
  and
  $$ \abs{\zeta_{2s-1}(x) - \zeta_{2s-1}^*(x)} \leq \zeta_{2s-1}^\sharp(x) $$
  for all~$x\in \R$ except at the finitely many discontinuities of~$\zeta_{2s-1}$.
  We insert this inequality in the integral~$\int_\R \abs{L(\phi, s-1/2, x)}^2 \zeta_{2s-1}(x)\df x$,
  replace~$L(\phi, s-1/2, x)$ by its Fourier expansion~\eqref{eq:def-Lphixs} and interchange sums and integral,
  which is allowed for~$s>1$ by the Wilton bound~\eqref{eq:wilton}.
  We thus obtain
  \begin{equation}
    \int_\R \abs{L(\phi, s-1/2, x)}^2 \zeta_{2s-1}(x) \df x = \Big(\int_\R \zeta_{2s-1} + O(H^{-1})\Big) I_s(0) + O\Big(\sum_{0<\abs{h}\leq H} \abs{I_s(h)} \Big),\label{eq:fourier-aftervaaler}
  \end{equation}
  where
  \begin{equation}
    I_s(h) = \ssum{m, n \geq 1 \\ n-m =h} a(m) \bar{a(n)} (mn)^{1-s}. \label{eq:def-Ish}
  \end{equation}
  Note that~$\int \zeta_{2s-1} = \int_K \abs{p_\H'}^{2(1-s)} \xi_{2s-1} \to \int \xi = \fm$ as~$s\to 1$.
  Letting~$H\to \infty$ sufficiently slowly as~$s\to 1^+$, we see that to prove the asymptotic behaviour~\eqref{eq:estim-int-fs2-abs} it suffices to show that
  \begin{align}
    I_s(0) \sim{}& \frac{\sigma_\phi^2}{2(s-1)}, \label{eq:aim-Ish-diag} \\
    I_s(h) \ll_h{}& 1, \label{eq:aim-Ish-offdiag}
  \end{align}
  where the second bound is uniform in~$s\in(1, 1+\delta]$. We first focus on the case when~$s_\phi \not\equiv k/2\pmod{1}$, which implies (since~$\phi$ is cuspidal) that~$s_\phi \in [1/2,1) \cup (\frac12+i\R)$.

  For~$h=0$ we have
  $$ I_s(0) = \sum_{n\geq 1} \frac{\abs{a(n)}^2}{n^{2(s-1)}}. $$
  We apply the Rankin--Selberg method~\cite[eq.~(8.10)]{Iwaniec1995}. Define
  $$ I^\pm(s) := \sum_{\sign(n)=\pm} \frac{\abs{a(n)}^2}{\abs{n}^{2(s-1)}}, $$
  so that~$I_s(0) = I^+(s)$. Define further, for all~$\ell\in \Z$,
  \begin{align}\label{dfcl}
    C_\ell(s) := \int_0^\infty W_{\ell/2, it_\phi}(y)^2 y^{2(s-1)} \frac{\df y}{y}.
  \end{align}
  By the Rankin--Selberg method, we have
  \begin{equation}
    C_k(s) I^+(s) + C_{-k}(s) I^-(s) = (4\pi)^{2(s-1)} J_\phi(s) \label{eq:combin-Ipm-phi}
  \end{equation}
  where
  $$ J_\phi(s) := \int_{\Gamma\backslash\H} \abs{\phi(z)}^2 E(2s-1, z) \frac{\df x \df y}{y^2} $$
  and~$E(\cdot,z)$ is the standard Eisenstein series at~$\infty$. On the other hand, let~$\psi = R_k \phi$ be the weight-raised form obtained from~$\phi$. It is of weight~$k+2$ and has a Fourier expansion
  $$ \psi(z) = \sum_{n\neq 0} a_\psi(n) \e(n\Re(z)) W_{\sign(n) \frac{k+2}2, it_\phi}(4\pi \abs{n} \Im(z)), $$
  where
  $$ a_\psi(n) = \begin{cases} -a(n) & (n\geq 1), \\ (s_\phi + \frac k2)(1-s_\phi + \frac k2) a(n) & (n\leq -1). \end{cases} $$
  By the Rankin--Selberg method again, we find
  \begin{equation}
    C_{k+2}(s) I^+(s) + \abs{(s_\phi + \tfrac k2)(1-s_\phi + \tfrac k2)}^2 C_{-k-2}(s) I^-(s) = (4\pi)^{2(s-1)} J_\psi(s).\label{eq:combin-Ipm-psi}
  \end{equation}
  The maps~$C_\ell$ are regular at~$s=1$, since~$W_{\ell/2, it_\phi}(y) \ll y^{1 - \Re(s_\phi) - \eps}$ in the neighbourhood of~$0$.
  The value~$C_\ell(1)$ is computed using the formula in~\cite[(7.611.4)]{GradshteynRyzhik2007}\footnote{The related formula~\cite[eq.~(19.6)]{DukeEtAl2002} should be corrected by a non-constant factor when~$k\neq 0$. This is due to the fact that~\cite[(3.384.9)]{GradshteynRyzhik2007}, which is used there, is not immediately applicable when~$\Re(s_\phi)=1/2$ due to lack of absolute convergence. We use instead~\cite[(7.611.4)]{GradshteynRyzhik2007} which is proven with the arguments described in~\cite[pp.~112--116]{Buchholz1953}.}. We obtain
    \begin{equation}
      C_\ell(1) = \frac{-\pi}{\sin(2\pi s_\phi)} \frac{\Psi(\ell)}{G(\ell)}, \qquad \begin{cases} \Psi(\ell) = \psi_E(s_\phi - \ell/2) - \psi_E(1-s_\phi - \ell/2) \\ G(\ell) = \Gamma(s_\phi - \ell/2) \Gamma(1-s_\phi-\ell/2) \end{cases}
    \end{equation}
    where~$\psi_E$ is the Euler psi function~\cite[Section~8.36]{GradshteynRyzhik2007}; this is valid as written for~$s_\phi \neq \tfrac12$, and is extended by continuity at~$s_\phi = \tfrac12$.

    Let~$\mu_k = (s_\phi + \frac k2)(1 - s_\phi + \frac k2)$.
    We solve the equations~\eqref{eq:combin-Ipm-phi} and \eqref{eq:combin-Ipm-psi} for~$I^+(s)$. This is possible if we prove that the determinant
    $$ \Delta_k(s) = \mu_k^2 C_k(s) C_{-k-2}(s) - C_{-k}(s) C_{k+2}(s) $$
    is non-zero. By continuity it suffices to prove this at~$s=1$. Using~$\mu_k = G(-k-2)/G(-k) = G(k) / G(k+2)$, we get
    $$ \Delta_k(1) = \frac{\pi^2}{\sin(2\pi s_\phi)^2} \frac{\mu_k}{G(k)G(-k)} \big( \Psi(k) \Psi(-k-2) - \Psi(-k) \Psi(k+2) \big), $$
    and since~$\Psi(k+2) - \Psi(k) = (2s_\phi-1)/\mu_k = -\Psi(-k-2) + \Psi(-k)$, we deduce
    $$ \Delta_k(1) = \frac{\pi^2(1-2s_\phi)}{\sin(2\pi s_\phi)^2} \frac{\Psi(k) + \Psi(-k)}{G(k)G(-k)}, $$
    and finally, since
    \begin{align*}
      \frac{\Psi(k) + \Psi(-k)}{G(k)G(-k)} {}&= -\tfrac1\pi \sin(\pi(s_\phi + \tfrac k2)) \sin(\pi(s_\phi - \tfrac k2))\big(\cot(\pi(s_\phi + \tfrac k2)) + \cot(\pi(s_\phi - \tfrac k2))\big) \\
      {}&= -\tfrac1\pi \sin(2\pi s_\phi), 
    \end{align*}
    we conclude
    \begin{equation}
      \label{eq:delta1-closedform}
      \Delta_k(1) = \frac{\pi(2s_\phi - 1)}{\sin(2\pi s_\phi)}
    \end{equation}
    independently of~$k$. We deduce
  \begin{equation}
    \label{eq:rel-Iplus-J}
    I^+(s) = \frac{(4\pi)^{2(s-1)}}{\Delta_k(s)}(\mu_k^2 C_{-k-2}(s) J_\phi(s) - C_{-k}(s) J_\psi(s)).
  \end{equation}
  By~\cite[Proposition~6.13]{Iwaniec2002}, we have
  \begin{equation}
    \label{eq:rankinselberg-Jphi}
    J_\phi(s) \sim \frac{1}{2(s-1)} \frac{\|\phi\|^2}{\Vol(\Gamma\backslash \H)},
  \end{equation}
  and similarly for~$\psi$. Therefore
  $$ I^+(s) \sim \frac1{2(s-1)} \frac{\mu_k^2 C_{-k-2}(1)\|\phi\|^2 - C_{-k}(1)\|\psi\|^2}{\Delta_k(1)\Vol(\Gamma\backslash \H)}. $$
  By~\cite[eq.~(4.3)]{DukeEtAl2002}, we have~$\|\psi\|^2 = \mu_k \|\phi\|^2$, and thus
  $$ I_s(0) = I^+(s) \sim \frac1{2(s-1)} \frac{\mu_k(\mu_k C_{-k-2}(1) - C_{-k}(1))}{\Delta_k(1)} \frac{\|\phi\|^2}{\Vol(\Gamma\backslash \H)}. $$
  Finally, we compute, by a method similar to~\eqref{eq:delta1-closedform},
  $$ \frac{\mu_k(\mu_k C_{-k-2}(1) - C_{-k}(1))}{\Delta_k(1)} = \frac{\mu_k}{(2s_\phi-1)} \frac{\Psi(-k) - \Psi(-k-2)}{G(-k)} = \frac1{G(-k)} $$
  which is~$(\Gamma(s_\phi + \frac k2) \Gamma(1 - s_\phi + \frac k2))^{-1}$ as required.
  This shows~\eqref{eq:aim-Ish-diag} with the claimed value of~$\sigma_\phi^2$.

  We next upper-bound~$I_s(h)$ for~$h\neq 0$. By conjugation we may assume~$h>0$.
  We start by considering
  \begin{align*}
    \tilde I_s(h):=\ssum{m, n \in \Z_{\neq 0} \\ n - m = h} a(m) \bar{a(n)} b_h(m, n),
  \end{align*}
  where for~$m, n\in\Z_{\neq 0}$, we let
  $$ b_h(m, n) = \int_0^\infty W_{\sign(m) k/2, it_\phi}(4\pi \abs{m} y) \bar{W_{\sign(n)k/2, it_\phi}(4\pi \abs{n} y)} \e^{-2\pi h y} y^{2(s-1)} \frac{\df y}y. $$
  For~$n - m = h$, we have
  \begin{align}
    b_h(m, n) ={}& \frac1{\abs{n}^{2(s-1)}} \int_0^\infty W_{\sign(m) k/2, it_\phi}(4\pi \abs{\tfrac mn} y) \bar{W_{\sign(n)k/2, it_\phi}(4\pi y)} y^{2(s-1)} \e^{-2\pi h y/\abs{n}} \frac{\df y}y \notag \\
    ={}& \Big( \int_0^\infty W_{\sign(m) k/2, it_\phi}(4\pi y)\bar{W_{\sign(n) k/2, it_\phi}(4\pi y)} y^{2(s-1)} \frac{\df y}y + O_h\Big(\frac1{\abs{n}}\Big)\Big) \frac1{\abs{n}^{2(s-1)}} \notag\\
    ={}& \Big(\int_0^\infty W_{\sign(m) k/2, it_\phi}(4\pi y)\bar{W_{\sign(n) k/2, it_\phi}(4\pi y)} y^{2(s-1)} \frac{\df y}y + O_h\Big(\frac1{\abs{n}}\Big)\Big) \frac1{\abs{mn}^{s-1}}.\label{lloftb}
  \end{align}
  Here in the second line we used the expression~\cite[eq.~(4.24)]{DukeEtAl2002} to bound the derivative of the Whittaker function. We deduce that~$b_h(m, n) \ll_h 1$.
  Also, if~$\sign(m)=\sign(n)=\pm 1$ by the definition~\eqref{dfcl} the integral in~\eqref{lloftb} can be expressed in terms of $C_{\pm k}$ giving
  \begin{align}\label{csac}
    b_h(m, n) =\Big(C_{\pm k}(s) + O_h\Big(\frac1{\abs{n}}\Big)\Big) \frac1{\abs{mn}^{s-1}}(4\pi)^{2(1-s)}.
  \end{align}
  Since there are finitely many indices~$m, n$ with~$\sign(m) \neq \sign(n)$ (depending on~$h$), the contribution to $\tilde I_s(h)$ of these terms is $O_h(1)$. Thus, using~\eqref{csac} for the remaining terms, we obtain
  \begin{align}\label{oeis}
    \tilde I_s(h) &= (4\pi)^{2(1-s)}(C_{k}(s) I_s^+(h) + C_{-k}(s) I_s^-(h)) +O_h(1),
  \end{align}
  with
  $$ I_s^\pm(h) := \ssum{m, n \in \Z \\ \sign(m)=\sign(n)=\pm1 \\ n-m=h} a(m) \bar{a(n)} \abs{mn}^{1-s} $$
  since the contribution of the term $O_h(1/\abs{n})$ is bounded by a constant times
  $$ \sum_{n\neq 0} \frac{\abs{a(n)}}{\abs{n}} \sum_{n - m = h} \abs{a(m)} \ll_h \sum_{n\neq 0} \frac{\abs{a(n)}}{\abs{n}^{1-\eps}} \ll 1 $$
  by the trivial bound~$\abs{a(m)} \ll m^\eps$ and~\eqref{eq:bound-meanabs-an}.

  We can also give a different expression for $\tilde I_s(h)$ by inserting the definition of $b_h$ and exchanging the order of summation and integration. We obtain
  \begin{align*}
    \tilde I_s(h)
    ={}& \int_0^\infty \ssum{m, n \\ n - m = h} a(m) \bar{a(n)} W_{\sign(m)k/2, it_\phi}(4\pi \abs{m} y) \bar{W_{\sign(n)k/2, it_\phi}(4\pi \abs{n} y)} \e^{-2\pi h y} y^{2(s-1)} \frac{\df y}y \\
    ={}& \int_{\Gamma \backslash \H} P_h(z, 2s-1) \abs{\phi(z)}^2 \frac{\df x\df y}{y^2},
  \end{align*}
  where~$P_h(z, 2s-1)$ is the Poincaré series
  $$ P_h(z, 2s-1) = \sum_{\gamma \in \Gamma_\infty \backslash \Gamma} (\Im \gamma z)^{2s-1} \e(h \gamma z). $$
  It is known~\cite{Selberg1965,GoldfeldSarnak1983} that~$s\mapsto P_h(z, 2s-1)$ for~$h\neq 0$ extends to an analytic function in a neighbourhood of~$\{\Re(s)\geq 1\}$ and satisfies~$\int_{\Gamma\backslash \H} \abs{P_h(z, 2s-1)}^2 \df x \df y / y^2 \ll_h 1$ uniformly for~$s\in\cV$. Since~$\phi(z) \ll 1$, we deduce by Cauchy--Schwarz that
  $$ \tilde I_s(h) = \int_{\Gamma \backslash \H} P_h(z, 2s-1) \abs{\phi(z)}^2 \frac{\df x\df y}{y^2} \ll_h 1, $$
  so that, by~\eqref{oeis},
  $$ C_{k}(s) I_s^+(h) + C_{-k}(s) I_s^-(h) = O_h(1) $$
  as~$s\to 1^+$. We carry out a similar argument for the weight-raised form~$\psi = R_k\phi$, and deduce by Cramer's rule as above that
  $$ I_s^+(h) = O_h(1), \qquad (s\to 1^+). $$
  This proves~\eqref{eq:aim-Ish-offdiag}, and completes the proof of~\eqref{eq:estim-int-fs2-abs}.

  A similar argument allows us to estimate~$\int_{\R} L(\phi, s-1/2, x)^2 \zeta_{2s-1}(x) \df x$. The analogue of~\eqref{eq:fourier-aftervaaler} is now
  $$ \int_{\R} L(\phi, s-1/2, x)^2 \zeta_{2s-1}(x) \df x \ll H^{-1} I_s(0) + \sum_{0<\abs{h}\leq H} (\abs{I_s(h)} + \abs{J_s(h)}), $$
  where
  $$ J_s(h) := \ssum{m, n\geq 1 \\ m+n = h} \frac{a(m) {a(n)}}{(mn)^{s-1}}. $$
  In this sum, since~$m, n$ have the same sign, they are obviously bounded in terms of~$h$, so that~$J_s(h) \ll_h 1$ for all~$h$. Letting~$H\to \infty$ slowly enough as~$s\to 1^+$ gives~\eqref{eq:estim-int-fs2-rel} as claimed, and concludes the proof of Proposition~\ref{prop:variance}.

Finally we discuss the case when~$y^{-k/2} \phi(z)$ is holomorphic, which implies~$s_\phi = k/2$ and~$a(n) = 0$ for~$n<0$. We then have~$W_{k/2, (k-1)/2}(y) = y^{k/2}\e^{-y/2}$ and thus~$C_k(s) = \Gamma(k + 2(s-1))$. The analogue of~\eqref{eq:rel-Iplus-J} is simply
    $$ I^+(s) = \frac{(4\pi)^{2(s-1)}}{\Gamma(k+2(s-1))} J_\phi(s). $$
    Inserting the estimate~\eqref{eq:rankinselberg-Jphi} we obtain~\eqref{eq:aim-Ish-diag} also with the required value. A similar argument shows that~$I_s(h) = O_h(1)$ for~$h\neq 0$ and~$J_s(h) = O_h(1)$ for all~$h$, concluding the proof.

\end{proof}

\section{Topological irreducibility with multiplier systems}\label{sec:top-irred}

We recall our assumptions on~$\Gamma$ and~$G$, which follow from our setting in Section~\ref{sec:automorphic-forms} as well as Lemma~\ref{lem:free-subgroup}: $G$ is a free subgroup of~$[\Gamma]_\D$, it is discrete, finitely generated with finitely many cusps; it has a cusp at~$1$, and has rank at least~$6$.
Recall that $\vth_g$ was defined in Section~\ref{vtdf}, and recall also the choice of fundamental domain made in Section~\ref{tos}.

We quote from~\cite[eq.~\eqrefmain{eq:def-ellK}]{BettinEtAl} the notion of \emph{$K$-length} $\ell_K$.
If $w=a_1\dotsb a_\ell\in\Sigma^*$ with $\ell\geq1$, then
\begin{align*}
  \ell_K(w) &:=1+ \abs{\{1\leq i \leq \ell-1,\ a_{i+1}-\hat{a_i}\neq\pm1\} }.
\end{align*}

The goal of this section is to prove the following proposition, which shows that the multiplier system~$\vth_g(x)$ satisfies the stability and mixing hypothesis from~\cite[Section~\refmain{sec:skew-products}]{BettinEtAl}, or in other words, that $\vth_g(x)$ is irreducible with respect to the induced map $\cF_K$ (see~\cite[Section~\refmain{ssec:bs:accel}]{BettinEtAl}). We shall use the fact that the fundamental domain~$R$ has at least~$8$ sides, as implied by Lemma~\ref{lem:free-subgroup}.

\begin{proposition}\label{prop:topological-irred}
  Suppose~$\abs{\Sigma} \geq 8$. For any admissible non-empty words~$w_1, w_3\in \Sigma^*$, and any~$\vth_0 \in \Theta$, there exists~$\ell_0\geq 1$ such that for all~$\ell\geq \ell_0$, we may find~$w_2\in \Sigma^*$ with the following properties:
  \begin{itemize}
    \item The word~$w_1w_2w_3$ is admissible, and~$\ell_K(w_1w_2w_3) = \ell$.
    \item We have~$\vth_{g_{w_1 w_2}}(x) = \vth_0$ for all $x\in[w_3]^\circ$.
  \end{itemize}
\end{proposition}

We first need some lemmas.
In the following we denote by~$a_+$ the side of $R$ which precedes~$1$ proceeding counter-clockwise, and~$a_-$ the side which follows~$1$.

\begin{center}
  \begin{tikzpicture}[scale=1.4]
    \draw (0,0) circle (1);
    \interval{1}{0}{50} node at (1.2, -0.65) {$[a_+]$};
    \interval{1}{-70}{0} node at (1.25, 0.5) {$[a_-]$};
    \intervalpartial{1}{50}{100}{0.5};
    \intervalpartial{1}{-70}{240}{0.5};
    \fill (1,0) circle (0.015) node[right] {$1$};
  \end{tikzpicture}
\end{center}

\begin{lemma}\label{lem:multiplier}
  Let~$w = a_1 \dotsb a_\ell \in \Sigma^*$ with~$\ell\geq 1$. Then the map
  $$ x \mapsto \vth_{g_w}(x) $$
  is constant on each arc~$[b]^\circ$, for $b\neq \hat{a_\ell}$. Denote by
  $$ \vth_w(b) = \vth_{g_w}(x), \qquad (x\in [b]^\circ) $$
  the associated common value. The following properties hold.
  \begin{enumerate}[label=(\roman*)]
    \item\label{it:vth-stab} The quantity~$\vth_{\hat{a_-}}(b)$ is independent of~$b$ for~$b\neq a_-$.
    \item\label{it:vth-sign} Whenever~$a_1, \hat{a_\ell} \neq a_+$ and~$g_w\cdot 1 \neq 1$, we have~$\vth_w(a_+) = (-1)^{-k} \vth_{\hat w}(a_+)^{-1}$.
    % \item For any~$\eta \in \Theta$, we may find~$w\in \Sigma^*$ such that~$\vth_w(b) \in \{\pm \eta\}$ for any~$b\neq \hat{a_\ell}$.
    % \item Whenever~$a_\ell \neq \hat{a_1}$, we can find~$n\geq 1$ such that~$\vth_{w^n}(b) \in \{\pm 1\}$ for~$b\neq \hat{a_\ell}$.
  \end{enumerate}
\end{lemma}
\begin{proof}
  First we observe that if~$g_w$ fixes $1$ (\emph{i.e.}~$[g_w]_\H$ fixes $\infty$), then~$\vth_{g_w}(x) = d_w^k\chi'(g_w)$ for some $d_w\in\{\pm1\}$ and all the claimed results follow immediately. Thus we can assume $g_w\cdot 1 \neq 1$. In particular, we can write $\vth_{g_w}$ as $\vth_{g_w}(x)=c_w^{-k}\chi'(g_w)\sg_{w}(x)^k$ for some $c_w\in\{\pm1\}$ satisfying $c_w=-c_{\hat w}$ and where $\sg_{w}(x)=\sign([x]_\H-[g_w]_\H^{-1}\infty)$. Notice that $\sg_{w}(x)=1$ if $g_w^{-1}\cdot 1$ precedes $x$ proceeding counter-clockwise from $1$ and $\sg_{w}(x)=-1$ otherwise.

  Since~$1$ is a vertex of~$R$, and thus is not in $[a_1]^\circ$, we have~$g_w^{-1}\cdot 1 \in [\hat{a_\ell}]$. Thus for~$x\in [b]^\circ$ and~$b\neq \hat{a_\ell}$ we have that $\sg_{w}(x)$ depends only on~$b$ and not on~$x$. This justifies the definition of~$\vth_w(b)$.
  \begin{enumerate}[label=(\roman*)]
    \item If~$g_{\hat{a_-}}^{-1}\cdot1 = 1$, then $g_{\hat{a_-}}$ fixes~$1$ and the claim is clear as above.
    Now, assume $g_{\hat{a_-}}^{-1}\cdot1\neq 1$. Since~$1$ is a vertex of~$R$,
    we have $g_{\hat{a_-}}^{-1}\cdot1=g_{a_-}\cdot 1\in [a_-]$.
    This implies that $\sg_{\hat{a_-}}(x)=1$ for all $x\notin [a_-]$ and the claim follows.
    \item We have $g_w^{-1}\cdot 1\in [\hat{a_\ell}]$ and thus, by hypothesis, $g_w^{-1}\cdot 1\not\in[a_+]^\circ$ so that $\sg_w(x)=1$ for all $x\in[a_+]^\circ$. In the same way, $\sg_{\hat w}(x)=1$ for all $x\in[a_+]^\circ$ and the result follows since $c_w=-c_{\hat w}$.
    % \item By definition of~$\Theta$, there exists~$g\in G$ such that~$\chi'(g)\in \{\pm \eta\}$. Write~$g = g_w$ for some admissible word~$w\in \Sigma^*$. Then~$\vth_{w}(b) \in\{ \pm \chi'(g) \}= \{\pm \eta\}$ as claimed.
    % \item We have~$\vth_{w^n}(b) \in \{\pm \chi'(g_w)^n\}$. Since~$\chi'$ takes values in the roots of unity of some fixed order, it suffices to pick~$n$ to be the common order of these roots.\qedhere
  \end{enumerate}
\end{proof}

Given $w_1\in\Sigma^*$ we extend the definition of $\vth_{w_1}$ also to words $w_2\in\Sigma^*$, with $w_1w_2$ admissible, by letting $\vth_{w_1}(w_2):=\vth_{w_1}(a)$ where $a$ is the first letter of $w_2$.
The relation~\eqref{pere} then reads
\begin{equation}
  \label{eq:pere-vth}
  \vth_{w_1w_2}(w_3) = \vth_{w_1}(w_2) \vth_{w_2}(w_3)
\end{equation}
whenever~$w_1w_2w_3 \in \Sigma^*$ with~$w_i$ not empty.

We need some further notation. For any~$s_1, s_2\in \Sigma$ with~$s_1$ adjacent to~$\hat{s_2}$, we let $P_{s_1, s_2}\in\Sigma^*$ be a (fixed) cuspidal cycle --- as defined in~\eqref{ccy} --- with first letter~$s_1$ and last letter~$s_2$. Upon replacing~$P_{s_1, s_2}$ by a concatenation power~$P_{s_1, s_2}^r$ if necessary, we can impose that~$\chi'(g_{P_{s_1, s_2}}) = 1$.

For~$b_1, b_2\in \Sigma$, we denote
$$ E(b_1, b_2) = \Big\{b\in \Sigma, \quad \begin{aligned}&\hat{b_1} \text{ neither equal nor adjacent to } b, \\
  &\text{and } \hat{b} \text{ neither equal nor adjacent to } b_2. \end{aligned} \Big\} $$
the set of letters~$b$ for which both transitions~$b_1b$ and~$bb_2$ are admissible and counted by~$\ell_K$.
Note that since we assumed $|\Sigma|\geq8$ we have~$\abs{E(b_1, b_2)} \geq 8 - 3 - 3 \geq 2$. In particular,~$E(b_1, b_2)$ is never empty.

In the rest of the section we will repeatedly make use of the following simple facts about the function~$\ell_K$: let~$w_1$ and~$w_2$ be two non-empty admissible words. Assume that the last letter~$c_1$ of~$w_1$, and the first letter~$c_2$ of~$w_2$, are such that~$\hat{c_1}$ and~$c_2$ are neither equal nor contiguous. Then
\begin{equation}
  \ell_K(w_1w_2) = \ell_K(w_1) + \ell_K(w_2).\label{eq:additivity-flips}
\end{equation}
Also, for $s_1,s_2$ as above we have $\ell_K(P_{s_1, s_2})=1$.

\begin{lemma}\label{lem:top-irred-sign}
  There exists~$b\in \Sigma$ such that for any non-empty~$w_3\in \Sigma^*$, there is~$\vth_0\in \Theta$ such that for any choice of~$\eta\in \{(\pm 1)^{-k}\}$, we may find~$u_2\in \Sigma^*$ with $bu_2w_3$ admissible and
  \begin{equation}\label{eq:top-irred-sign}
    \vth_{bu_2}(w_3) = \eta \vth_0, \qquad \ell_K(bu_2w_3) = \ell_K(w_3) + 4.
  \end{equation}
\end{lemma}
\begin{proof}
  We pick~$u\in\Sigma$ distinct from the letters
  $$ \{ a_+ - 2,a_+-1, a_+, a_-, a_-+1\}. $$

  Since we assumed~$\abs{\Sigma} \geq 8 > 5$, there is at least one such choice of~$u$. Let~$v = \hat{u+1}$, so that there is a cuspidal cycle~$P_{u, v}$. Let
  $$ P \in \{ P_{u,v}, \hat{P_{u,v}} \}, $$
  and fix~$b'\in E(a_+, c)$ with $c$ denoting the first letter of $w_3$. Let
  $$ b := \hat{a_-}, \qquad u_2 = P a_+ b'\in\Sigma^*.$$
  By the period relation~\eqref{eq:pere-vth} we have
  $$ \vth_{bu_2}(w_3) = \vth_{\hat{a_-}}(P) \vth_P(a_+) \vth_{a_+}(b') \vth_{b'}(c) = \vth_0 \vth_P(a_+) $$
  where by Lemma~\ref{lem:multiplier}.\ref{it:vth-stab} the number~$\vth_0\in \Theta$ is independent of our choice of~$P$.
  By Lemma~\ref{lem:multiplier}.\ref{it:vth-sign}, we have
  $$ \vth_{P}(a_+) = (-1)^{-k} \vth_{\hat{P}}(a_+), $$
  thus picking appropriately~$P$, we find a corresponding word~$u_2$ with~$\vth_{bu_2}(w_3) = \eta \vth_0$.
  Moreover, a quick computation using the hypothesis on $u$ and $b'$ shows that
  $$ \ell_K(bu_2w_3) = \ell_K(\hat{a_-} P a_+ b' w_3) = 4 + \ell_K(w_3). $$
  This proves~\eqref{eq:top-irred-sign}.
\end{proof}

\begin{lemma}\label{lem:top-irred-U}
  For any non-empty~$w_1\in \Sigma^*$, any letter~$b\in \Sigma$, and any~$\vth_0\in \Theta$, there exists~$\ell_0'\in \N$ such that for any~$\ell'\geq \ell_0'$, we may find a word~$u_1\in\Sigma^*$ with $w_1u_1b$ admissible and
  $$ \vth_{w_1u_1}(b) \in \{(\pm 1)^k \vth_0\}, \qquad \ell_K(w_1u_1b) = \ell'. $$
\end{lemma}
\begin{proof}
  By the definition of $\Theta$ there exists~$g\in G$ such that~$\chi'(g_{w_1}g) \in\{(\pm 1)^k \vth_0 \}$. Since~$G$ is generated by~$\{g_d, d\in \Sigma\}$ (see Section~\ref{sec:operator-cc}), we may write~$g = g_{d_1} \dotsb g_{d_q}$ for some non-negative integer~$q$ and~$d_j\in \Sigma$. We may assume~$d_i \neq \hat{d_j}$ for all~$i\neq j$, since such pairs can be removed without altering~$\chi'(g)$. Then, denoting by $b_0$ the last letter of $w_1$, we define $T\in\Sigma^*$ as follows:
  \begin{itemize}
    \item If~$q=0$ or if there exists~$j$ with~$d_j \neq \hat{b_0}$, in which case by reordering we may assume~$j=1$, we let~$T = d_1 \dotsb d_q$.
    \item Otherwise we have~$q\geq 1$ and~$g = g_{\hat{b_0}^q}$. We then pick~$e\in \Sigma \setminus\{b_0, \hat{b_0}\}$ and let~$T = e \hat{b_0}^q \hat e$.
  \end{itemize}
  Note that~$T$ is the empty word if~$q=0$.
  In any case, the word~$w_1 T$ is admissible and~$\chi'(g_{w_1} g_T) \in\{(\pm 1)^k \vth_0 \}$. We let~$\ell_0' := \ell_K(w_1 T)+2$, and consider~$\ell'\geq \ell_0'$ arbitrary.
  We let~$r := \ell' - \ell_K(w_1 T)-1 \geq 1$ and we pick~$a_1, \dotsc, a_{r}\in\Sigma$ subject to the following constraints:
  \begin{itemize}
    \item $a_1$ is neither equal nor adjacent to~$\hat c$, where~$c$ is the last letter of~$w_1T$,
    \item for all~$j\in \{2, \dotsc, r\}$, $a_j$ is neither equal nor adjacent to~$a_{j-1}+1$,
    \item $b$ is neither equal nor adjacent to~$a_r+1$.
  \end{itemize}
  The condition on~$a_j$ excludes~$6$ letters for~$j=r$ and~$3$ letters for~$j<r$, so that~$\abs{\Sigma}\geq 8$ ensures the existence of~$a_1, \dotsc, a_r$.
  For each~$j$, we consider the cuspidal cycles~$P_j := P_{a_j, \hat{a_j+1}}$.
  We let
  $$ u_1 = T P_1 \dotsb P_{r}\in\Sigma^*. $$
  By construction, we have~$\chi'(g_{w_1u_1}) \in\{(\pm 1)^k \vth_0 \}$, and therefore
  $$ \vth_{w_1u_1}(b) \in \{(\pm 1)^k \vth_0\}. $$
  Moreover, we have
  $$ \ell_K(w_1 u_1 b) = \ell_K(w_1 T P_{1} \dotsb P_{r} b) = \ell_K(w_1 T) + r+1 = \ell', $$
  as claimed.
\end{proof}

\begin{proof}[Proof of Proposition~\ref{prop:topological-irred}]
  Let~$w_1, w_3$ be two given admissible words. We apply Lemma~\ref{lem:top-irred-sign} for~$w_3$, and let~$b$ and~$\vth_1$ be the quantities whose existence is asserted. We then apply Lemma~\ref{lem:top-irred-U} for the triple~$(w_1, b, \vth_0\vth_1^{-1})$. Note that~$\ell_0'$ depends on~$w_1$, $b$ and~$\vth_0\vth_1^{-1}$. Since~$\Sigma$ and~$\Theta$ are finite, we can pick~$\ell_0 = \ell_0(w_1, w_3)$ large enough so that~$\ell_0 \geq \ell_0' + \ell_K(w_3)+3$. Then, Lemma~\ref{lem:top-irred-U} applied to $\ell'=\ell - \ell_K(w_3) -3\geq \ell_0'$ yields the existence of a word $u_1\in\Sigma^*$ with~$w_1u_1b$ admissible and a sign~$\eta\in \{(\pm 1)^k\}$ such that
  $$ \vth_{w_1u_1}(b) =\eta \vth_0\vth_1^{-1}, \qquad \ell_K(w_1u_1b) = \ell - \ell_K(w_3) -3.$$
  Since $b$ and~$\vth_1$ are as given in Lemma~\ref{lem:top-irred-sign}, by the conclusion of the lemma we may find~$u_2\in\Sigma^*$ for which~$bu_2w_3$ is admissible and
  $$ \vth_{bu_2}(w_3) = \eta^{-1} \vth_1, \qquad \ell_K(bu_2w_3) = \ell_K(w_3) + 4. $$
  We let~$w_2 = u_1 b u_2$. The word~$w_1w_2w_3$ is admissible since~$w_1u_1b$ and~$bu_2w_3$ both are. We readily get by~\eqref{eq:pere-vth}
  $$ \vth_{w_1w_2}(w_3) = \vth_{w_1u_1}(b) \vth_{bu_2}(w_3) = \vth_0, $$
  and check that
  $$ \ell_K(w_1w_2w_3) = \ell_K(w_1u_1b) + \ell_K(bu_2w_3) - 1 = \ell, $$
  as required.
\end{proof}

\begin{remark}
  The condition~$\abs{\Sigma} \geq 8$ in Proposition~\ref{prop:topological-irred} cannot be completely removed, as the example of the group~$\Gamma = \Gamma_1(4) = \{\gamma \equiv \smat{1&\ast\\0&1} \pmod{4}\}$ shows: using the disk model as in Section~\ref{sec:disk-model} we may choose the following fundamental domain for $[\Gamma]_\D \backslash \D$:
  \begin{center}
    \begin{tikzpicture}[scale=0.65]
      \draw (0, 0) circle (3);
      \interval{3}{0}{126.87} node at (3, 2) {$[a]$};
      \interval{3}{233.13}{360} node at (3, -3) {$[\hat a]$};
      \interval{3}{126.87}{180} node at (-3.3, 1.5) {$[b]$};
      \interval{3}{180}{233.13} node at (-3.3, -1.5) {$[\hat b]$};
      \draw (3, 0) node[right] {$1$};
      \draw (-1.2, 2.45) node[above left ]{$p_\D(-\frac12)$};
      \draw (-1.2, -2.45) node[below left] {$p_\D(\frac12)$};
      \draw (-3, 0) node[left] {$-1$};
    \end{tikzpicture}
  \end{center}
  It is easily seen that for any non-empty word~$w = a_1 \dotsb a_n\in \Sigma^*$, $\ell_K(w)$ is the number of blocks in an alternating decomposition of $w$ into subwords from~$\{a, \hat b\}^*$ and~$\{\hat a, b\}^*$.
  This implies, for instance, that any word~$u\in \Sigma^*$ with~$aua$ admissible must have~$\ell_K(aua)$ odd, so the conclusion of Proposition~\ref{prop:topological-irred} does not hold for $w_1=w_3=a$ and $\ell$ even. Numerically, it seems that there is no such objection when~$\abs{\Sigma} = 6$, which leaves open the question as to whether~$\abs{\Sigma}\geq 6$ is a sufficient condition in Proposition~\ref{prop:topological-irred}.
\end{remark}

\begin{remark}\label{rmk:ellc-mixing-gamma04}
  Consider the group~$\Gamma = \Gamma_1(4)$ equipped with the multiplier system given by~$u_\gamma(x) = \sign(cx + d)$.
  We have observed numerically that the conclusion of Proposition~\ref{prop:topological-irred} holds for the first few values of~$\ell$, provided that we replace~$\cF_K$ by~$\cF_c$, the cuspidal acceleration defined in~\cite[Section~\refmain{sec:cusp-acceleration}]{BettinEtAl} (for which~$\CC$ in~\eqref{eq:def-CC} is the transfer operator). The analogue of~$\ell_K$ is the number of cuspidal flips~$\ell_c$, defined in~\cite[eq.~\eqrefmain{eq:def-ellc}]{BettinEtAl}. This provides evidence that the cuspidal acceleration~$\cF_c$ is topologically mixing in this case.
\end{remark}

\section{Theta series}\label{sec:theta}
In this section we explain the changes needed to the above argument to prove~Theorem~\ref{thethe}.

Let $\chi$ be a primitive Dirichlet character modulo $N$ and let $\nu\in\{0,1\}$ be such that $\chi(-1)=(-1)^\nu$. Let $\theta_\chi$ be defined by
\begin{align}
  \theta_\chi (z)=y^{(1+2\nu)/4}\sum_{n\in\Z} n^\nu \chi(n)\e(n^2z),\qquad z\in\H.
\end{align}
For each $\gamma=(\smat{ a & b \\ c&d })\in\Gamma_0(4N^2)$ and $z\in\H$ we
let
\begin{align}
  u_\gamma(z)=\eps_d^{-1}\Big(\frac cd\Big)\chi(\gamma)j_\gamma(z)^{1/2+\nu},
\end{align}
where $\chi(\gamma):=\chi(d)$, $(\frac cd)$ is the Kronecker symbol,
$$\eps_d=\Big(\frac{-1}{d}\Big)^{1/2}=\begin{cases}
  1 &d\equiv 1\mod 4,\\
  i & d\equiv -1\mod 4,
\end{cases}
$$
and where the square root takes its principal value, i.e. its argument is in $(-\pi/2,\pi/2]$.
Then, by~\cite[Proposition 2.2]{Shimura1973} we have that $\theta_\chi$ is a Maa\ss{} form of weight $k=\frac12+\nu$ for $\Gamma_0(4N^2)$, that is
\begin{align}
  \theta_\chi(\gamma z)=u_\gamma (z) \theta_\chi(z)\qquad\forall \gamma\in\Gamma_0(4N^2).
\end{align}
Also, by~\cite[p.~36]{SerreStark1977} and~\cite[(10.54) and Corollary~10.7]{Iwaniec1997} we have that if $\chi$ is not totally even, then $\theta_\chi$ is cuspidal, whereas if $\chi$ is totally even, then the constant term in the Fourier expansion of $\theta_\chi$ at a cusp $\cb$ is of the shape $\theta_{\chi,\cb}(y)=A_{\cb}y^{1/4}$
(and in particular $B_\cb=0$ at all cusps). Thus, we set
\begin{align*}
  \rho_{\theta_{\chi}}:=\rho_\chi=\begin{cases}1/4 & \text{$\chi$ totally even},\\
    0 & \text{otherwise,}
  \end{cases}
\end{align*}
as well as $s_\phi=\frac{1 + 2\nu}4$. We warn that for~$\nu=0$ this choice deviates from the choice of~$s_\phi$ made in~\ref{hyp:laplacian}.

We also need to extend the definition of $u_\gamma$ to real arguments. Given $x\in\R\setminus\{\gamma^{-1}\infty\}$ and $\gamma\notin\Gamma_\infty$, we let
\begin{equation}
  \label{eq:theta-multiplier}
  \begin{aligned}
    u_\gamma(x):=\eps_d^{-1}\Big(\frac cd\Big)\chi(\gamma) \sign(c)^{-1/2-\nu}\Big(\frac{x-\gamma^{-1}\infty}{|x-\gamma^{-1}\infty|}\Big)^{1/2+\nu}=\lim_{y\to 0^+}u_\gamma(x+iy).
  \end{aligned}
\end{equation}
Notice that if $c<0$ and $x>\gamma^{-1}\infty$ this differs by a minus sign from what one obtains by simply setting $y=0$, but ensures that the cocycle condition still holds at real arguments.

\medskip

The bounds~\eqref{eq:decay-tail-phi}-\eqref{eq:decay-tail-phi-diff} and~\cite[eq.~(2.11)]{DrappeauNordentoft2022}\footnote{As usual $y^{-1/2-\eps}$ is to be replaced by $y^{-\rho_{\theta_\chi}-\eps}$.} hold also for $\phi=\theta_\chi$, as the same proofs still work in this case. In particular, we can replicate the arguments in Section~\ref{sec:estim-period-fun}\footnote{We took care throughout Section~\ref{sec:estim-period-fun} to choose the sign in the exponent in expressions such as $(-1)^{\pm k}$ in such a way that the claimed estimates hold true for~$k\in\frac12 \Z$ with the choice~$(-1)^k = \e^{\pi k i}$. Also,~$\chi(\gamma)$ in Section~\ref{sec:estim-period-fun} should be replaced by~${\tilde \chi}(\gamma) = \eps_d^{-1}\big(\frac c d\big) \chi(\gamma)$.} obtaining the analogues of Lemmas~\ref{lem:properties-cE},~\ref{lem:hE-at-half},~\ref{lem:Ex-Eab},~\ref{lem:Eab-Ex} and~\ref{lem:hE-gamman-holder} for $\theta_\chi$ with the only changes that
\begin{itemize}
  \item $\mathcal D_{\theta_\chi}:=\C\setminus\{1/4,3/4\}$ for $\chi$ totally even and otherwise $\mathcal D_{\theta_\chi}=\C$.
  \item The last (harmless) error term in~\eqref{eq:bound-cR} is to be multiplied by $|cj(\upsilon,\alpha)|^{\rho_\phi}$.
\end{itemize}
As a consequence, Lemma~\ref{lem:hE-at-half} and Propositions~\ref{prop:hE-upperbound-alls} and~\ref{prop:hE-gamman-estim} hold with the caveat that $\hE_\gamma$ has a possible pole at $s=\frac12-s_\phi=\frac14$ if $\chi$ is totally even (in particular, in this case we need to add the condition $s-1/4\gg1$ in the bounds~\eqref{eq:bound-hEgamma} and~\eqref{eq:hE-rough})
and that the asymptotics~\eqref{eq:estim-hEgamma-half} and~\eqref{eq:hE-half} need to be slightly adapted in a straightforward way. In any case the bounds~\eqref{eq:bound-hEgamma} and~\eqref{eq:hE-rough} suffice for our purpose.
The proof is again essentially identical with very minor changes. We signal in particular that
\begin{itemize}
  \item $\mathcal M_\beta(\phi,s)$ can still be absorbed in the bounds~\eqref{eq:bound-hEgamma} and~\eqref{eq:hE-rough}, since $B_{\cb}=0$ at all cusps;
  \item in the notation of Proposition~\ref{prop:hE-gamman-estim}, we have
  \begin{align*}
    [\phi^{\sigma}]^{\sigma^{-1}\gamma}(z)&=\Big(\lim_{t\to0^+}j_{\sigma^{-1}\gamma}(x+it)^{-k}\cdot \lim_{t\to0^+}j_{\sigma}(\sigma^{-1}\gamma (x+it))^{-k}\Big)u_\gamma(x) \phi(z)\\
    &=\sign(c_{\sigma^{-1}\gamma})^{k}\sign(x-\omega)^{-k}\sign(c_{\sigma})^{k}\sign(\sigma^{-1}\gamma x- \sigma^{-1}\infty)^{-k}u_\gamma(x) \phi(z)
  \end{align*}
and with this value the term~$\cE_x(\phi,s)$ in~\eqref{eq:expr-hE-final} arises again with the factor~$u_\gamma(x)$.
\end{itemize}

Finally,
the identity~\eqref{EtoL} is replaced by the classical Mellin formula
\begin{align}
  \cE_x(\theta_\chi, s)=(4\pi)^{k/2} (2\pi)^{-s+(1-2\nu)/4}\Gamma(s-(1-2\nu)/4)L(\theta_\chi,s,x),\qquad \Re(s)>3/4
\end{align}
where $L(\theta_\chi,s,x)=2(4\pi)^{-k/2} L_\chi(2s-1/2,x)$, with $L_\chi$ defined in~\eqref{lcdf}.

The rest of the argument then follows essentially identically, with items (1), (2) and (4) of Lemma~\ref{lem:estimates-vphig} holding unchanged, whereas~\eqref{eq:vphi-eisenstein-parab-estim} could either be slightly adapted or more simply replaced by the bound
$$ \vphi_h(x) \ll n^{1/4} +\abs{x-1}^{-1/4} $$
which ensures that $L(\theta_\chi,1/2,x)=2(4\pi)^{-k/2}L_\chi(1/2,x)$ has moments of any order less than $4$.

In order to obtain Theorem~\ref{thethe}, it is then sufficient to verify that also the multiplier system corresponding to $u_\gamma$ is irreducible with respect to the induced map $\mathcal F_K$, and to show that $\vmu=0$ and $\vSigma = \frac{\varphi (N)}{2N} \Id$.

Regarding the irreducibility, we use the smaller group $\Gamma(4N^2)$. % to ensure~$R$ has at least $8$ sides if $N=1$.
We then have $\Theta=\{\pm 1, \pm i\}$ and the arguments of Section~\ref{sec:top-irred} are easily adapted (we warn that~$\gamma \mapsto (\frac cd)$ is no longer multiplicative in this situation).

As for the mean and variance, one option would be to replicate the arguments of Section~\ref{sec:comp-moments}. However, in this case, we can also apply the approach of~\cite[p.~1416]{BettinDrappeau2022} and~\cite[Section~7.2]{DrappeauNordentoft2022}, using standard analytic methods (the approximate functional equation for the Hurwitz zeta function), which gives $\E_Q[L_\chi(1/2,x)]=o(\sqrt{\log Q})$, $\E_Q[L_\chi(1/2,x)^2]=o(\log Q)$ and $\E_Q[|L_\chi(1/2,x)|^2]\sim \frac{\varphi(N)}{N}\log Q$. In conjunction with the convergence of the third moment\footnote{In~\cite{BettinDrappeau2022} the fourth moment (9.6) is used. This is not available in this case when $\chi$ is totally even, but the third absolute moment can easily be bounded, and this suffices here.}, this yields the claim.

\medskip
\begin{remark}\label{tfeo}
  If $f:\Z\to\C$ is non-zero, $N$-periodic and odd or even, then Theorem~\ref{thethe} holds also for $L_f(s, x) := \sum_{n\geq 1}f(n) \frac{\e(n^2x)}{n^s}$ with variance $\sigma_f^2 := \frac{1}{2N} \sum_{n \pmod N} |f(n)|^2$. To see this, we express $f$ in terms of primitive characters of conductor dividing $N$ and deduce that for $z\in\H$
  \begin{equation}
    \label{eq:def-thetaf}
    \begin{aligned}
      \theta_{f} (z)&:=y^{(1+2\nu_{f})/4}\sum_{n\in\Z} n^{\nu_f} f(n)\e(n^2z)
      =\sum_{d|N}\sumstar_{\substack{\chi \bmod{d} \\ \chi(-1)=(-1)^{\nu_f}}}\sum_{\ell | N/d}\alpha_{\chi,\ell}\theta_\chi(\ell^2 z),
    \end{aligned}
  \end{equation}
  for some $\alpha_{\chi,\ell}\in\C$, where $\sum^*$ indicates the sum is restricted to primitive characters, and where $\nu_f=0$ if $f$ is even, and $\nu_f=1$ if $f$ is odd.
  Since $d\ell |N$, the old forms $\theta_\chi(\ell^2 \,\cdot)$ are also Maa\ss{} forms of weight $k=\frac12+\nu_{f}$ for $\Gamma_0(4N^2)$. Although the multiplier system does depend on $\chi$, this is bypassed by passing to the smaller subgroup $\Gamma_1(4N^2)$, where the characters become trivial. Consequently, $\theta_{f} (z)$ is a Maa\ss{} form of weight $k$ for $\Gamma_1(4N^2)$. By linearity we obtain the required properties and bounds for the period function associated to $\theta_{f}$. Thus, the central limit theorem follows as above from the moment estimates
  $ \E_Q[L_{f}(1/2,x)]=o(\sqrt{\log Q}),$ $\E_Q[L_{f}(1/2,x)^2]=o(\log Q)$ and $\E_Q[|L_{f}(1/2,x)|^2]\sim 2\sigma_{f}^2\log Q. $
\end{remark}

\begin{remark}\label{rmk:theta-general-f}
  The parity assumption in the previous remark is superficial, and could be removed by defining a variant of the operator $\LL$, keeping track of the actions of the multiplier systems~$u_\gamma^{(k)}(x)$ defined in~\eqref{eq:theta-multiplier} of both weights~$k=1/2$ and~$k=3/2$. While proofreading a preliminary version of this paper, an LLM observed that the two multipliers are actually conjugate of each other, and this leads to a simplification, since with the notation~\eqref{eq:def-thetaf}, the values~$L_{f_1}(1/2, x) + \bar{L_{f_2}(1/2, x)}$ satisfy the period relation~\eqref{eq:def-periodfun-informal} with the single multiplier~$u_\gamma^{(1/2)}(x)$ for every~$N$-periodic functions~$f_i:\Z\to\C$ with~$f_1$ even and~$f_2$ odd.  Then the CLT in Theorem~\ref{thethe} holds for these values.
  Now given~$z\in\C$, a non-zero $N$-periodic function~$f:\Z\to\C$ and writing~$f_+, f_-$ its even and odd parts, decompose
  $$ \Re(z L_f(1/2, x)) = \Re(z L_{f_+}(1/2, x) + \bar{z} \bar{L_{f_-}(1/2, x)}). $$
  The (real-valued) CLT for the right-hand side is a consequence of the above discussion with~$f_1 = z f_+$ and~$f_2 = z f_-$. One then concludes with the Cramér-Wold device~\cite{CramerWold1936}.
  We refer to~\cite[Section~11.5]{BruggemanEtAl2018} and~\cite[Section~3.3]{DrappeauNordentoft2022} for earlier occurrences of using conjugation to alter the weight, and~\cite[Section~5.3]{Marklof2007} for an alternative way of handling even and odd parts by altering the group instead.
\end{remark}

\bibliographystyle{alpha}

\bibliography{LLcfrac-ext}

\end{document}